\documentclass[10pt]{amsart}

\IfFileExists{srcltx.sty}{\usepackage{srcltx}}

\usepackage[latin1]{inputenc}
\usepackage{xspace,amssymb,amsfonts,euscript}
\usepackage{amsthm,amsmath}
\usepackage{palatino}
\usepackage{euscript}
\input xy \xyoption {all}
\usepackage{mathrsfs}

\numberwithin{equation}{section}

\RequirePackage{color}
\definecolor{myred}{rgb}{0.75,0,0}
\definecolor{mygreen}{rgb}{0,0.5,0}
\definecolor{myblue}{rgb}{0,0,0.65}

\RequirePackage{ifpdf}
\ifpdf
  \IfFileExists{pdfsync.sty}{\RequirePackage{pdfsync}}{}
  \RequirePackage[pdftex,
   colorlinks = true,
   urlcolor = myblue, 
   citecolor = mygreen, 
   linkcolor = myred, 
   bookmarksopen=true]{hyperref}
\else
  \RequirePackage[hypertex]{hyperref}
\fi

\RequirePackage{ae, aecompl, aeguill} 

    \def\AM{{\mathbb{A}}}
  \def\bg{{\mathfrak b}}  \def\BM{{\mathbb{B}}}
    \def\CM{{\mathbb{C}}}
    \def\DM{{\mathbb{D}}}
    \def\EM{{\mathbb{E}}}
    \def\FM{{\mathbb{F}}}
  \def\gg{{\mathfrak g}}  \def\GM{{\mathbb{G}}}
    \def\HM{{\mathbb{H}}}
\def\IG{{\mathfrak I}}    
    
    \def\KM{{\mathbb{K}}}
    \def\LM{{\mathbb{L}}}
    \def\MM{{\mathbb{M}}}
  \def\ng{{\mathfrak n}}  
    
  \def\pg{{\mathfrak p}}  \def\PM{{\mathbb{P}}}
    \def\QM{{\mathbb{Q}}}
\def\RG{{\mathfrak R}}    
  \def\sg{{\mathfrak s}}  \def\SM{{\mathbb{S}}}
  \def\tg{{\mathfrak t}}  \def\TM{{\mathbb{T}}}

    \def\ZM{{\mathbb{Z}}}

\def\AB{{\mathbf A}}    
\def\BB{{\mathbf B}}    
    \def\CC{{\mathcal{C}}}
    \def\DC{{\mathcal{D}}}
    \def\EC{{\mathcal{E}}}
    \def\FC{{\mathcal{F}}}
\def\GB{{\mathbf G}}    \def\GC{{\mathcal{G}}}
\def\HB{{\mathbf H}}    \def\HC{{\mathcal{H}}}
\def\IB{{\mathbf I}}    \def\IC{{\mathcal{I}}}
\def\JB{{\mathbf J}}

    \def\MC{{\mathcal{M}}}
    \def\NC{{\mathcal{N}}}
    \def\OC{{\mathcal{O}}}
\def\PB{{\mathbf P}}    \def\PC{{\mathcal{P}}}

\def\SB{{\mathbf S}}    \def\SC{{\mathcal{S}}}
\def\TB{{\mathbf T}}    \def\TC{{\mathcal{T}}}
\def\UB{{\mathbf U}}    \def\UC{{\mathcal{U}}}

\def\XB{{\mathbf X}}    \def\XC{{\mathcal{X}}}

\def\ES{{\EuScript E}}

\def\IS{{\EuScript I}}
\def\JS{{\EuScript J}}
\def\KS{{\EuScript K}}

\def\O{\Omega}

\newcommand{\nc}{\newcommand} \newcommand{\renc}{\renewcommand}

\newcommand{\rdots}{\mathinner{ \mkern1mu\raise1pt\hbox{.}
    \mkern2mu\raise4pt\hbox{.}
    \mkern2mu\raise7pt\vbox{\kern7pt\hbox{.}}\mkern1mu}}

\def\reg{{\mathrm{reg}}}

\def\rs{{\mathrm{rs}}}

\DeclareMathOperator{\Coh}{Coh}

\DeclareMathOperator{\rk}{rk}
\DeclareMathOperator{\Lie}{Lie}
\DeclareMathOperator{\Modgr}{Mod^{gr}}
\DeclareMathOperator{\Dist}{Dist}

\def\to{\rightarrow}

\def\longto{\longrightarrow}

\def\onto{\twoheadrightarrow}
\nc{\triright}{\stackrel{[1]}{\to}}
\nc{\longtriright}{\stackrel{[1]}{\longto}}

\nc{\Br}{\mathcal{B}}
\nc{\HotRR}{{}_R\mathcal{K}_R}
\nc{\HotR}{\mathcal{K}_R}
\nc{\excise}[1]{}
\nc{\defect}{\text{df}}
\nc{\h}[1]{\underline{H}_{#1}}

\nc{\Ga}{\mathbb{G}_a} 
\nc{\Gm}{\mathbb{G}_{\mathrm{m}}} 
\nc{\GmR}{\mathbb{G}_{\mathrm{m}, \RG}}
\nc{\GmF}{\mathbb{G}_{\mathrm{m}, \FM}}
\nc{\GmE}{\mathbb{G}_{\mathrm{m}, \EM}}

\nc{\IH}{{\mathrm{IH}}}

\nc{\ic}{\mathbf{IC}}

\nc{\gl}{{\mathfrak{gl}}}
\renc{\sl}{{\mathfrak{sl}}}
\renc{\sp}{{\mathfrak{sp}}}

\nc{\HBM}{H^{BM}}

\DeclareMathOperator{\For}{For} 
 \DeclareMathOperator{\Hom}{Hom}
 \DeclareMathOperator{\ch}{ch}
\DeclareMathOperator{\End}{End} 
 
\DeclareMathOperator{\Rep}{Rep}

\newtheorem{thm}{Theorem}[section]
\newtheorem{lem}[thm]{Lemma}

\newtheorem{prop}[thm]{Proposition}
\newtheorem{cor}[thm]{Corollary}
  
\newtheorem{conj}[thm]{Conjecture}

\theoremstyle{definition}
\newtheorem{defi}[thm]{Definition}

\theoremstyle{remark}
\newtheorem{remark}[thm]{Remark}

\DeclareMathOperator{\Ext}{Ext}

\DeclareMathOperator{\Tor}{Tor}
\DeclareMathOperator{\Spec}{Spec}

\DeclareMathOperator{\grk}{grk}

\DeclareMathOperator{\PParity}{PParity}
\DeclareMathOperator{\Parity}{Parity}
\DeclareMathOperator{\Perv}{Perv}
\DeclareMathOperator{\Tilt}{Tilt}

\newcommand{\into}{\hookrightarrow}

\def\pt{{\mathrm{pt}}}

\def\Gr{{\EuScript Gr}}

\def\Fl{{\EuScript Fl}}
\def\tNC{\widetilde{\NC}}
\def\Flag{\mathscr{B}}

\newcommand{\tSC}{\widetilde{\SC}}
\newcommand{\tUp}{\widetilde{\Upsilon}}

\def\GD{\check{G}}
\def\BD{\check{B}}
\def\TD{\check{T}}

\def\ID{\check{I}}
\def\JD{\check{J}}

\def\XBD{\check{\XB}}

\def\Mod{\mathrm{Mod}}

\newcommand{\tgg}{\widetilde{\gg}}

\makeatletter
\def\lotimes{\@ifnextchar_{\@lotimessub}{\@lotimesnosub}}
\def\@lotimessub_#1{\mathchoice{\mathbin{\mathop{\otimes}^L}_{#1}}%
  {\otimes^L_{#1}}{\otimes^L_{#1}}{\otimes^L_{#1}}}
\def\@lotimesnosub{\mathbin{\mathop{\otimes}^L}}
\makeatother

\newcommand{\Waff}{W_{\mathrm{aff}}}
\newcommand{\Baff}{\BM_{\mathrm{aff}}}
\newcommand{\simto}{\xrightarrow{\sim}}
\newcommand{\Cox}{\mathrm{Cox}}
\newcommand{\Haff}{\HM_{\mathrm{aff}}}
\newcommand{\HW}{\HM_{W}}
\newcommand{\Msph}{\MM_{\mathrm{sph}}}

\newcommand{\BStop}{\mathsf{BS}^{\mathrm{top}}}
\newcommand{\BSalg}{\mathsf{BS}^{\mathrm{alg}}}
\newcommand{\BSgeom}{\mathsf{BS}^{\mathrm{coh}}}

\newcommand{\Db}{D^{\mathrm{b}}}

\newcommand{\Dmix}{D^{\mathrm{mix}}}
\newcommand{\Kb}{K^{\mathrm{b}}}

\newcommand{\tIB}{\widetilde{\IB}}
\newcommand{\tJB}{\widetilde{\JB}}
\newcommand{\tIG}{\widetilde{\IG}}

\newcommand{\lav}{{\check \lambda}}
\newcommand{\Inv}{\mathsf{Inv}}

\newcommand{\vv}{\mathsf{v}}
\newcommand{\mm}{\mathbf{m}}
\newcommand{\us}{\underline{s}}
\newcommand{\ut}{\underline{t}}

\newcommand{\SSC}{\mathscr{S}}
\newcommand{\mix}{\mathrm{mix}}
\newcommand{\dmix}{\Delta^\mix}
\newcommand{\nmix}{\nabla^\mix}
\newcommand{\tTC}{\widetilde{\TC}}
\newcommand{\gr}{\mathrm{gr}}
\newcommand{\tC}{\widetilde{C}}

\newcommand{\simple}{\Phi^{\mathrm{s}}}

\newcommand{\co}{\mathsf{H}}
\newcommand{\tco}{\widetilde{\mathsf{H}}}

\newcommand{\Lder}{\mathsf{L}}
\newcommand{\Rder}{\mathsf{R}}
\newcommand{\TF}{\mathsf{T}}

\newcommand{\DNC}{\mathrm{DNC}}

\begin{document}

\begin{abstract}
Let $\GB$ be a connected reductive group over an algebraically closed field $\FM$ of good characteristic, satisfying some mild conditions. In this paper we relate tilting objects in the heart of Bezrukavnikov's exotic t-structure on the derived category of equivariant coherent sheaves on the Springer resolution of $\GB$, and Iwahori-constructible $\FM$-parity sheaves on the affine Grassmannian of the Langlands dual group. As applications we deduce in particular the missing piece for the proof of the Mirkovi{\'c}--Vilonen conjecture in full generality (i.e.~for good characteristic), a modular version of an equivalence of categories due to Arkhipov--Bezrukavnikov--Ginzburg, and an extension of this equivalence.
\end{abstract}

\title[Exotic tilting sheaves and parity sheaves]{Exotic tilting sheaves, \\
parity sheaves on affine Grassmannians, \\
and the Mirkovi{\'c}--Vilonen conjecture}

\author{Carl Mautner}
 \address{Department of Mathematics, University of California, Riverside, CA 92521, USA}
\email{mautner@math.ucr.edu}

\thanks{The material in this article is based upon work supported by the National Science Foundation under Grant No. 0932078 000 while the first author was in residence at the Mathematical Sciences Research Institute in Berkeley, California, during the Fall 2014 semester.  C.M. thanks MSRI and the Max Planck Institut f\"ur Mathematik in Bonn for excellent working conditions.}
  
\author{Simon Riche}
\address{Universit{\'e} Blaise Pascal - Clermont-Ferrand II, Laboratoire de Math{\'e}matiques, CNRS, UMR 6620, Campus universitaire des C{\'e}zeaux, F-63177 Aubi{\`e}re Cedex, France}
\email{simon.riche@math.univ-bpclermont.fr}

\thanks{S.R. was supported by ANR Grants No.~ANR-2010-BLAN-110-02 and ANR-13-BS01-0001-01.}

\maketitle

\section{Introduction}

\subsection{Summary}
\label{ss:intro}

Let $\FM$ be an algebraically closed field, and let $\GB$ be a connected reductive group over $\FM$ which is a product of simply-connected quasi-simple groups and of general linear groups. We assume that the characteristic $p$ of $\FM$ is very good for each quasi-simple factor of $\GB$. Let also $\GD$ be the complex Langlands dual group. The main result of this paper is an equivalence of categories relating tilting objects in the heart of Bezrukavnikov's \emph{exotic t-structure} on the derived category of $\GB \times \Gm$-equivariant coherent sheaves on the Springer resolution $\tNC$ of $\GB$, and Iwahori-constructible parity sheaves on the affine Grassmannian $\Gr$ of $\GD$, with coefficients in $\FM$ (in the sense of~\cite{jmw}).

We provide several applications of this result; in particular
\begin{enumerate}
\item
\label{it:intro-perverse}
a proof that spherical parity sheaves on $\Gr$, with coefficients in a field of good characteristic, are perverse, which provides the last missing step in the proof of the Mirkovi{\'c}--Vilonen conjecture~\cite{mv} on stalks of standard spherical perverse sheaves on $\Gr$ in the expected generality;
\item
\label{it:intro-equiv}
a construction of an equivalence of categories relating $\Db \Coh^{\GB \times \Gm}(\tNC)$ to the ``modular mixed derived category'' of Iwahori-constructible sheaves on $\Gr$ (in the sense of~\cite{modrap2}), which is a modular generalization of an equivalence due to Arkhipov--Bezrukavnikov--Ginzburg~\cite{abg}; 
\item
an ``extension'' of~\eqref{it:intro-equiv} to an equivalence relating 
$\Db \Coh^{\GB \times \Gm}(\tgg)$ (where $\tgg$ is the Grothendieck resolution of $\GB$) to the modular mixed derived category of Iwahori-\emph{equivariant} sheaves on $\Gr$.
\end{enumerate}

A weaker version of~\eqref{it:intro-perverse} was obtained earlier by Juteau--Mautner--Williamson \cite{jmw2} using a case-by-case argument. The application to the Mirkovi{\'c}--Vilonen conjecture is due to Achar--Rider~\cite{arider}. In the case $p=0$, the equivalence in~\eqref{it:intro-equiv} plays an important role in the representation theory of Lusztig's quantum groups at a root of unity, see~\cite{abg, bezru-tilting, bl}. We expect our equivalence to play a comparable role in the modular representation theory of connected reductive groups.\footnote{One year after this paper was written, this expectation was indeed confirmed in the article~\cite{prinblock} by P.~Achar and the second author.} A similar result has been obtained independently by Achar--Rider~\cite{arider2}, under the assumption that spherical parity sheaves are perverse. Our methods are different from theirs; see~\S\ref{ss:comparison} below for a detailed comparison.

\subsection{Main result}

To state our results more precisely, let us choose a Borel subgroup $\BB \subset \GB$, and a maximal torus $\TB \subset \BB$. Let $\bg$ be the Lie algebra of $\BB$. Then the Springer resolution $\tNC$ is defined as
\[
\tNC:=\{(\xi, g\BB) \in \gg^* \times \GB/\BB \mid \xi_{|g \cdot \bg}=0\}.
\]
This variety is endowed with a natural action of $\GB \times \Gm$, defined by
\[
(g,x) \cdot (\xi, h\BB) := (x^{-2} g \cdot \xi, gh\BB),
\]
so that we can consider the derived category $\Db \Coh^{\GB \times \Gm}(\tNC)$ of $\GB \times \Gm$-equiva\-riant coherent sheaves on $\tNC$. This category possesses a remarkable t-structure, called the \emph{exotic t-structure}, defined by Bezrukavnikov~\cite{bezru-tilting} in the case $p=0$, and studied in our generality in~\cite{mr}. The heart $\ES^{\GB \times \Gm}(\tNC)$ of this t-structure has a natural structure of a graded highest weight category, with weights the lattice $\XB = X^*(\TB)$ of characters of $\TB$, and ``normalized'' standard, resp.~costandard, objects denoted $\Delta^\lambda_{\tNC}$, resp.~$\nabla^\lambda_{\tNC}$. In particular, we will be interested in the category $\Tilt(\ES^{\GB \times \Gm}(\tNC))$ of \emph{tilting objects} in $\ES^{\GB \times \Gm}(\tNC)$ (i.e.~those objects which possess both a standard filtration and a costandard filtration). This category is Krull--Schmidt, and its indecomposable objects are parametrized in a natural way by $\XB \times \ZM$. For $\lambda \in \XB$, we denote by $\TC^\lambda$ the indecomposable object attached to $(\lambda,0)$. Then for any $n \in \ZM$ the object associated with $(\lambda,n)$ is $\TC^\lambda \langle n \rangle$, where $\langle n \rangle$ is the $n$-th power of the functor $\langle 1 \rangle$ of tensoring with the tautological $1$-dimensional $\Gm$-module (see~\S\ref{ss:notation-KW-tilting} for details). For any $\TC$ in $\Tilt(\ES^{\GB \times \Gm}(\tNC))$, $\mu \in \XB$ and $m \in \ZM$, we denote by $(\TC : \Delta^\lambda_{\tNC} \langle m \rangle)$, resp.~$(\TC : \nabla^\lambda_{\tNC} \langle m \rangle)$, the multiplicity of the standard object $\Delta^\lambda_{\tNC} \langle m \rangle$, resp.~of the costandard object $\nabla^\lambda_{\tNC} \langle m \rangle$, in a standard (resp. costandard) filtration of $\TC$.

Let now $\GD$ be a complex connected group, with a maximal torus $\TD \subset \GD$, and assume that $(\GD, \TD)$ is Langlands dual to $(\GB, \TB)$, in the sense that the root datum of $(\GD, \TD)$ is dual to that of $(\GB, \TB)$. In particular, we have an identification $\XB=X_*(\TD)$. We let $\BD \subset \GD$ be the Borel subgroup containing $\TD$ whose roots are the coroots of $\BB$ (which we will consider as the \emph{negative} coroots). Let $\mathscr{O}:=\CM[ \hspace{-1pt} [z] \hspace{-1pt} ]$ and $\mathscr{K}:=\CM( \hspace{-1pt} (z) \hspace{-1pt} )$. Then the affine Grassmannian $\Gr$ of $\GD$ is defined as
\[
\Gr := \GD(\mathscr{K}) / \GD (\mathscr{O}),
\]
with its natural ind-variety structure. We denote by $\ID$ the Iwahori subgroup of $\GD(\mathscr{O})$ determined by $\BD$, i.e.~the inverse image of $\BD$ under the morphism $\GD(\mathscr{O}) \to \GD$ defined by the evaluation at $z=0$. Then $\ID$ acts naturally on $\Gr$ via left multiplication on $\GD(\mathscr{K})$, and the orbits of this action are parametrized in a natural way by $\XB$; we denote by $\Gr_\lambda$ the orbit associated with $\lambda$ and by $i_\lambda : \Gr_\lambda \hookrightarrow \Gr$ the inclusion. We let
\[
\Parity_{(\ID)}(\Gr,\FM)
\]
be the category of parity sheaves on $\Gr$, with coefficients in $\FM$, with respect to the stratification by $\ID$-orbits (in the sense of~\cite{jmw}). This category is defined as an additive subcategory of the derived category $\Db_{(\ID)}(\Gr, \FM)$ of $\ID$-constructible $\FM$-sheaves on $\Gr$. It is Krull--Schmidt, and its indecomposable objects are parametrized in a natural way by $\XB \times \ZM$. We denote by $\EC_\lambda$ the indecomposable object attached to $(\lambda,0)$; then for any $n \in \ZM$ the object associated with $(\lambda,n)$ is $\EC_\lambda[n]$.

The main result of this paper (whose proof is given in~\S\ref{ss:proof-main}) is the following.

\begin{thm}
\label{thm:main}
There exists an equivalence of additive categories
\begin{equation}
\label{eqn:equiv-main}
\Theta \colon
\Parity_{(\ID)}(\Gr,\FM) \simto \Tilt(\ES^{\GB \times \Gm}(\tNC))
\end{equation}
which satisfies the following properties:
\begin{enumerate}
\item
\label{it:main-thm-shifts}
$\Theta \circ [1] \cong \langle -1 \rangle \circ \Theta$;
\item
\label{it:main-thm-ch}
for all $m \in \ZM$, $\lambda \in \XB$ and $\EC \in \Parity_{(\ID)}(\Gr,\FM)$, we have
\begin{align*}
(\Theta(\EC) : \Delta^\lambda_{\tNC} \langle m \rangle) &= \dim_\FM \bigl( \HM^{m-\dim(\Gr_{-\lambda})}(\Gr_{-\lambda}, i_{-\lambda}^* \EC) \bigr); \\
(\Theta(\EC) : \nabla^\lambda_{\tNC} \langle m \rangle) &= \dim_\FM \bigl( \HM^{m-\dim(\Gr_{-\lambda})}(\Gr_{-\lambda}, i_{-\lambda}^! \EC) \bigr);
\end{align*}
\item
\label{it:main-thm-indec}
$\Theta(\EC_\lambda) \cong \TC^{-\lambda}$
for any $\lambda \in \XB$.
\end{enumerate}
\end{thm}

\subsection{Outline of the proof}
\label{ss:outline}

Our strategy of proof of Theorem~\ref{thm:main} is based on the description of both sides in~\eqref{eqn:equiv-main} in terms of an appropriate category of ``Soergel bimodules.''\footnote{Our ``Soergel bimodules'' are in fact not bimodules over any ring, but rather modules over a certain algebra built out of two copies of a polynomial algebra. We use this terminology since these objects play the role which is usually played by actual Soergel bimodules.} This idea is very classical, see e.g.~\cite{soergel-kategorie, by, dodd} for examples in characteristic zero, and~\cite{soergel, modrap1} for examples in positive characteristic. 

On the ``constructible side'' (i.e.~the left-hand side in~\eqref{eqn:equiv-main}), this description is obtained using the total cohomology functor. The arguments in this section follow well-established techniques; see~\S\ref{ss:key-lemma} below for a discussion of the only new idea that is needed.

On the ``coherent side'' (i.e.~the right-hand side in~\eqref{eqn:equiv-main}), we adapt a construction due to Dodd~\cite{dodd} in characteristic zero, which uses a ``Kostant--Whittaker reduction'' functor. This construction uses modular (and integral) versions of some classical results of Kostant, which are treated in the companion paper~\cite{riche3}. Our constructions are slightly different from Dodd's, however, in that we do not use a deformation to (asymptotic) $\mathcal{D}$-modules, but only to coherent sheaves on the Grothendieck resolution 
\[
\tgg:=\{(\xi, g\BB) \in \gg^* \times \GB/\BB \mid \xi_{|g \cdot \ng}=0\},
\]
where $\ng$ is the Lie algebra of the unipotent radical of $\BB$. (On the constructible side, this deformation amounts to replacing the category $\Parity_{(\ID)}(\Gr, \FM)$ by the category $\Parity_{\ID}(\Gr, \FM)$ of $\ID$-equivariant $\FM$-parity sheaves on $\Gr$.)

Our proof of fully-faithfulness of the Kostant--Whittaker reduction functor is also different from the proof in~\cite{dodd}. One of the crucial ideas in our proof, which we learnt in papers of Soergel~\cite{soergel-kategorie, soergel} and was used also in~\cite{modrap1}, is to compute $\Hom$-spaces between Soergel bimodules from the analogue for parity sheaves.  We use this computation to prove the ``coherent side''; see the proof of Theorem~\ref{thm:main-geom} for more details.

We also use some ideas from categorification, related to the fact that both categories appearing in~\eqref{eqn:equiv-main} provide categorifications of the spherical module $\Msph$ of the affine Hecke algebra $\Haff$ attached to $\GB$ (in the sense that their split Grothen\-dieck groups are equipped with natural actions of $\Haff$, and are naturally isomorphic to $\Msph$). On the left-hand side this uses the realization of $\Haff$ in terms of constructible sheaves on the affine flag variety of $\GD$, and on the right-hand side this uses the Kazhdan--Lusztig--Ginzburg description of $\Haff$ in terms of equivariant coherent sheaves on the Steinberg variety of $\GB$, a ``categorical'' counterpart of which is provided by the ``geometric braid group action'' studied in~\cite{riche, br}.

Note that, surprisingly, our proof does \emph{not} use the geometric Satake equivalence from~\cite{mv}.

\subsection{Equivalences of triangulated categories}

One of our motivations for the study of the equivalence in Theorem~\ref{thm:main} was the desire to obtain a ``modular version'' of an equivalence of categories due to Arkhipov--Bezrukavnikov--Ginzburg~\cite{abg} in the case $p=0$. Our version involves the ``modular mixed derived category''
\[
\Dmix_{(\ID)}(\Gr, \FM) := \Kb \Parity_{(\ID)}(\Gr, \FM),
\]
introduced and studied (in a more general setting) in~\cite{modrap2}. In particular, it can be endowed with a ``perverse t-structure'' and a ``Tate twist'' autoequivalence $\langle 1 \rangle$, and possesses ``standard objects'' $\dmix_{\lambda}$ ($\lambda \in \XB$) and ``costandard objects'' $\nmix_\lambda$ ($\lambda \in \XB$) which have the same properties as ordinary standard and costandard perverse sheaves. For $\lambda \in \XB$, we denote by $\EC^\mix_\lambda$ the object $\EC_\lambda$, considered as an object of $\Dmix_{(\ID)}(\Gr, \FM)$. In case $\FM=\overline{\QM}_\ell$, the category $\Dmix_{(\ID)}(\Gr, \FM)$ is equivalent to the bounded derived category of the category $\tilde{\PC}$ of~\cite[\S 4.4]{bgs}, see~\cite[Remark~2.2]{modrap2}.

As an immediate application of Theorem~\ref{thm:main} we obtain the following theorem, whose proof is given in~\S\ref{ss:proof-equivalence-tNC}, and which provides a ``modular analogue'' of the equivalence of~\cite[Theorem~9.4.1]{abg}.

\begin{thm}
\label{thm:equivalence-tNC}
There exists an equivalence of triangulated categories
\[
\Phi \colon \Dmix_{(\ID)}(\Gr,\FM) \simto \Db \Coh^{\GB \times \Gm}(\tNC)
\]
which satisfies
\[
\Phi \circ \langle 1 \rangle \cong \langle 1 \rangle [1] \circ \Phi, \quad \Phi(\dmix_{\lambda}) \cong \Delta_{\tNC}^{-\lambda}, \quad \Phi(\nmix_{\lambda}) \cong \nabla_{\tNC}^{-\lambda}, \quad \Phi(\EC^\mix_\lambda) \cong \TC^{-\lambda}
\]
for all $\lambda \in \XB$.
\end{thm}

\begin{remark}
One should think of $\Phi$ as some kind of ``Ringel--Koszul duality,'' i.e.~a composition of a Koszul duality (as in~\cite[Theorem~2.12.6]{bgs}, which sends simple objects to projective objects) and a Ringel duality (which sends projective objects to tilting objects). The relevance of such equivalences in Lie-theoretic contexts was pointed out in~\cite{bg}; see also~\cite{by}. The idea that, in a modular context, simple perverse sheaves should be replaced by parity sheaves in ``Koszul-type'' statements is implicit in~\cite{soergel}, and was developed more explicitly in~\cite{rsw}.
\end{remark}

As explained in~\S\ref{ss:outline}, our proof of Theorem~\ref{thm:main} is based on the consideration of a ``deformation'' of the picture, replacing $\ID$-constructible sheaves by $\ID$-equivariant sheaves, and equivariant coherent sheaves on $\tNC$ by equivariant coherent sheaves on $\tgg$. Using these considerations, we define objects $\Delta^\lambda_{\tgg}$ and $\nabla^\lambda_{\tgg}$ in $\Db \Coh^{\GB \times \Gm}(\tgg)$ which satisfy
\[
\Lder i^*(\Delta^\lambda_{\tgg}) \cong \Delta^\lambda_{\tNC}, \qquad \Lder i^*(\nabla^\lambda_{\tgg}) \cong \nabla^\lambda_{\tNC}
\]
for all $\lambda \in \XB$, where $i \colon \tNC \hookrightarrow \tgg$ is the inclusion (see~\S\ref{ss:standard-KW}). We also define an additive and Karoubian subcategory ${\EuScript Tilt}$  of $\Db \Coh^{\GB \times \Gm}(\tgg)$, stable under the shift $\langle n \rangle$ for any $n \in \ZM$, and indecomposable objects $\tTC^\lambda$ in ${\EuScript Tilt}$ satisfying $\Lder i^*(\tTC^\lambda) \cong \TC^\lambda$ and such that ${\EuScript Tilt}$ is Krull--Schmidt with indecomposable objects $\tTC^\lambda \langle n \rangle$ for $\lambda \in \XB$ and $n \in \ZM$, see~\S\ref{ss:proof-equivalence-tgg}. On the other hand, we consider the equivariant modular mixed derived category
\[
\Dmix_{\ID}(\Gr, \FM) := \Kb \Parity_{\ID}(\Gr, \FM),
\]
and denote by $\dmix_{\ID, \lambda}$, resp.~$\nmix_{\ID, \lambda}$, the standard, resp.~costandard, object associated with $\lambda \in \XB$. For any $\lambda \in \XB$, the parity complex $\EC_\lambda$ can be naturally ``lifted'' to the category $\Parity_{\ID}(\Gr, \FM)$; we again denote by $\EC_\lambda^\mix$ this object viewed as an object in $\Dmix_{\ID}(\Gr, \FM)$.

The following ``deformation'' of Theorem~\ref{thm:equivalence-tNC} is proved in~\S\ref{ss:proof-equivalence-tgg}. (In case $p=0$, a similar result can be deduced from the main result of~\cite{dodd}, though this equivalence is not explicitly stated there.)

\begin{thm}
\label{thm:equivalence-tgg}
There exists an equivalence of triangulated categories
\[
\Psi \colon \Dmix_{\ID}(\Gr,\FM) \simto \Db \Coh^{\GB \times \Gm}(\tgg)
\]
which satisfies
\[
\Psi \circ \langle 1 \rangle \cong \langle 1 \rangle [1] \circ \Psi, \quad \Psi(\dmix_{\ID, \lambda}) \cong \Delta_{\tgg}^{-\lambda}, \quad \Psi(\nmix_{\ID, \lambda}) \cong \nabla_{\tgg}^{-\lambda}, \quad \Psi(\EC^\mix_{\lambda}) \cong \tTC^{-\lambda}
\]
for all $\lambda \in \XB$.
\end{thm}

In~\S\ref{ss:compatibilities} we also prove that the equivalences $\Psi$ and $\Phi$ are compatible in the natural way.

\subsection{Parity, tilting and the Mirkovi{\'c}--Vilonen conjecture}
\label{ss:mv-conjecture}

Our other main motivation for the study of Theorem~\ref{thm:main} comes from the Mirkovi{\'c}--Vilonen conjecture~\cite[Conjecture~13.3]{mv}. In this subsection, we let $\GB$ be more generally any split connected reductive group scheme over an arbitrary commutative ring $k$.  As before, we fix a (split) maximal torus $\TB$ and let $\XB$ denote the weight lattice.  We again let $\GD$ and $\TD \subset \GD$ denote the unique (up to isomorphism) complex connected reductive group and maximal torus with Langlands dual root datum, and consider its affine Grassmannian $\Gr:=\GD(\mathscr{K})/\GD(\mathscr{O})$.

The \emph{geometric Satake equivalence}, proven by Mirkovi{\'c}--Vilonen~\cite{mv} for any $k$ which is Noetherian and of finite global dimension, is an equivalence of abelian categories
\[
\mathsf{S}_k \colon \Perv_{(\GD(\mathscr{O}))}(\Gr, k) \simto \Rep(\GB),
\]
where $\Perv_{(\GD(\mathscr{O}))}(\Gr, k)$ is the category of $\GD(\mathscr{O})$-constructible $k$-perverse sheaves on $\Gr$, and $\Rep(\GB)$ is the category of algebraic $\GB$-modules which are of finite type over $k$. (This equivalence is compatible with the natural monoidal structures on these categories.)

After a choice of positive roots, the $\GD(\mathscr{O})$-orbits on $\Gr$ are parametrized in a natural way by the set  of dominant weights $\XB^+ \subset \XB$; we denote by $\Gr^\lambda$ the orbit associated with $\lambda \in \XB^+$, and by $i^\lambda \colon \Gr^\lambda \hookrightarrow \Gr$ the inclusion. For any $\lambda \in \XB^+$, we consider the perverse sheaf
\[
\IC_!(\lambda, k) := {}^p (i^\lambda)_! \underline{k}_{\Gr^\lambda} [\dim(\Gr^\lambda)].
\]
By~\cite[Proposition~13.1]{mv}, this object corresponds, under $\mathsf{S}_k$, to the Weyl $\GB$-module associated with $\lambda$. 
Towards the end of their paper, Mirkovi{\'c}--Vilonen state the following conjecture.

\begin{conj}[Mirkovi{\'c}--Vilonen~\cite{mv}]
\label{conj:mv}
The cohomology modules of the stalks of $\IC_!(\lambda, \ZM)$ are free.
\end{conj}

We understand this conjecture was formulated based on evidence in type $A$ and as part of an attempt to produce a ``modular analogue'' of results of Bezukavnikov and collaborators (in particular, the equivalence of~\cite[Theorem~9.4.1]{abg}), like the one we prove as Theorem~\ref{thm:equivalence-tNC}.

Using~\cite[Proposition~8.1(a)--(b)]{mv}, one can also state this conjecture equivalently as the property that, for any field $\Bbbk$, the dimensions 
\[
\dim_\Bbbk \bigl( \HM^m(\Gr_\mu, i_\mu^* \IC_!(\lambda, \Bbbk)) \bigr)
\]
are independent of $\Bbbk$. 
Note that, by~\cite[Lemma~7.1]{mv}, if $\Bbbk$ has characteristic zero, we have $\IC_!(\lambda, \Bbbk) = \IC \CC(\Gr^\lambda, \underline{\Bbbk})$, and the graded dimensions of stalks of these objects can be computed in terms of affine Kazhdan--Lusztig polynomials~\cite{kl} and vanish in either all odd or all even degrees.

It was realized by Juteau~\cite{juteau} that Conjecture~\ref{conj:mv} does not hold as stated, namely that the cohomology modules of the stalks of $\IC_!(\lambda, \ZM)$ can have $p$-torsion if $p$ is a prime number which is bad for $\GD$.

On the other hand, if $\IC_!(\lambda, \ZM)$ does not have $p$-torsion for any $\lambda \in \XB^+$, then it follows that for any field $\Bbbk$ of characteristic $p$, the tilting objects in $\Perv_{(\GD(\mathscr{O}))}(\Gr, \Bbbk)$ are parity complexes, or equivalently that all the parity sheaves $\EC_\lambda$ for $\lambda \in -\XB^+$ (i.e.~the parity sheaves which are $\GD(\mathscr{O})$-constructible) are perverse. (The equivalence between these properties follows from~\cite[Proposition~3.3]{jmw2}.)

The statement that the parity sheaves $\EC_\lambda$ for $\lambda \in -\XB^+$ with coefficients in a field $\Bbbk$ of characteristic $p$ are perverse, was proved using a case-by-case analysis by Juteau--Mautner--Williamson~\cite[Theorem 1.8]{jmw2} provided that $p$ is bigger than explicit bounds.  The bounds in~\cite{jmw2} are a byproduct of the method of proof and are in general stronger than the condition that $p$ is good (most notably when $\GB$ is quasi-simple of type $C_n$, in which case the bound is $p>n$).

Achar--Rider~\cite{arider} then proved that if all the parity sheaves $\EC_\lambda$ for $\lambda \in -\XB^+$ with coefficients in a field $\Bbbk$ of characteristic $p$ are perverse, then the cohomology modules of the stalks of the perverse sheaves $\IC_!(\lambda, \ZM)$ have no $p$-torsion.

In~\S\ref{ss:proof-parity-perverse} we give a uniform proof of the following.

\begin{cor}
\label{cor:parity-perverse}
If $p$ is good for $\GD$, then for any $\lambda \in -\XB^+$ the parity sheaf $\EC_\lambda$ is perverse.
\end{cor}

By the main result of~\cite{arider}, this implies:

\begin{thm}[Mirkovi{\'c}--Vilonen conjecture]
\label{thm:mv}
If $p$ is good for $\GD$, then
the cohomology modules of the stalks of the perverse sheaves $\IC_!(\lambda, \ZM)$ have no $p$-torsion.
\end{thm}

Our proof of Corollary~\ref{cor:parity-perverse} is based on the explicit description of the tilting objects $\TC^\lambda$ for $\lambda \in \XB^+$ obtained in~\cite{mr}.

\subsection{Other applications}
\label{ss:intro-other}

Now, let us come back to our assumptions on $\GB$ and $p$ from~\S\ref{ss:intro}.
In this subsection we state some other applications of Theorem~\ref{thm:main}. Most of these results are already known (at least when $p=0$), but their proof usually requires the geometric Satake equivalence of~\S\ref{ss:mv-conjecture}. We find it useful to explain how these results can be derived directly using our methods.

We
denote by $\Parity_{(\GD(\mathscr{O}))}(\Gr,\FM)$ the subcategory of $\Parity_{(\ID)}(\Gr,\FM)$ consisting of objects which are $\GD(\mathscr{O})$-constructible, i.e.~which are direct sums of objects $\EC_\lambda[i]$ where $\lambda \in -\XB^+$ and $i \in \ZM$.
We also denote by $\PParity_{(\GD(\mathscr{O}))}(\Gr,\FM)$ the subcategory of $\Parity_{(\GD(\mathscr{O}))}(\Gr,\FM)$ consisting of objects which are perverse sheaves. By Corollary~\ref{cor:parity-perverse}, this category consists of objects which are direct sums of parity sheaves $\EC_\lambda$ where $\lambda \in -\XB^+$. 

To conform to the notation used most of the time in this context, for $\lambda \in \XB^+$ we set $\EC^\lambda:=\EC_{w_0 \lambda}$ (where $w_0$ is the longest element in the Weyl group $W$ of $(\GB,\TB)$), and (as in~\S\ref{ss:mv-conjecture}) $\Gr^\lambda = \sqcup_{\mu \in W\lambda} \Gr_{\mu}$. We also denote by $L_\lambda$ the unique $\TD$-fixed point in $\Gr_\lambda$, and by $\imath_\lambda \colon \{L_\lambda\} \hookrightarrow \Gr$ the inclusion.

We denote by $\Tilt(\GB)$ the additive category of tilting $\GB$-modules, and by $\mathsf{T}(\lambda)$ the indecomposable tilting $\GB$-module with highest weight $\lambda$ (for $\lambda \in \XB^+$). We also denote by $\mathsf{M}(\lambda)$ and $\mathsf{N}(\lambda)$ the standard and costandard $\GB$-modules with highest weight $\lambda$ respectively (see~\cite[\S 4.3]{mr}). We denote by $(\mathsf{T}(\lambda) : \mathsf{M}(\mu))$ and $(\mathsf{T}(\lambda) : \mathsf{N}(\mu))$ the corresponding multiplicities.

From Theorem~\ref{thm:main} one can deduce the following result (see~\S\ref{ss:proof-Satake}).

\begin{prop}
\label{prop:equiv-parity-tilt}
There exists an equivalence of additive categories
\[
\overline{\mathsf{S}}_\FM \colon \PParity_{(\GD(\mathscr{O}))}(\Gr,\FM) \simto \Tilt(\GB)
\]
which satisfies $\overline{\mathsf{S}}_\FM(\EC^{\lambda}) \cong \mathsf{T}(\lambda)$ for all $\lambda \in \XB^+$. Moreover, for $\lambda,\mu \in \XB^+$ we have
\begin{equation}
\label{eqn:formula-stalks-parity}
\sum_{k \in \ZM} \dim \bigl( \HM^{k-\dim(\Gr^{\mu})}(\imath_\mu^* \EC^\lambda) \bigr) \cdot \vv^k = \sum_{\nu \in \XB^+} \bigl( \mathsf{T}(\lambda) : \mathsf{N}(-w_0\nu) \bigr) \cdot \MC^{-w_0\mu}_\nu(\vv^{-2}),
\end{equation}
where $\MC_\eta^\chi(\vv)$ is Lusztig's $q$-analogue~\cite{lusztig}.
\end{prop}

In case $p=0$, $\mathsf{M}(\lambda)$ and $\mathsf{T}(\lambda)$ both coincide with the simple $\GB$-module with highest weight $\lambda$, so that $\bigl( \mathsf{T}(\lambda) : \mathsf{M}(\nu) \bigr)=\delta_{\lambda, \nu}$. On the other hand we have $\EC^\lambda=\IC \CC(\Gr^\lambda, \underline{\FM})$. By~\cite{kl}, the dimensions of the stalks of this perverse sheaf can be expressed in terms of Kazhdan--Lusztig polynomials for the affine Weyl group $\Waff$ of $\GB$. Hence in this case, \eqref{eqn:formula-stalks-parity} provides a geometric proof of the relation between affine Kazhdan--Lusztig polynomials and Lusztig's $q$-analogue, conjectured in~\cite{lusztig} and proved by different methods in~\cite{ka}.

Once it is known that the parity sheaves $\EC^\lambda$ are perverse, as remarked in~\cite[Proposition~3.3]{jmw2}, it follows that $\mathsf{S}_\FM(\EC^\lambda) \cong \mathsf{T}(\lambda)$ (where $\mathsf{S}_\FM$ is the geometric Satake equivalence, see~\S\ref{ss:mv-conjecture}). Hence one can obtain a different construction of an equivalence $\overline{\mathsf{S}}_\FM$ by simply restricting $\mathsf{S}_\FM$.
In this setting, a sketch of a different proof of~\eqref{eqn:formula-stalks-parity} is given in~\cite[Remark~4.2]{jmw2}.

One can also apply our results to describe the (equivariant) cohomology of spherical parity sheaves and their costalks (see \S \ref{ss:proof-cohomology} for the proof).

\begin{prop}
\label{prop:cohomology}
\begin{enumerate}
\item
\label{it:cohom}
For any $\lambda \in \XB^+$, there exist isomorphisms of graded vector spaces, resp.~of graded $\HM^\bullet_{\ID}(\mathrm{pt}; \FM)$-modules,
\[
\HM^\bullet(\Gr, \EC^\lambda) \cong \mathsf{T}(\lambda), \qquad \HM_{\ID}^\bullet(\Gr, \EC^\lambda) \cong \mathsf{T}(\lambda) \otimes \HM^\bullet_{\ID}(\mathrm{pt}; \FM),
\]
where the grading on $\mathsf{T}(\lambda)$ is obtained from the $\Gm$-action through the cocharacter given by the sum of the positive coroots.
\item
\label{it:cohom-stalks}
For any $\lambda,\mu \in \XB^+$, there exist isomorphisms of graded vector spaces, resp.~of graded $\HM^\bullet_{\ID}(\mathrm{pt}; \FM)$-modules,
\begin{align*}
\HM^{\bullet-\dim(\Gr^{\mu})}(\imath_\mu^! \EC^\lambda) &\cong \bigl( \mathsf{T}(\lambda) \otimes \Gamma(\tNC, \OC_{\tNC}(-w_0 \mu)) \bigr)^\GB, \\
\HM_{\ID}^{\bullet-\dim(\Gr^{\mu})}(\imath_\mu^! \EC^\lambda) &\cong \bigl( \mathsf{T}(\lambda) \otimes \Gamma(\tgg, \OC_{\tgg}(-w_0 \mu)) \bigr)^\GB,
\end{align*}
where in both cases $\mathsf{T}(\lambda)$ is in degree $0$, the global sections are equipped with the grading induced by the $\Gm$-actions on $\tNC$ and $\tgg$, and the $\HM^\bullet_{\ID}(\mathrm{pt}; \FM)$-action on $( \mathsf{T}(\lambda) \otimes \Gamma(\tgg, \OC_{\tgg}(-w_0 \mu)))^\GB$ is induced by the natural morphism $\tgg \to \mathrm{Lie}(\TB) = \Spec(\HM^\bullet_{\ID}(\mathrm{pt}; \FM))$.
\end{enumerate}
\end{prop}

The first isomorphism in~\eqref{it:cohom} can be alternatively deduced from the fact that $\mathsf{S}_\FM(\EC^\lambda) \cong \mathsf{T}(\lambda)$, see the comments after Proposition~\ref{prop:equiv-parity-tilt}. Then one can deduce the second isomorphism using~\cite[Lemma~2.2]{yz}. In case $p=0$, the isomorphisms in~\eqref{it:cohom-stalks} are proved in~\cite[Corollary~2.4.5 and Proposition~8.7.1]{gr}.\footnote{The conventions for the choice of positive roots and the normalization of line bundles are different in~\cite{gr}, which explains the formal difference between the two formulas. A more important difference is that in~\cite{gr} we construct an \emph{explicit and canonical} isomorphism, while here we only claim the \emph{existence} of such an isomorphism.}

\subsection{Comparison with~\cite{arider2}}
\label{ss:comparison}

As mentioned already, in~\cite{arider2} Achar and Rider have obtained a different proof of Theorem~\ref{thm:equivalence-tNC}. Their methods are quite different from ours, and closer to the methods used in~\cite{abg}. In fact, while for us most of the work is required on the ``coherent side,'' in their approach the most difficult constructions appear on the ``topological side.'' Moreover, the exotic t-structure does not play any role in the construction of their equivalence.\footnote{The relation between their equivalence and the exotic t-structure is studied in~\cite[Section~8]{arider2}, but only after the equivalence is constructed.} Another important difference is that their arguments rely on the geometric Satake equivalence, while ours do not.

The assumptions in~\cite{arider2} are also different from ours: in fact they assume that the field $\FM$ is such that
any spherical parity sheaf with coefficients in $\FM$ on $\Gr$ is perverse.  Hence Corollary~\ref{cor:parity-perverse} allows to extend the validity of their results to all good characteristics.

\subsection{A key lemma}
\label{ss:key-lemma}

A very important role in our arguments is played by the following easy lemma (see~\cite[Lemma~3.3.3]{by}).

\begin{lem}
\label{lem:key}
Let $k$ be an integral domain, and let $\KM$ be its field of fractions. Let $A$ and $B$ be $k$-algebras, and let $\varphi \colon A \to B$ be an algebra morphism. If the morphism $\KM \otimes_k \varphi \colon \KM \otimes_k A \to \KM \otimes_k B$ is an isomorphism, then the ``restriction of scalars'' functor
\[
\Mod(B) \to \Mod(A)
\]
is fully-faithful on modules which are $k$-free.
\end{lem}

We learnt how useful this observation can be in~\cite{by}. It also plays an important role in~\cite{modrap1}. In practice, this lemma can be used 
when we are interested in modules over a $k$-algebra $B$ which is ``complicated'' or not well understood, but for which we have a ``simplified model'' $A$ which is ``isomorphic to $B$ up to torsion'', i.e.~such that we have a $k$-algebra morphism $A \to B$ inducing an isomorphism $\KM \otimes_k A \simto \KM \otimes_k B$. 

For instance, on the ``constructible side'' of our proof of Theorem~\ref{thm:main}, working over $\FM$ directly we would have to consider the equivariant cohomology algebra $\HM^\bullet_{\ID}(\Gr;\FM)$. Using the results in~\cite{yz} (which rely on the geometric Satake equivalence) one can obtain a description of this algebra in terms of the distribution algebra of the universal centralizer associated with $\GB$. This algebra is a rather complicated object; in particular it is not finitely generated over $\FM$. On the other hand, if $\KM$ is a field of characteristic zero, the algebra $\HM^\bullet_{\ID}(\Gr;\KM)$ has a nice description (in terms similar to the description of the cohomology of a finite flag variety as a coinvariant algebra) which follows from~\cite[Theorem~1]{bf}.\footnote{This simpler description is related to the preceding one using the fact that, over a field of characteristic zero, the distribution algebra of a smooth group scheme is isomorphic to the enveloping algebra of its Lie algebra.} Hence, instead of working over $\FM$, our main constructions are done with coefficients in a finite localization $\RG$ of $\ZM$. In this case, we do not have a very explicit description of the algebra $\HM^\bullet_{\ID}(\Gr;\RG)$. But the same construction as the one used by Bezrukavnikov--Finkelberg in the case of $\KM$ provides a ``simplified model'' for this algebra, which is isomorphic to $\HM^\bullet_{\ID}(\Gr;\RG)$ ``up to $\RG$-torsion,'' see~\S\ref{ss:cohomology-Gr}. Since the modules over this algebra that we want to consider are all $\RG$-free, using Lemma~\ref{lem:key} we can replace the ``complicated algebra'' $\HM^\bullet_{\ID}(\Gr;\RG)$ by its ``simplified model'' without loosing any information.

Given this strategy, we also have to work over $\RG$ on the ``coherent side.'' Most of the complications appearing in this setting are treated in~\cite{br} and in~\cite{riche3}.

\subsection{Contents}

In Section~\ref{sec:preliminaries} we prove some preliminary technical results, and introduce some objects which will play an important role in the later sections. In Section~\ref{sec:constructible-side} we describe the left-hand side in~\eqref{eqn:equiv-main} in terms of our Soergel bimodules. In Section~\ref{sec:KW-reduction} we define the Kostant--Whittaker reduction functor, building on the main results of~\cite{riche3}. In Section~\ref{sec:tilting-KW} we use this functor to describe the right-hand side in~\eqref{eqn:equiv-main} in terms of Soergel bimodules. Finally, in Section~\ref{sec:proofs} we prove the results stated in the introduction.

\subsection{Some notation and conventions}

All rings in this paper are tacitly assumed to be commutative and unital.

If $A$ is an algebra, we denote by $\Mod(A)$ the category of left $A$-modules.
If $A$ is a $\ZM$-graded algebra, we denote by $\Modgr(A)$ the category of $\ZM$-graded left $A$-modules. We denote by $\langle 1 \rangle$ the shift of the grading defined by $(M \langle 1 \rangle)_n=M_{n-1}$. If $M$ is a free graded $A$-module of finite rank, we denote by $\grk_A(M) \in \ZM[\vv,\vv^{-1}]$ its graded rank, with the convention that $\grk_A(A\langle n \rangle)=\vv^{n}$.

If $X$ is a scheme, we denote by $\OC_X$ its structure sheaf, and by $\OC(X)$ the global sections of $\OC_X$.
If $Y$ is a scheme and $X$ is a $Y$-scheme, we denote by $\Omega_{X/Y}$, or simply $\Omega_X$, the sheaf of relative differentials, and by $\Omega(X/Y)$, or simply $\Omega(X)$, its global sections.

If $k$ is a ring and $V$ is a free $k$-module of finite rank, by abuse we still denote by $V$ the affine $k$-scheme $\Spec \bigl( \mathrm{S}_k(\Hom_k(V,k)) \bigr)$, where $\mathrm{S}$ denotes the symmetric algebra.

If $\RG$ is a finite localization of $\ZM$, we define a \emph{geometric point} of $\RG$ to be an algebraically closed field whose characteristic $p \geq 0$ is not invertible in $\RG$. If $\FM$ is a such a geometric point, then there exists a unique algebra morphism $\RG \to \FM$, so that tensor products on the form $\FM \otimes_\RG (-)$ make sense.

If $X$ is a Noetherian scheme and $\AB$ is an affine $X$-group scheme, we denote by $\Rep(\AB)$ the category of representations of $\AB$ which are coherent as $\OC_X$-modules. If $\AB$ is flat over $X$, then this category is abelian.

At various points in the paper we consider certain schemes and affine group schemes that could be defined over various base rings.  When it is not clear from context, we will use a subscript to specify the base ring.  For example, we write $X_\ZM$, resp.~$X_\FM$, to denote the $\ZM$-scheme $X$, resp.~its base change to $\FM$.  In order to avoid notational clutter, we will affix a single subscript $k$ to some constructions like fiber products, for example writing $(X \times_Z Y)_k$ and $\OC(X \times_Z Y)_k$ rather than $X_k \times_{Z_k} Y_k$ and $\OC(X_k \times_{Z_k} Y_k)$, or categories of equivariant coherent sheaves, writing $\Coh^G(X)_k$ instead of $\Coh^{G_k}(X_k)$. We will also use the abbreviation
$D^G(X)_k := \Db \Coh^{G_k} (X_k)$.

\subsection{Acknowledgements}

Soon after we started working on this paper, we were informed by
Pramod Achar and Laura Rider that they were also working on a similar project, which was completed in~\cite{arider2}. We thank them for keeping us informed of their progress, and for useful discussions. We also thank the referee for helpful comments.

\section{Preliminary results}
\label{sec:preliminaries}

\subsection{Reminder on parity sheaves}
\label{ss:reminder-parity}

Let $X$ be a complex algebraic variety, and
\[
X=\bigsqcup_{s \in \SSC} X_s
\]
be a finite (algebraic) stratification of $X$ into affine spaces. For any $s \in \SSC$, we denote by $i_s \colon X_s \hookrightarrow X$ the embedding. If $k$ is a Noetherian ring of finite global dimension, we denote by $\Db_\SSC(X,k)$ the $\SSC$-constructible derived category of sheaves of $k$-modules on $X$. Note that any $k$-local system on any stratum $X_s$ is constant. If $k'$ is a $k$-algebra which is also Noetherian and of finite global dimension, we will denote by
\[
k'(-) \colon \Db_\SSC(X,k) \to \Db_\SSC(X,k')
\]
the derived functor of extension of scalars.

The following is a slight generalization of~\cite[Definition~2.4]{jmw} (where it is assumed that $k$ is a complete local principal ideal domain).

\begin{defi}
\label{def:parity}
An object $\FC$ in $\Db_\SSC(X,k)$ is called $*$-\emph{even}, resp.~$!$-\emph{even}, if $\HC^n(i_s^*\FC)$, resp.~$\HC^n(i_s^! \FC)$, vanishes if $n$ is odd, and is a projective $k$-local system if $n$ is even. It is called \emph{even} if it is both $*$-even and $!$-even.

An object $\FC$ is called $*$-\emph{odd}, resp.~$!$-\emph{odd}, resp.~\emph{odd} if $\FC[1]$ is $*$-even, resp.~$!$-even, resp.~even.

Finally, an object is called a \emph{parity complex} if it is isomorphic to the direct sum of an even and an odd object.
\end{defi}

We will denote by $\Parity_\SSC(X,k)$ the full subcategory of $\Db_\SSC(X,k)$ consisting of parity complexes; it is stable under direct sums and direct summands. It is clear also that if $k'$ is a $k$-algebra satisfying the assumptions above, then the functor $k'(-)$ restricts to a functor from $\Parity_\SSC(X,k)$ to $\Parity_\SSC(X,k')$ (which we denote similarly).\footnote{This remark uses the property that the functor $k'(-)$ commutes with the functors $i_s^!$ and $i_s^*$. For $i_s^*$, this follows from~\cite[Proposition~2.6.5]{ks}. For $i_s^!$, we observe that~\cite[Proposition~3.1.11]{ks} provides a morphism of functors $k'(-) \circ i_s^! \to i_s^! \circ k'(-)$. To prove that this morphism is an isomorphism on $\Db_{\SSC}(X,k)$, it suffices to remark that $\Db_{\SSC}(X,k)$ is generated (as a triangulated category) by objects of the form $(i_{t})_*(\underline{M}_{X_t})$ where $t \in \SSC$ and $M$ is a finitely generated flat $k$-module, and that our morphism is clearly an isomorphism on such objects.}

We now assume that a connected complex algebraic group $A$ acts on $X$, stabilizing each stratum $X_s$. We will assume that $\HM^n_A(\mathrm{pt}; k)$ vanishes if $n$ is odd, and is projective over $k$ otherwise.
One can consider the $\SSC$-constructible $A$-equivariant derived category $\Db_{A,\SSC}(X,k)$ in the sense of Bernstein--Lunts. We let $\For \colon \Db_{A,\SSC}(X,k) \to \Db_{\SSC}(X,k)$ be the forgetful functor. More generally, if $A' \subset A$ is a closed subgroup satisfying the same assumption as $A$, we have a forgetful functor $\For \colon \Db_{A,\SSC}(X,k) \to \Db_{A',\SSC}(X,k)$. We say that an object $\FC$ in $\Db_{A,\SSC}(X,k)$ is a \emph{parity complex} if $\For(\FC)$ is a parity complex in the sense of Definition~\ref{def:parity}. We denote by $\Parity_{A,\SSC}(X,k)$ the subcategory of $\Db_{A,\SSC}(X,k)$ consisting of parity complexes. For $k'$ a $k$-algebra satisfying the assumptions above, we also have an ``extension of scalars'' functor $k'(-)$ in this setting. If $\SSC$ is the stratification by $A$-orbits, we will omit it from the notation.

We refer to~\cite{jmw} for the main properties of parity sheaves, in the case $k$ is a complete local principal ideal domain. Here we will only need the properties below, which follow from~\cite[Proposition~2.6]{jmw}. (Note that the proof in \emph{loc}.~\emph{cit}.~does not use the running assumptions on the ring of coefficients.)

\begin{lem}
\label{lem:properties-parity}
\begin{enumerate}
\item
\label{it:parity-cohomology}
If $\FC$ is in $\Parity_{A, \SSC}(X,k)$, then $\HM^\bullet_{A}(X, \FC)$
is a finitely generated projective module over $\HM^\bullet_A(\mathrm{pt}; k)$. If $A' \subset A$ is a closed subgroup as above, then the natural morphism
\[
\HM^\bullet_{A'}(\mathrm{pt}; k) \otimes_{\HM^\bullet_A(\mathrm{pt}; k)} \HM^\bullet_{A}(X, \FC) \to \HM^\bullet_{A'}(X, \For(\FC))
\]
is an isomorphism. If $k'$ is a $k$-algebra as above, then the natural morphism
\[
k' \otimes_k \HM^\bullet_{A}(X, \FC) \to \HM^\bullet_{A}(X, k'(\FC))
\]
is also an isomorphism.
\item
\label{it:morph-parity-For}
If $\FC$, $\GC$ are in $\Parity_{A, \SSC}(X,k)$, then $\Hom^\bullet_{\Db_{A,\SSC}(X, k)}(\FC, \GC)$ is a finitely generated projective $\HM^\bullet_A(\mathrm{pt}; k)$-module.
If $A' \subset A$ is a closed subgroup as above, then the natural morphism
\[
\HM^\bullet_{A'}(\mathrm{pt}; k) \otimes_{\HM^\bullet_A(\mathrm{pt}; k)} \Hom^\bullet_{\Db_{A,\SSC}(X, k)}(\FC, \GC) \to \Hom^\bullet_{\Db_{A',\SSC}(X, k)}(\For(\FC), \For(\GC))
\]
is an isomorphism.
If $k'$ is a $k$-algebra as above, then the natural morphism
\[
k' \otimes_k \Hom^\bullet_{\Db_{A,\SSC}(X, k)}(\FC, \GC) \to \Hom^\bullet_{\Db_{A,\SSC}(X, k')}(k'(\FC), k'(\GC))
\]
is also an isomorphism.\qed
\end{enumerate}
\end{lem}

\begin{remark}
In~\cite{jmw}, the ring $k$ is assumed to be a principal ideal domain, so every finitely generated projective $k$-module is free.
\end{remark}

We will also use the following observation.

\begin{lem}
\label{lem:parity-indec-For}
Assume that $k$ is complete local.
If $\FC$ is an object of $\Parity_{A,\SSC}(X,k)$ which is indecomposable, then its image $\For(\FC)$ in $\Parity_{\SSC}(X,k)$ is also indecomposable.
\end{lem}

\begin{proof}
Under our assumptions on $k$, $\Parity_{A,\SSC}(X,k)$ and $\Parity_{\SSC}(X,k)$ are Krull--Schmidt categories (see~\cite[Remark~2.1]{jmw}). Hence to prove the indecomposability of $\For(\FC)$ it suffices to prove that $\End_{\Db_{\SSC}(X, k)}(\For(\FC))$ is a local ring. By Lemma~\ref{lem:properties-parity} the ring
$\Hom^\bullet_{\Db_{\SSC}(X, k)}(\For(\FC), \For(\FC))$
is a quotient of $\Hom^\bullet_{\Db_{A, \SSC}(X, k)}(\FC, \FC)$. Hence $\End_{\Db_{\SSC}(X, k)}(\For(\FC))$ is a quotient of $\End_{\Db_{A, \SSC}(X, k)}(\FC)$. Since the latter is local, the former is also local, which finishes the proof.
\end{proof}

Assume now that $k$ is an integral domain. Recall that, in this setting, the \emph{rank} of a projective $k$-module $M$, denoted $\rk_k(M)$, is the dimension of the $\mathrm{Frac}(k)$-vector space $\mathrm{Frac}(k) \otimes_k M$.
To $X$ one can associate the free $\ZM[\vv,\vv^{-1}]$-module $\MM_X$ with basis $(\mathbf{e}_s : s \in \SSC)$ indexed by $\SSC$, and to any $\FC$ in $\Parity_{A,\SSC}(X,k)$ the elements
\begin{align*}
\ch_X^*(\FC)=\sum_{\substack{s \in \SSC \\ j \in \ZM}} \rk_k \bigl( \HM^{j-\dim(X_s)}(X_s, i_s^*\FC) \bigr) \cdot \vv^j \cdot \mathbf{e}_s, \\
\ch_X^!(\FC)=\sum_{\substack{s \in \SSC\\ j \in \ZM}} \rk_k \bigl( \HM^{j-\dim(X_s)}(X_s, i_s^!\FC) \bigr) \cdot \vv^{-j} \cdot \mathbf{e}_s.
\end{align*}
We also define a bilinear form $\langle- , - \rangle$ on $\MM_X$, with values in $\ZM[\vv,\vv^{-1}]$, by setting
\[
\langle \vv^n \mathbf{e}_s, \vv^m \mathbf{e}_t \rangle = \vv^{-n-m} \delta_{s,t}.
\]
The following result is also an easy application of~\cite[Proposition~2.6]{jmw}, whose proof is left to the reader.

\begin{lem}
\label{lem:grk-Hom-parity}
Assume that $k$ is an integral domain.
\begin{enumerate}
\item
\label{it:ch-parity-D}
If $\FC$ is in $\Parity_{A, \SSC}(X,k)$, then its Grothendieck--Verdier dual $\DM_X(\FC)$ is also in $\Parity_{A, \SSC}(X,k)$, and moreover we have
\[
\ch_X^!(\DM_X(\FC)) = \ch_X^*(\FC).
\]
\item
\label{it:dim-Hom-parity-ch}
If $\FC$, $\GC$ are in $\Parity_{A, \SSC}(X,k)$, then we have
\[
\pushQED{\qed}
\grk_{\HM^\bullet_A(\mathrm{pt}; k)} \bigl( \Hom^\bullet_{\Db_{A,\SSC}(X, k)}(\FC, \GC) \bigr) = \langle \ch_X^*(\FC), \ch_X^!(\GC) \rangle.
\qedhere
\popQED
\]
\end{enumerate}
\end{lem}

\begin{remark}
In the paper we will also use the straightforward generalization of the notions and results of this subsection to the case of ind-varieties, in the setting of~\cite[\S 2.7]{jmw}.
\end{remark}

\subsection{Equivariant coherent sheaves}
\label{ss:equiv-coh}

Let $X$ be a Noetherian scheme, and let $\HB$ be an affine group scheme over $X$. Consider a Noetherian $X$-scheme $Y$ endowed with an action of $\HB$, i.e.~we are given a morphism $a\colon \HB \times_X Y \to Y$ of $X$-schemes which satisfies the natural compatibility property with the group structure on $\HB$. Let also $p \colon \HB \times_X Y \to Y$ be the projection. Then
one can define the category $\Coh^\HB(Y)$ of $\HB$-equivariant coherent sheaves on $Y$ as the category of pairs $(\FC, \phi)$ where $\FC$ is a coherent sheaf on $Y$ and $\phi \colon p^*(\FC) \simto a^*(\FC)$ is an isomorphism which satisfies the usual cocyle condition. If $\HB$ is flat over $X$, then this category is abelian.

We define the ``universal stabilizer'' as the fiber product
\begin{equation}
\label{eqn:stabilizer}
\SB:= Y \times_{Y \times_X Y} (\HB \times_X Y)
\end{equation}
where the morphism $Y \to Y \times_X Y$ is the diagonal embedding (which we will denote by $\Delta$), and the morphism $\HB \times_X Y \to Y \times_X Y$ is $a \times p$. Then $\SB$ is an affine group scheme over $Y$, and the natural morphism $\SB \to \HB \times_X Y$ is a closed embedding of $Y$-group schemes.

The goal of this subsection is to recall the construction of a faithful functor
\begin{equation}
\label{eqn:stabilizer-functor}
\Coh^{\HB}(Y) \to \Rep(\SB)
\end{equation}
whose composition with the forgetful functor $\Rep(\SB) \to \Coh(Y)$ coincides with the natural forgetful functor $\For^\HB_Y \colon \Coh^\HB(Y) \to \Coh(Y)$.

We begin with a remark. Let $X'$ be a Noetherian $X$-scheme. Then one can consider the $X'$-group scheme $\HB':=\HB \times_X X'$, which acts naturally on $Y':=Y \times_X X'$. If $f \colon Y' \to Y$ is the natural morphism, then we remark that the usual pullback functor $f^*\colon \Coh(Y) \to \Coh(Y')$ induces a functor from $\Coh^\HB(Y)$ to $\Coh^{\HB'}(Y')$, which we also denote $f^*$. In particular, using this construction for the morphism $Y \to X$, we obtain a functor
\[
p_1^* \colon \Coh^{\HB}(Y) \to \Coh^{\HB \times_X Y}(Y \times_X Y),
\]
where $Y \times_X Y$ is considered as a $Y$-scheme through the second projection.

Using this functor we can define~\eqref{eqn:stabilizer-functor} as the composition
\[
\Coh^{\HB}(Y) \xrightarrow{p_1^*} \Coh^{\HB \times_X Y}(Y \times_X Y) \to \Coh^{\SB}(Y \times_X Y) \xrightarrow{\Delta^*} \Coh^{\SB}(Y) = \Rep(\SB),
\]
where the second arrow is the restriction functor, and $\SB$ acts trivially on $Y$. (Note that $\Delta$ is an $\SB$-equivariant morphism of $Y$-schemes.)

More generally,
if $Z \subset Y$ is a (not necessarily $\HB$-stable) closed subscheme, one can consider the restriction $\SB_Z:= Z \times_Y \SB$, and the composition 
\[
\Coh^\HB(Y) \to \Rep(\SB_Z)
\]
of~\eqref{eqn:stabilizer-functor} with restriction to $Z$. The composition of this functor with the forgetful functor $\Rep(\SB_Z) \to \Coh(Z)$ is the composition $\Coh^\HB(Y) \xrightarrow{\For^\HB_Y} \Coh(Y) \to \Coh(Z)$, where the second functor is restriction.

\subsection{Deformation to the normal cone}
\label{ss:def-normal-cone}

Let 
$A$ be a ring, and let
$I \subset A$ be an ideal. We define the associated \emph{deformation to the normal cone}\footnote{Our terminology comes from geometry: if $X$ is the spectrum of $A$ and $Y \subset X$ the closed subscheme associated with $I$, then the spectrum of $\DNC_I(A)$ is (a slight variant of) the usual deformation to the normal cone of $X$ along $Y$ considered e.g.~in~\cite[Chap.~5]{fulton}.} $\DNC_I(A)$ as the subalgebra of $A[\hbar^{\pm 1}]$ generated by $A[\hbar]$ together with the elements of the form $\hbar^{-1} f$ for $f \in I$.

\begin{lem}
\label{lem:def-normal-cone-BC}
Let $A'$ be a ring, and let $A \to A'$ be a flat ring morphism. We denote by $I'$ the ideal of $A'$ generated by the image of $I$. Then there exists a canonical isomorphism of $k$-algebras
\[
A' \otimes_A \DNC_I(A) \simto \DNC_{I'}(A').
\]
\end{lem}

\begin{proof}
The 
natural morphism
\[
A' \otimes_{A} \DNC_I(A) \to A' \otimes_{A} \bigl( A[\hbar^{\pm 1}] \bigr) = A'[\hbar^{\pm 1}]
\]
is injective. The image of this morphism is clearly $\DNC_{I'}(A')$, which provides the desired isomorphism.
\end{proof}

\subsection{Two lemmas on graded modules over polynomial rings}
\label{ss:nakayama}

Let $k$ be a ring, and let $V$ be a free $k$-module of finite rank. We denote by $A$ be the symmetric algebra of $V$ over $k$, which we consider as a $\ZM$-graded $k$-algebra with the generators $V \subset A$ in degree $2$. We consider the trivial $A$-module $k$ as a graded module concentrated in degree $0$.

\begin{lem}
\label{lem:nakayama}
\begin{enumerate}
\item 
\label{it:nakayama-Tor1}
Let $M$ be a graded $A$-module which is bounded below for the internal grading, and assume that
\[ 
k \otimes_A M \text{ is graded free over $k$} \quad \text{and} \quad \Tor_{1}^{A}(k,M)=0.
\]
Then $M$ is graded free over $A$.
\item 
\label{it:nakayama-complex}
Let $M$ be an object of the bounded derived category of graded $A$-modules, and assume that the cohomology modules of $M$ are all bounded below for the internal grading. Assume that the complex $k\lotimes_{A} M$ is concentrated in degree $0$, and graded free over $k$. Then $M$ is concentrated in degree $0$, and is graded free over $A$.
\end{enumerate}
\end{lem}

\begin{proof}
\eqref{it:nakayama-Tor1} Let $M'$ be a free graded $A$-module and let $f \colon M' \to M$ be a morphism such that the induced morphism $k \otimes_A M' \to k \otimes_{A} M$ is an isomorphism. (Such an $M'$ exists by our first assumption.) By the graded Nakayama lemma, $f$ is surjective. Let $M''$ be its kernel, and consider the exact sequence
\[
\Tor_{1}^{A}(k,M) \to k \otimes_{A} M'' \to k \otimes_{A} M' \to k \otimes_{A} M \to 0.
\]
This exact sequence shows, using our second assumption, that $k \otimes_{A} M''=0$. By the graded Nakayama lemma again, we deduce that $M''=0$, which finishes the proof.

\eqref{it:nakayama-complex} An easy argument using the graded Nakayama lemma shows that $M$ is concentrated in non-positive degrees. Now, consider the truncation triangle
\[
\tau_{<0} M \to M \to \mathsf{H}^0(M) \xrightarrow{+1}.
\]
Applying the functor $k \lotimes_{A}(-)$ we obtain a distinguished triangle
\[
k \lotimes_{A} \tau_{<0} M \to k \lotimes_{A} M \to k \lotimes_{A}\mathsf{H}^0(M) \xrightarrow{+1}.
\]
The associated long exact sequence of cohomology implies that $\Tor_{1}^{A}(k,\mathsf{H}^0(M))=0$, hence $\mathsf{H}^0(M)$ is free by~\eqref{it:nakayama-Tor1}. Considering again the long exact sequence we obtain that $k \lotimes_{A} \tau_{<0} M=0$, hence that $\tau_{<0} M=0$ by the graded Nakayama lemma, which finishes the proof.
\end{proof}

Now we assume 
that we are given an open $k$-subscheme $V' \subset V$.
Then, for any algebraically closed field $\FM$ and any ring morphism $k \to \FM$, we set 
\[
V_\FM := \Spec(\FM) \times_{\Spec(k)} V, \qquad
V'_\FM := \Spec(\FM) \times_{\Spec(k)} V'.
\]

\begin{lem}
\label{lem:morphism-grk}
Let $M$, $N$ be free graded $A$-modules of finite rank, and let $\varphi \colon M \to N$ be a morphism of graded $A$-modules. Assume that $\grk_{A}(M)=\grk_{A}(N)$, and that the morphism
\[
\varphi'_\FM \colon \OC(V'_\FM) \otimes_{A} M \to \OC(V'_\FM) \otimes_{A} N
\]
induced by $\varphi$ is injective for any algebraically closed field $\FM$ and any ring morphism $k \to \FM$. Then $\varphi$ is an isomorphism.
\end{lem}

\begin{proof}
%
%
Let $\varphi_\FM \colon \FM \otimes_k M \to \FM \otimes_k N$ be the morphism induced by $\varphi$. Then,
considering the commutative diagram
\[
\xymatrix@R=0.6cm{
\FM \otimes_k M \ar[d] \ar[r]^-{\varphi_\FM} & \FM \otimes_k N \ar[d] \\
\OC(V'_\FM) \otimes_{A} M \ar[r]^-{\varphi'_\FM} & \OC(V'_\FM) \otimes_{A} N
}
\]
(where the verticall arrows are injective since the morphism $V'_\FM \to V_\FM$ is an open embbeding) we see that $\varphi_\FM$ is injective, hence an isomorphism under our assumption on graded ranks. Now if $i \in \ZM$, the $k$-modules $M_i$ and $N_i$ are free of (the same) finite rank, and the determinant of the restriction $\varphi_i \colon M_i \to N_i$ of $\varphi$ (in any fixed choice of bases) does not belong to any prime ideal of $k$. Hence this determinant is invertible, proving that $\varphi$ is an isomorphism.
\end{proof}

\subsection{Affine braid groups and associated Hecke algebras}
\label{ss:Haff}

Let $(\XB,\Phi,{\check \XB},{\check \Phi})$ be a root datum, and let $\Phi^+ \subset \Phi$ be a system of positive roots. We will assume that ${\check \XB}/ \ZM {\check \Phi}$ has no torsion (or in other words that the connected reductive groups with root datum $(\XB,\Phi,{\check \XB},{\check \Phi})$ have a simply-connected derived subgroup). For $\alpha \in \Phi$, we denote by $\alpha^\vee$ the corresponding coroot, and by $s_\alpha$ the associated reflection.

Let $W$ be the corresponding Weyl group, and $\Waff := W \ltimes \XB$ be the associated affine Weyl group. To avoid confusions, for $\lambda \in \XB$ we denote by $t_\lambda$ the corresponding element of $\Waff$. We let $\ZM \Phi \subset \XB$ be the root lattice; then the subgroup $\Waff^\Cox:=W \ltimes (\ZM \Phi) \subset \Waff$ is a Coxeter group with generators given by reflections along walls of the fundamental dominant alcove. (In other words, we use the same conventions as in~\cite[\S 1.4]{LuAff}.)
The simple reflections which belong to $W$ will be called \emph{finite}; the ones which do not belong to $W$ will be called \emph{affine}.

Let us consider the length function $\ell \colon \Waff \to \ZM_{\geq 0}$ defined as follows: for $w \in W$ and $\lambda \in \XB$ we set
\begin{equation}
\label{eqn:length} 
\ell(w \cdot t_\lambda)=\sum_{\alpha \in \Phi^+ \cap w^{-1} (\Phi^+)} |\langle \lambda, \alpha^\vee \rangle | + \sum_{\alpha \in \Phi^+ \cap w^{-1}(-\Phi^+)}
  |1 + \langle \lambda, \alpha^\vee \rangle |.
\end{equation}
Then the restriction of $\ell$ to $\Waff^\Cox$ is the length function associated with the Coxeter structure considered above. We denote by $\O$ the subgroup of $\Waff$ consisting of elements of length $0$; it is a commutative group isomorphic to $\XB/ \ZM\Phi$ via the composition of natural maps $\O \hookrightarrow W \ltimes \XB \twoheadrightarrow \XB \twoheadrightarrow \XB/\ZM\Phi$. Moreover,
any element of $\Waff$ can be written in the form $\omega v$ for unique $\omega \in \Omega$ and $v \in \Waff^\Cox$. (For all of this, see~\cite[\S 1.5]{LuAff}.)

We will also consider the braid group $\Baff$ associated with $\Waff$. It is defined as the group generated by elements $T_w$ for $w \in \Waff$, with relations $T_{vw} = T_v T_w$ for all $v,w \in \Waff$ such that $\ell(vw)=\ell(v)+\ell(w)$. There exists a canonical surjection $\Baff \twoheadrightarrow \Waff$ sending $T_w$ to $w$. One can define (following Bernstein and Lusztig), for each $\lambda \in \XB$, an element $\theta_\lambda \in \Baff$, see e.g.~\cite[\S 1.1]{riche} for details. (This element is denoted $\overline{T}_\lambda$ in~\cite[\S 2.6]{LuAff}.) Then $\Baff$ admits a second useful presentation (usually called the \emph{Bernstein presentation}), with generators $\{T_w, \, w \in W\}$ and $\{\theta_\lambda, \, \lambda \in \XB\}$, and the following relations (where $v,w \in W$, $\lambda, \mu \in \XB$, and $\alpha$ runs over simple roots):
\begin{enumerate}
\item
$T_v T_w = T_{vw}$ \quad if $\ell(vw)=\ell(v)+\ell(w)$;
\item
$\theta_\lambda \theta_\mu = \theta_{\lambda+\mu}$;
\item
$T_{s_\alpha} \theta_\lambda = \theta_\lambda T_{s_\alpha}$ \quad if $\langle \lambda, \alpha^\vee \rangle = 0$;
\item
$\theta_\lambda = T_{s_\alpha} \theta_{\lambda-\alpha} T_{s_\alpha}$ \quad if $\langle \lambda, \alpha^\vee \rangle = 1$.
\end{enumerate}
(See~\cite{br-appendix} for a proof of this fact in the case $\XB/\ZM\Phi$ is finite; the general case is similar.)

The following lemma is proved in~\cite[Lemma~6.1.2]{riche2}.

\begin{lem}
\label{lem:affine-simple-conjugate}
For any affine simple reflection $s_0$, there exist a finite simple reflection $t$ and an element $b \in \Baff$ such that $T_{s_0} = b \cdot T_t \cdot b^{-1}$. \qed
\end{lem}

We define the affine Hecke algebra $\Haff$ as the quotient of the group algebra of $\Baff$ over $\ZM[\vv,\vv^{-1}]$ by the relations
\[
(T_s+\vv^{-1})(T_s-\vv)=0
\]
for all finite simple reflections $s$. (Note that the same formula for affine simple roots automatically follows by Lemma~\ref{lem:affine-simple-conjugate}.) We denote by $\HW$ the subalgebra of $\Haff$ generated by $\ZM[\vv,\vv^{-1}]$ and the elements $T_w$ for $w \in W$, and by $\MM_{\mathrm{triv}}$ the $\HW$-module which is free of rank one over $\ZM[\vv,\vv^{-1}]$, and where $T_s$ acts by multiplication by $\vv$ for each finite simple reflection $s$. 

We define the ``spherical'' right $\Haff$-module
\[
\Msph := \MM_{\mathrm{triv}} \otimes_{\HW} \Haff.
\]
We denote by $\mm_0 \in \Msph$ the element $1 \otimes 1$. For $\lambda \in \XB$, we denote by $w_\lambda$ the shortest representative in $W t_\lambda \subset \Waff$, and set $\mm_\lambda:=\mm_0 \cdot T_{w_\lambda}$. Then the elements $\mm_\lambda$, $\lambda \in \XB$, form a $\ZM[\vv,\vv^{-1}]$-basis of $\Msph$. We define a bilinear form $\langle -, - \rangle$ on $\Msph$, with values in $\ZM[\vv,\vv^{-1}]$, by setting
\[
\langle \vv^i \mm_\lambda, \vv^j \mm_\mu \rangle = \vv^{-j-i} \delta_{\lambda,\mu}.
\]

For any sequence of simple reflections $\us=(s_1, \cdots, s_r)$ and any $\omega \in \Omega$ we will consider the element
\[
\mm(\omega, \us) := \mm_0 \cdot T_{\omega} \cdot (T_{s_1} + \vv^{-1}) \cdots (T_{s_r}+\vv^{-1}) \qquad \in \Msph.
\]

\begin{remark}
Our element $T_s \in \Haff$ corresponds to the element denoted $H_s$ in~\cite{soergel-kl}, while our $\vv$ corresponds to $v^{-1}$ in~\cite{soergel-kl}.
\end{remark}

\section{Constructible side}
\label{sec:constructible-side}

\subsection{Overview}
\label{ss:overview}

In this section we describe the category $\Parity_{(\ID)}(\Gr, \FM)$ in terms of an appropriate category of Soergel bimodules using a ``total cohomology'' functor. Similar constructions appear in~\cite{soergel-kategorie, soergel, modrap1} for flag varieties of reductive groups, and in~\cite{by} for flag varieties of Kac--Moody groups (with coefficients in characteristic zero).
The main difference with these works is that in our case the cohomology algebra is much more complicated. To overcome this difficulty we work over a certain ring $\RG$ of integers, which allows to replace this cohomology algebra by a ``simplified model,'' see~\S\ref{ss:key-lemma} for a discussion of this idea.

After setting the notation in~\S\ref{ss:notation-constructible}, we introduce our ``Soergel bimodules'' in~\S\S\ref{ss:deformations}--\ref{ss:algebraic-BS}. In~\S\ref{ss:cohomology-Gr} we study the equivariant cohomology of $\Gr$. In~\S\ref{ss:topological-BS} we explain how the category $\Parity_{(\ID)}(\Gr, \FM)$ can be recovered from a certain category of (equivariant) ``Bott--Samelson parity sheaves'' over $\RG$. Then in~\S\ref{ss:cohomology-functors} we introduce our ``total cohomology functor,'' and in~\S\S\ref{ss:equivalence}--\ref{ss:proof-H-ff} we prove that this functor induces an equivalence between ``Bott--Samelson'' parity sheaves and Soergel bimodules.
Finally, in~\S\ref{ss:graded-ranks-parity} we derive a formula for the graded rank of the space of morphisms between certain Soergel bimodules, which will play an important role in a proof on the ``coherent side'' (see~\S\ref{ss:equivalence-tilting}).

\subsection{Notation}
\label{ss:notation-constructible}

In this section we let $\GD$ be a connected reductive algebraic group over $\CM$, with a chosen Borel subgroup $\BD \subset \GD$ and maximal torus $\TD \subset \BD$. We let ${\check \XB}:=X^*(\TD)$ be the lattice of characters of $\TD$, and ${\check \Phi} \subset {\check \XB}$ be the roots of $\GD$. We also fix a finite localization $\RG$ of $\ZM$. In the whole section we will make the following assumptions:
\begin{enumerate}
\item
${\check \XB} / \ZM {\check \Phi}$ has no torsion (or in other words the connected reductive groups which are Langlands dual to $\GD$ have a simply-connected derived subgroup);
\item
all the torsion primes of the ``refined root system'' ${\check \Phi} \subset {\check \XB}$ (in the sense of~\cite[Section~5]{demazure}) are invertible in $\RG$.
\end{enumerate}
In the later sections we will apply our results in the case $\GD$ is a product of simple groups (of adjoint type) and general linear groups; in this case the first condition is automatic, and the second condition means that the prime numbers which are not very good for some simple factor of $\GD$ are invertible in $\RG$.

Let $\mathscr{O}:=\CM[ \hspace{-1pt} [z] \hspace{-1pt} ]$ and $\mathscr{K}:=\CM( \hspace{-1pt} (z) \hspace{-1pt} )$. We consider the affine Grassmannian
\[
\Gr := \GD(\mathscr{K}) / \GD (\mathscr{O}),
\]
with its natural ind-variety structure. We denote by $\ID$ the Iwahori subgroup of $\GD(\mathscr{O})$ determined by $\BD$, i.e.~the inverse image of $\BD$ under the morphism $\GD(\mathscr{O}) \to \GD$ defined by the evaluation at $z=0$. Then $\ID$ acts naturally on $\Gr$ via left multiplication on $\GD(\mathscr{K})$. We also let the multiplicative group $\Gm$ act on $\Gr$ by loop rotation (i.e.~via $x \cdot g(z) = g(x^{-1} z)$), so that we obtain an action of the semi-direct product $\ID \rtimes \Gm$.

The main players of this section are the categories
\[
\Parity_{\ID \rtimes \Gm} (\Gr,\RG), \quad \Parity_{\ID} (\Gr,\RG) \quad \text{and} \quad \Parity_{(\ID)}(\Gr,\RG).
\]
Here, in the right-hand side, we use the notation $\Parity_{(\ID)}(\Gr,\RG)$ for the category $\Parity_{\SSC}(\Gr,\RG)$ where $\SSC$ is the stratification of $\Gr$ by orbits of $\ID$. If $\FM$ is a field (not necessarily algebraically closed) whose characteristic is invertible in $\RG$, we can consider the unique algebra morphism $\RG \to \FM$, and the categories
\begin{equation}
\label{eqn:Parity-FM}
\Parity_{\ID \rtimes \Gm} (\Gr,\FM), \quad \Parity_{\ID} (\Gr,\FM) \quad \text{and} \quad \Parity_{(\ID)}(\Gr,\FM).
\end{equation}

We let $\XB:=X_*(\TD)$ be the lattice of cocharacters of $\TD$, and $\Phi \subset \XB$ be the coroots of $\GD$ (with respect to $\TD$). The choice of $\BD$ determines a system of positive roots: more precisely we denote by ${\check \Phi}^+ \subset {\check \Phi}$ the roots which are \emph{opposite} to the $\TD$-weights in the Lie algebra of $\BD$. We denote by $\Phi^+ \subset \Phi$ the corresponding system of positive coroots.
To these data one can associate the affine Weyl group $\Waff$ and its length function $\ell$ as in~\S\ref{ss:Haff}.

Recall that the $\ID$-orbits on $\Gr$ are parametrized in a natural way by $\Waff/W \cong \XB$, and that each $\ID$-orbit is stable under the action of $\ID \rtimes \Gm$. More precisely, any $\lambda \in \XB$ defines a point $z^\lambda \in \TD(\mathscr{K}) \subset \GD(\mathscr{K})$. We set $L_\lambda := z^\lambda \GD(\mathscr{O}) / \GD(\mathscr{O}) \in \Gr$, and $\Gr_\lambda:=\ID \cdot L_\lambda$. Then we have
\[
\Gr = \bigsqcup_{\lambda \in \XB} \Gr_\lambda.
\]
Moreover, the dimension of $\Gr_\lambda$ is the length of the shortest representative in $t_\lambda W \subset \Waff$, i.e.~$\ell(w_{-\lambda})$.
%
For any $w \in t_\lambda W$ we also set $\Gr_w := \Gr_\lambda$.


Let $\FM$ be as above, and consider the categories 
in~\eqref{eqn:Parity-FM}. By~\cite{jmw} these categories
are all Krull--Schmidt, and their indecomposable objects can be described as follows. For any $\lambda \in \XB$ there exists a unique indecomposable object $\EC_\lambda$ in $\Parity_{\ID \rtimes \Gm} (\Gr,\FM)$
which is supported on $\overline{\Gr_\lambda}$
and whose restriction to $\Gr_\lambda$
is 
$\underline{\FM}_{\Gr_\lambda}[\dim(\Gr_\lambda)]$.
Moreover any indecomposable object in $\Parity_{\ID \rtimes \Gm} (\Gr,\FM)$
is isomorphic to $\EC_\lambda[n]$
for some unique $\lambda \in \XB$ and $n \in \ZM$. By Lemma~\ref{lem:parity-indec-For}, the images of $\EC_\lambda$
under the appropriate forgetful functors to $\Parity_{\ID} (\Gr,\FM)$ and $\Parity_{(\ID)} (\Gr,\FM)$
remain indecomposable; for simplicity these images will still be denoted by $\EC_\lambda$.
The same description of indecomposable objects as above applies in the categories $\Parity_{\ID} (\Gr,\FM)$
and $\Parity_{(\ID)} (\Gr,\FM)$.

The connected components of $\Gr$ are parametrized by $\O$; for $\omega \in \O$ we denote by $\Gr_{(\omega)}$ the corresponding component. (In fact, if we identify $\O$ with $\XB/\ZM\Phi$ as in~\S\ref{ss:Haff}, then $\Gr_{(\omega)}$ is the union of the orbits $\Gr_\lambda$ where $\lambda$ has image $\omega$ in $\XB/\ZM\Phi$.)

Below we will also use the affine flag variety
\[
\Fl := \GD(\mathscr{K}) / \ID,
\]
with its natural ind-variety structure, and the natural $\ID$-action. The $\ID$-orbits on $\Fl$ are parametrized in a natural way by $\Waff$, and are stable under the loop rotation action. If $\Fl_w$ is the orbit associated with $w \in \Waff$, then we have
$\dim(\Fl_w)=\ell(w)$, and the image of $\Fl_w$ under the natural projection $\Fl \twoheadrightarrow \Gr$ is $\Gr_w$.
For each simple reflection $s$, the orbit $\Fl_s$ is isomorphic to $\AM^1_{\CM}$, and its closure $\overline{\Fl_s}$ is isomorphic to $\PM^1_{\CM}$.

We define
\[
\tg := \XBD \otimes_\ZM \RG, \qquad \tg^* := \Hom_\RG(\tg_\RG, \RG) = \XB \otimes_\ZM \RG.
\]
(In fact, $\tg$ is the Lie algebra of the $\RG$-torus which is Langlands dual to $\TD$.) Then there exists a canonical isomorphism of graded $\RG$-algebras 
\begin{equation}
\label{eqn:cohomology-TD}
\OC(\tg^*) = \mathrm{S}(\tg) \simto \HM^\bullet_{\TD}(\mathrm{pt}; \RG),
\end{equation}
where $\tg$ is in degree $2$, and $\mathrm{S}(\tg)$ is the symmetric algebra of $\tg$. Using~\cite[Theorem~1.3(2)]{totaro}, we deduce, under our assumptions on $\RG$,\footnote{Recall that the torsion primes of $\GD$ (in the sense of~\cite{totaro}) are the same as the torsion primes of the ``refined root system'' ${\check \Phi} \subset {\check \XB}$ (in the sense of~\cite{demazure}).} a canonical isomorphism
\begin{equation}
\label{eqn:cohomology-GD}
\OC(\tg^*/ W) \simto \HM^\bullet_{\GD}(\mathrm{pt}; \RG).
\end{equation}
We will also identify $\HM^\bullet_{\Gm}(\mathrm{pt}; \RG)$ with $\RG[\hbar]$ (where $\hbar$ is an indeterminate, in degree $2$) in the natural way.

\begin{lem}
\label{lem:t/W-affine-space}
The $\RG$-scheme $\tg^*/W$ is isomorphic to an affine space. Moreover, $\OC(\tg^*)$ is free over $\OC(\tg^*/W)$.
\end{lem}

\begin{proof}
The claims follow from~\cite[Th{\'e}or{\`e}me~2(c) \& Th{\'e}or{\`e}me~3]{demazure} and our assumption on $\RG$, since $\OC(\tg^*/W)=\mathrm{S}_\ZM({\check \XB})^W \otimes_\ZM \RG$.
\end{proof}

\begin{remark}
The main result of this section will be proved in the $\ID$-equivariant setting; the $\ID \rtimes \Gm$-equivariant setting will be used only for technical purposes. However, similar results hold in the $\ID \rtimes \Gm$-equivariant case. On the ``coherent side'', replacing $\ID$ by $\ID \rtimes \Gm$ amounts to deforming coherent sheaves on $\tgg$ to asymptotic $\DC$-modules on $\Flag$; see~\cite{dodd} for details in the characteristic zero case.
\end{remark}

\subsection{Some algebras}
\label{ss:deformations}

We denote by $\Delta \subset \tg^* / W \times \tg^* / W$ the diagonal copy of $\tg^* / W$, and by $I_1 \subset \OC(\tg^* / W \times \tg^* / W)$ the associated ideal
(i.e.~the ideal generated by elements of the form $f \otimes 1 - 1 \otimes f$ for $f \in \OC(\tg^* / W)$). We also denote by $I_2 \subset \OC(\tg^* / W \times \tg^*)$, resp.~$I_3 \subset \OC(\tg^* \times \tg^*)$, the ideal generated by the image of $I_1$ under the ring morphism associated with the quotient morphism $\tg^* / W \times \tg^* \to \tg^* / W \times \tg^* /W$, resp.~$\tg^* \times \tg^* \to \tg^* / W \times \tg^* /W$.

We will consider the $\ZM$-graded algebras\footnote{Here the letter ``$C$'' stands for ``coinvariants,'' since these algebras will play the role played by the coinvariant algebra in~\cite{soergel-kategorie, soergel, modrap1}.}
\begin{align*}
C_\hbar := \DNC_{I_2}(\OC(\tg^* / W \times \tg^*)),& \qquad \tC_\hbar := \DNC_{I_3}(\OC(\tg^* \times \tg^*)), \\
C := C_\hbar / \hbar \cdot C_\hbar,& \qquad \tC := \tC_\hbar / \hbar \cdot \tC_\hbar.
\end{align*}
Here the grading is induced by the grading on $\OC(\tg^*)$ and $\RG[\hbar]$ from~\S\ref{ss:notation-constructible}. We also set $C'_\hbar := \DNC_{I_1}(\OC(\tg^* / W \times \tg^*/W))$.

By Lemmas~\ref{lem:def-normal-cone-BC} and~\ref{lem:t/W-affine-space},
we have canonical isomorphisms
\begin{equation}
\label{eqn:isom-Chbar}
C_\hbar \cong \OC(\tg^* / W \times \tg^*) \otimes_{\OC(\tg^* / W \times \tg^*/W)} C'_\hbar, \quad
\tC_\hbar \cong \OC(\tg^* \times \tg^*) \otimes_{\OC(\tg^* / W \times \tg^*)} C_\hbar.
\end{equation}

%

\begin{lem}
\label{lem:deformation-flat}
\begin{enumerate}
\item
The two natural 
ring morphisms
$\OC(\tg^*/W)[\hbar] \to C'_\hbar$
are flat.
\label{it:def-flat-1}
\item
The natural 
ring morphisms $\OC(\tg^*/W)[\hbar] \to C_\hbar$ and $\OC(\tg^*)[\hbar] \to C_\hbar$
are flat.
\label{it:def-flat-2}
\item
The two natural ring morphisms
$\OC(\tg^*)[\hbar] \to \tC_\hbar$
are flat.
\label{it:def-flat-3}
\end{enumerate}
\end{lem}

\begin{proof}
First we treat~\eqref{it:def-flat-1}.
By symmetry, it is enough to prove the claim in the case of the morphism induced by the first projection $\tg^* / W \times \tg^*/W \to \tg^* / W$.
By Lemma~\ref{lem:t/W-affine-space}, we can fix an isomorphism of $\RG$-schemes
$\tg^* / W \cong \mathbb{A}^n_\RG$ for some $n \in \ZM_{\geq 0}$. Then $\Delta$ identifies with the diagonal copy of $\mathbb{A}^n_\RG$ in $\mathbb{A}^{2n}_\RG$. Writing $\mathbb{A}^{2n}_\RG$ as the direct sum of the diagonal and antidiagonal copies of $\mathbb{A}^n_\RG$, we obtain ring isomorphisms
\[
C'_\hbar \cong \RG[x_1, \cdots, x_n] \otimes \DNC_{I_+}(\RG[y_1, \cdots, y_n]) \\
\cong \RG[x_1, \cdots, x_n, z_1, \cdots, z_n, \hbar],
\]
where $I_+ \subset \RG[y_1, \cdots, y_n]$ is the ideal of the subscheme $\{0\} \subset \mathbb{A}^n_{\RG}$, and $z_i := \hbar^{-1} y_i$.
With these identifications, the morphism under consideration is defined by $\hbar \mapsto \hbar$, $x_i \mapsto x_i + \hbar z_i$. This morphism is clearly flat.


The second claim in~\eqref{it:def-flat-2} is an immediate consequence of~\eqref{it:def-flat-1} and the first isomorphism in~\eqref{eqn:isom-Chbar}. To prove the first claim we decompose the morphism as the composition $\OC(\tg^*/W)[\hbar] \to C_\hbar' \to C_\hbar$. Now the first morphism is flat by~\eqref{it:def-flat-1}, and the second one is flat by the first isomorphism in~\eqref{eqn:isom-Chbar} since the projection $\tg^* \to \tg^*/W$ is flat (see Lemma~\ref{lem:t/W-affine-space}). This implies the desired claim.

Finally, in~\eqref{it:def-flat-3} the flatness of the morphism induced by the first projection $\tg^* \times \tg^* \to \tg^*$ follows from~\eqref{it:def-flat-2} and the second isomorphism in~\eqref{eqn:isom-Chbar}. Then the flatness of the other morphism follows by symmetry.
\end{proof}


Let us denote by
\[
f_1, f_2 \colon \OC(\tg^*)[\hbar] \to \tC_\hbar
\]
the morphisms considered in Lemma~\ref{lem:deformation-flat}\eqref{it:def-flat-3}.
These morphisms
endow $\tC_\hbar$ with the structure of a graded $\OC(\tg^*)[\hbar]$-bimodule.
In fact this algebra has a natural structure of bialgebra in the category of graded $\OC(\tg^*)[\hbar]$-bimodules, constructed as follows.
Consider the $\RG[\hbar]$-algebra morphism
\begin{equation}
\label{eqn:comultiplication}
\OC(\tg^* \times \tg^*)[\hbar] \to \tC_\hbar \otimes_{\OC(\tg^*)[\hbar]} \tC_{\hbar}
\end{equation}
sending any $x$ in the first copy of $\OC(\tg^*)$ to $f_1(x) \otimes 1$, and any $y$ in the second copy of $\OC(\tg^*)$ to $1 \otimes f_2(y)$. One can easily check that the image under this morphism of any element of the form $g \otimes 1 - 1 \otimes g$ with $g \in \OC(\tg^*/W)$ belongs to $\hbar \cdot (\tC_\hbar \otimes_{\OC(\tg^*)[\hbar]} \tC_{\hbar})$. It follows from Lemma~\ref{lem:deformation-flat}\eqref{it:def-flat-3} that $\tC_\hbar \otimes_{\OC(\tg^*)[\hbar]} \tC_{\hbar}$ is flat over $\RG[\hbar]$; in particular it has no $\hbar$-torsion. Hence~\eqref{eqn:comultiplication} factors in a unique way through a graded $\RG[\hbar]$-algebra morphism
\[
\tC_\hbar \to \tC_\hbar \otimes_{\OC(\tg^*)[\hbar]} \tC_{\hbar},
\]
which provides our comultiplication morphism.
Using this structure, if $M$ and $N$ are graded $\tC_\hbar$-modules, then we obtain that the tensor product $M \otimes_{\OC(\tg^*)[\hbar]} N$ has a natural structure of graded $\tC_\hbar$-module.

Similar constructions provide a structure of graded $(\OC(\tg^*/W)[\hbar], \OC(\tg^*)[\hbar])$-bimo\-dule on $C_\hbar$, and a graded $\RG[\hbar]$-algebra morphism
\[
C_\hbar \to C_\hbar \otimes_{\OC(\tg^*)[\hbar]} \tC_{\hbar}.
\]
Hence, if $M$ is a graded $C_\hbar$-module and $N$ is a graded $\tC_\hbar$-module, we obtain that the tensor product $M \otimes_{\OC(\tg^*)[\hbar]} N$ has a natural structure of graded $C_\hbar$-module.

Applying the functor $\RG \otimes_{\RG[\hbar]} (-)$, we also obtain graded algebra morphisms
\[
\tC \to \tC \otimes_{\OC(\tg^*)} \tC, \qquad C \to C \otimes_{\OC(\tg^*)} \tC,
\]
and the corresponding structures for tensor products of graded modules.

\subsection{``Algebraic'' Bott--Samelson category}
\label{ss:algebraic-BS}

The group $\Waff$ acts naturally on $\tg^* \times \AM^1_{\RG}$, via the formulas
\[
v \cdot (\xi,x)=(v \cdot \xi, x), \qquad t_\lambda \cdot (\xi,x) = (\xi + x \lambda, x)
\]
for $\xi \in \tg^*$, $x \in \AM^1_\RG$, $v \in W$ and $\lambda \in \XB$. For this action, the subspace $\tg^* = \tg^* \times \{0\} \subset \tg^* \times \AM^1_\RG$ is stable, and the action of $\Waff$ on this subspace factors through the natural action of $W=\Waff / \XB$.
The $\Waff$-action on $\tg^* \times \AM^1_\RG$ induces an action on the graded algebra $\OC(\tg^*)[\hbar] = \OC(\tg^* \times \AM^1_\RG)$. If $w \in \Waff$, we denote by $(\OC(\tg^*)[\hbar])^w$ the subalgebra of $w$-invariants.

Below we will need the following easy lemma.

\begin{lem}
\label{lem:invariants-hbar}
For any simple reflection $s$, the morphism
\[
\RG \otimes_{\RG[\hbar]} (\OC(\tg^*)[\hbar])^s \to \OC(\tg^*)^s
\]
induced by the restriction morphism $\OC(\tg^*)[\hbar] \to \OC(\tg^*)$
is an isomorphism.
\end{lem}

\begin{proof}
If $s$ is finite, then the claim is obvious. The general case follows, using Lemma~\ref{lem:affine-simple-conjugate}.
\end{proof}

For any $w \in \Waff$, we define the graded $\OC(\tg^* \times \tg^*)[\hbar]$-module $E^{\hbar}_w$ as follows. As a graded $\RG$-module, we have $E^\hbar_w=\OC(\tg^*)[\hbar]$. The \emph{right} copy of $\OC(\tg^*)[\hbar]$ in $\OC(\tg^* \times \tg^*)[\hbar] = \OC(\tg^*)[\hbar] \otimes_{\RG[\hbar]} \OC(\tg^*)[\hbar]$ acts in the natural way, by multiplication. And any $f$ in the \emph{left} copy of $\OC(\tg^*)[\hbar]$ acts by multiplication by $w^{-1} \cdot f$.
Since the induced action of $\OC(\tg^* \times \tg^*)$ on $E^\hbar_w / \hbar \cdot E^\hbar_w$ factors through an action of $\OC(\tg^* \times_{\tg^*/W} \tg^*)$, there exists a unique extension of the action of $\OC(\tg^* \times \tg^*)[\hbar]$ to an action of $\tC_\hbar$ on $E^\hbar_w$. By restriction, one can also consider $E^\hbar_w$ as a graded $C_\hbar$-module.

If $s$ is a simple reflection, we also consider the graded $\OC(\tg^* \times \tg^*)[\hbar]$-module 
\[
D^\hbar_s:= \OC(\tg^*)[\hbar] \otimes_{(\OC(\tg^*)[\hbar])^s} \OC(\tg^*)[\hbar] \langle -1 \rangle.
\]
Using the same arguments as above and Lemma~\ref{lem:invariants-hbar}, one can check that the action of $\OC(\tg^* \times \tg^*)[\hbar]$ on $D^\hbar_s$ extends in a canonical way to an action of $\tC_\hbar$, so that $D^\hbar_s$ can be considered as a graded $\tC_\hbar$-module.

Now, using the constructions of~\S\ref{ss:deformations} one can define, for any $\omega \in \Omega$ and
any sequence $\underline{s}=(s_1, \cdots, s_n)$ of simple reflections, the graded $C_\hbar$-module
\[
D_{\hbar}(\omega, \underline{s}) := E^\hbar_\omega \otimes_{\OC(\tg^*)[\hbar]} D^\hbar_{s_1} \otimes_{\OC(\tg^*)[\hbar]} \cdots \otimes_{\OC(\tg^*)[\hbar]} D^\hbar_{s_n}.
\]

We will also consider the corresponding constructions for $C$-modules: we set
\[
E_w:=E^\hbar_w / \hbar \cdot E^\hbar_w, \qquad D_s := D_s^\hbar / \hbar \cdot D_s^\hbar,
\]
and then
\[
D(\omega, \underline{s}) := D_{\hbar}(\omega, \underline{s}) / \hbar \cdot D_{\hbar}(\omega, \underline{s}) \cong E_\omega \otimes_{\OC(\tg^*)} D_{s_1} \otimes_{\OC(\tg^*)} \cdots \otimes_{\OC(\tg^*)} D_{s_n}.
\]
Note for later use that, by Lemma~\ref{lem:invariants-hbar}, we have a canonical isomorphism
\begin{equation}
\label{eqn:Ds-tensor-product}
D_s \simto \OC(\tg^*) \otimes_{\OC(\tg^*)^s} \OC(\tg^*).
\end{equation}

With these definitions one can define the category $\BSalg$ with
\begin{itemize}
\item
\emph{objects}: triples $(\omega, \underline{s}, i)$ with $\omega \in \Omega$, $\underline{s}$ a sequence of simple reflections indexed by $(1, \cdots, n)$ for some $n \in \ZM$, and $i \in \ZM$;
\item 
\emph{morphisms}:
for $\omega, \omega' \in \Omega$, $\underline{s}$ and $\underline{t}$ sequences of simple reflections, and $i,j \in \ZM$,
\begin{multline*}
\Hom_{\BSalg} \bigl( (\omega, \underline{s}, i), (\omega', \underline{t}, j) \bigr) =
\\
\begin{cases}
\Hom_{\Modgr(C)} \bigl( D(\omega, \underline{s}) \langle -i \rangle, D(\omega, \underline{t}) \langle -j \rangle \bigr) & \text{if $\omega=\omega'$;} \\
0 & \text{if $\omega \neq \omega'$.}
\end{cases}
\end{multline*}
\end{itemize}

\subsection{Equivariant cohomology of $\Gr$}
\label{ss:cohomology-Gr}

The graded algebras $C_\hbar$ and $C$ defined in~\S\ref{ss:deformations} can be used to describe the algebras $\HM^\bullet_{\ID \rtimes \Gm} (\Gr;\RG)$ and $\HM^\bullet_{\ID} (\Gr;\RG)$ ``up to torsion,'' as follows.
There exists a natural graded algebra morphism 
\[
\OC(\tg^*)[\hbar] \xrightarrow[\sim]{\eqref{eqn:cohomology-TD}} \HM^\bullet_{\ID \rtimes \Gm}(\mathrm{pt};\RG) \to \HM^\bullet_{\ID \rtimes \Gm} (\Gr;\RG).
\]
On the other hand, we have a canonical isomorphism 
\[
\HM^\bullet_{\ID \rtimes \Gm} (\Gr;\RG) \cong \HM^\bullet_{\GD(\mathscr{O}) \rtimes \Gm} \bigl( (\ID \rtimes \Gm) \backslash (\GD(\mathscr{K}) \rtimes \Gm);\RG \bigr),
\]
so that there also exists a natural graded algebra morphism
\[
\OC(\tg^*/W)[\hbar] \xrightarrow[\sim]{\eqref{eqn:cohomology-GD}} \HM^\bullet_{\GD(\mathscr{O}) \rtimes \Gm}(\mathrm{pt};\RG) \to \HM^\bullet_{\ID \rtimes \Gm} (\Gr;\RG).
\]
induced by the multiplication of $\GD(\mathscr{O}) \rtimes \Gm$ on $\GD(\mathscr{K}) \rtimes \Gm$ on the \emph{right}.
Using the fact that the projection $\GD(\mathscr{K}) \rtimes \Gm \to \mathrm{pt}$ factors through $\Gm$, it is not difficult to check that the images of $\hbar$ under these two morphisms coincide (see the proof of Lemma~\ref{lem:morphism-cohomology-factorization} below for similar considerations). Hence
combining them we obtain an algebra morphism
\begin{equation}
\label{eqn:morp-cohomology-Gr}
\OC(\tg^*/W \times \tg^*)[\hbar] \to \HM^\bullet_{\ID \rtimes \Gm} (\Gr;\RG).
\end{equation}

\begin{remark}
Note that we have switched the order of the factors here: the left-hand factor of $\tg^*/W \times \tg^*$ is related to the multiplication of $\GD(\mathscr{O})$ on $\GD(\mathscr{K})$ on the right, while the right-hand factor is related to the multiplication of $\ID$ on $\GD(\mathscr{K})$ on the left. This choice of convention complicates some formulas in this section, but it will make the comparison with the constructions on the ``coherent'' side easier. Another option would have been to work with the (less customary) variety $\Gr':=\GD(\mathscr{O}) \backslash \GD(\mathscr{K})$ instead of $\Gr$. Here $\Gr'$ is isomorphic to $\Gr$ through $\GD(\mathscr{O}) g \mapsto g^{-1} \GD(\mathscr{O})$, but this isomorphism switches the role of left and right multiplications.
\end{remark}

\begin{lem}
\label{lem:morphism-cohomology-factorization}
Morphism~\eqref{eqn:morp-cohomology-Gr} factors in a unique way through a graded $\RG[\hbar]$-algebra morphism
\[
\gamma_\hbar \colon C_\hbar \to \HM^\bullet_{\ID \rtimes \Gm} (\Gr;\RG).
\]
\end{lem}

\begin{proof}
There exist natural converging spectral sequences
\begin{align*}
E_2^{pq} = \HM^{p}_{\GD(\mathscr{O}) \rtimes \Gm}(\mathrm{pt}; \RG) \otimes_\RG \HM^{q}(\Gr; \RG) &\Rightarrow \HM^{p+q}_{\GD(\mathscr{O}) \rtimes \Gm} (\Gr;\RG), \\
E_2^{pq} = \HM^{p}_{\ID \rtimes \Gm}(\mathrm{pt}; \RG) \otimes_\RG \HM^{q}(\Gr; \RG) &\Rightarrow \HM^{p+q}_{\ID \rtimes \Gm} (\Gr;\RG).
\end{align*}
In both cases, the spectral sequence degenerates since the left-hand side vanishes unless $p$ and $q$ are even. This implies in particular that $\HM^\bullet_{\ID \rtimes \Gm} (\Gr;\RG)$ is $\RG[\hbar]$-free, proving the unicity of the factorization. It also follows that the natural morphism
\[
\HM^\bullet_{\ID \rtimes \Gm}(\mathrm{pt}; \RG) \otimes_{\HM^\bullet_{\GD(\mathscr{O}) \rtimes \Gm}(\mathrm{pt}; \RG)} \HM^\bullet_{\GD(\mathscr{O}) \rtimes \Gm} (\Gr;\RG) \to \HM^\bullet_{\ID \rtimes \Gm} (\Gr;\RG)
\]
is an isomorphism. Using this, we see that to prove the existence of the factorization it suffices to prove that the natural algebra morphism
\begin{equation}
\label{eqn:morp-cohomology-Gr-2}
\OC(\tg^* / W \times \tg^*/W)[\hbar] \to \HM^\bullet_{\GD(\mathscr{O}) \rtimes \Gm} (\Gr;\RG)
\end{equation}
defined in a way similar to~\eqref{eqn:morp-cohomology-Gr} factors through $C'_\hbar$.

The latter property can be proved as follows.\footnote{A similar claim is asserted without details in~\cite{bf}. We thank V.~Ginzburg for explaining this proof to one of us.} Using the same spectral sequence argument as above, one can check that $\HM^\bullet_{\GD(\mathscr{O}) \rtimes \Gm} (\Gr;\RG)$ is $\RG[\hbar]$-free, and that the natural morphism
\[
\RG \otimes_{\RG[\hbar]} \HM^\bullet_{\GD(\mathscr{O}) \rtimes \Gm} (\Gr;\RG) \to \HM^\bullet_{\GD(\mathscr{O})} (\Gr;\RG)
\]
is an isomorphism. From these facts we see that it suffices to prove that the morphism
\[
\OC(\tg^*/W \times \tg^*/W) \to \HM^\bullet_{\GD(\mathscr{O})} (\Gr;\RG)
\]
defined as for~\eqref{eqn:morp-cohomology-Gr-2} (but with the $\Gm$-equivariance omitted) factors through $\OC(\Delta)$.

Now we make the following observation. Let $H$ be a topological group, acting on a topological space $X$, and let $Y:=H \times X$. We endow $Y$ with an action of $H \times H$ via $(h_1,h_2) \cdot (k,x) = (h_1 k h_2^{-1}, h_2 \cdot x)$. Then there exists a natural morphism 
\begin{equation}
\label{eqn:morphism-equiv-coh}
\HM^\bullet_{H \times H}(\mathrm{pt}; \RG) \to \HM^\bullet_{H \times H}(Y; \RG).
\end{equation}
We claim that~\eqref{eqn:morphism-equiv-coh} factors through the morphism $\HM^\bullet_{H \times H}(\mathrm{pt}; \RG) \to \HM^\bullet_{H}(\mathrm{pt}; \RG)$ induced by restriction to the diagonal copy of $H$. Indeed one can consider the composition
\begin{equation}
\label{eqn:morphism-equiv-coh-2}
\HM^\bullet_{H \times H}(Y; \RG) \to \HM^\bullet_{H}(Y; \RG) \to \HM^\bullet_{H}(X; \RG),
\end{equation}
where the first morphism is induced by restriction to the diagonal copy, and the second morphism by restriction to the $H$-stable subspace $X= \{1\} \times X \subset Y$. Since $Y$ identifies with the induced variety $(H \times H) \times^H X$ (via the morphism $[(h_1,h_2) : x] \mapsto (h_1 h_2^{-1}, h_2 \cdot x)$), \eqref{eqn:morphism-equiv-coh-2} is an isomorphism. Since the composition of~\eqref{eqn:morphism-equiv-coh} and~\eqref{eqn:morphism-equiv-coh-2} clearly factors through $\HM^\bullet_H(\mathrm{pt}; \RG)$, the same holds for~\eqref{eqn:morphism-equiv-coh}.

We take for $H$ a maximal compact subgroup of $\GD$, so that we have isomorphisms $\HM^\bullet_{\GD(\mathscr{O})}(\mathrm{pt}; \RG) \cong \HM^\bullet_{\GD}(\mathrm{pt}; \RG) \cong \HM^\bullet_{H}(\mathrm{pt}; \RG)$. Similarly, using the K{\"u}nneth formula (which is applicable here since our cohomology spaces are free over $\RG$) we obtain isomorphisms
\[
\OC(\tg^*/W \times \tg^*/W) \cong \HM^\bullet_{\GD(\mathscr{O}) \times \GD(\mathscr{O})}(\mathrm{pt}; \RG) \cong \HM^\bullet_{H \times H}(\mathrm{pt}; \RG),
\]
and one can identify the morphism $\OC(\tg^*/W \times \tg^*/W) \to \OC(\Delta)$ with the morphism $\HM^\bullet_{H \times H}(\mathrm{pt}; \RG) \to \HM^\bullet_{H}(\mathrm{pt}; \RG)$ considered above.

If $\Omega H$ denotes the group of polynomial loops from the unit circle to $H$, then as in~\cite[\S 1.2]{ginzburg-2} we have a natural homeomorphism $\Omega H / H \simto \Gr$. Writing $\Omega H = H \times \Omega^0 H$ (where $\Omega^0 H$ is the space of \emph{based} loops, i.e.~those sending the base point of the circle to the identity) we obtain isomorphisms
\[
\HM^\bullet_{\GD(\mathscr{O})} (\Gr;\RG) \cong \HM^\bullet_H(\Omega H / H; \RG) \cong \HM^\bullet_{H \times H}(H \times \Omega^0 H; \RG).
\]
Hence we are in the setting considered above, with $X=\Omega^0 H$, and the desired claim follows from our general observation.
\end{proof}

Using a spectral sequence argument as in the proof of Lemma~\ref{lem:morphism-cohomology-factorization}, one can check that
the natural morphism
\[
\HM^\bullet_{\ID}(\mathrm{pt};\RG) \otimes_{\HM^\bullet_{\ID \rtimes \Gm} (\mathrm{pt};\RG)} \HM^\bullet_{\ID \rtimes \Gm} (\Gr;\RG) \to \HM^\bullet_{\ID} (\Gr;\RG)
\]
is an isomorphism.
We denote by
\[
\gamma \colon C \to \HM^\bullet_{\ID} (\Gr;\RG)
\]
the composition of $\RG \otimes_{\RG[\hbar]} \gamma_\hbar$ with this isomorphism. Then $\gamma$ is a graded algebra morphism.

Since $\Gr$ is the disjoint union of its connected components $\Gr_{(\omega)}$ ($\omega \in \Omega$) which are $\ID \rtimes \Gm$-stable, there exist natural isomorphisms of graded algebras
\[
\HM^\bullet_{\ID \rtimes \Gm} (\Gr;\RG) \cong \prod_{\omega \in \Omega} \HM^\bullet_{\ID \rtimes \Gm} (\Gr_{(\omega)};\RG), \qquad \HM^\bullet_{\ID} (\Gr;\RG) \cong \prod_{\omega \in \Omega} \HM^\bullet_{\ID} (\Gr_{(\omega)};\RG).
\]
For any $\omega \in \Omega$, we denote by
\[
\gamma^{(\omega)}_\hbar \colon C_\hbar \to \HM^\bullet_{\ID \rtimes \Gm} (\Gr_{(\omega)};\RG), \qquad \text{resp.} \qquad \gamma^{(\omega)} \colon C \to \HM^\bullet_{\ID} (\Gr_{(\omega)};\RG)
\]
the composition of $\gamma_\hbar$, resp.~$\gamma$, with the projection on the factor parametrized by $\omega$.

\begin{prop}
\label{prop:cohomology-Gr}
For all $\omega \in \Omega$, the morphisms $\QM \otimes_\RG \gamma^{(\omega)}_\hbar$ and $\QM \otimes_\RG \gamma^{(\omega)}$ are isomorphisms.
\end{prop}

\begin{proof}
It is sufficient to prove the claim for $\gamma^{(\omega)}_\hbar$. Then, by construction of this morphism, it is sufficient to prove that the similar morphism
\[
C'_\hbar \to \HM^\bullet_{\GD(\mathscr{O}) \rtimes \Gm} (\Gr_{(\omega)};\RG)
\]
becomes an isomorphism after applying $\QM \otimes_\RG (-)$. However,
since $\QM$ is flat over $\RG$, by Lemma~\ref{lem:def-normal-cone-BC} we have a natural isomorphism
\[
\QM \otimes_{\RG} C'_\hbar \cong \DNC_{I_1^{\QM}}(\OC( \tg^*_\QM / W \times \tg^*_\QM / W )),
\]
where $\tg_\QM^*:=\QM \otimes_\RG \tg^*$ and $I_1^{\QM}$ is the ideal of the diagonal copy of $\tg_\QM^*/W$. Similarly, we have
\[
\QM \otimes_{\RG} \HM^\bullet_{\GD(\mathscr{O}) \rtimes \Gm} (\Gr_{(\omega)};\RG) \cong \HM^\bullet_{\GD(\mathscr{O}) \rtimes \Gm} (\Gr_{(\omega)};\QM).
\]
Hence our claim follows from~\cite[Theorem 1]{bf}.
\end{proof}

\begin{remark}
Unlike for the case of $\QM$, the morphisms $\gamma^{(\omega)}_\hbar$ and $\gamma^{(\omega)}$ are \emph{not} isomorphisms. In fact, $C_{\hbar}$ is a finitely generated $\RG$-algebra, whereas $\HM^\bullet_{\ID \rtimes \Gm}(\Gr_{(\omega)}; \RG)$ is not finitely generated in general, see~\cite{yz}.
\end{remark}

\subsection{``Topological'' Bott--Samelson category}
\label{ss:topological-BS}

Let $\EM$ be either $\RG$ or $\FM$. Recall the standard convolution product on the category $\Db_{\ID \rtimes \Gm}(\Fl, \EM)$, defined by
\[
\FC \star \GC := \mu_*(\FC \, \widetilde{\boxtimes} \, \GC),
\]
where $\Fl \, \widetilde{\times} \, \Fl$ is the quotient of $(\GD(\mathscr{K}) \rtimes \Gm) \times \Fl$ by the natural diagonal action of $\ID \rtimes \Gm$, $\mu \colon \Fl \, \widetilde{\times} \, \Fl \to \Fl$ is defined by $\mu([g : h \ID]) = gh \ID$, and $\FC \, \widetilde{\boxtimes} \, \GC$ is the ``twisted external product'' of $\FC$ and $\GC$, i.e.~the unique object whose pullback to $(\GD(\mathscr{K}) \rtimes \Gm) \times \Fl$ is the external product of the pullback of $\FC$ with $\GC$.
A similar construction provides a bifunctor
\[
(-) \star (-) \colon \Db_{\ID \rtimes \Gm}(\Fl, \EM) \times \Db_{\ID \rtimes \Gm}(\Gr, \EM) \to \Db_{\ID \rtimes \Gm}(\Gr, \EM).
\]

We now introduce some ``Bott--Samel\-son objects'' in $\Parity_{\ID \rtimes \Gm}(\Gr,\EM)$, as follows. Each connected component $\Gr_{(\omega)}$ contains a unique $0$-dimensional $\ID$-orbit; we denote by $\delta_\omega^{\EM}$ the constant (skyscraper) sheaf on this orbit (with coefficients $\EM$). On the other hand, for any simple reflection $s$, we have the $\ID \rtimes \Gm$-equivariant parity complex $\EC_{s,\EM}:=\underline{\EM}_{\overline{\Fl_s}}[1]$ on $\Fl$. Then, if $\omega \in \Omega$ and if $\us=(s_1, \cdots, s_r)$ is a sequence of simple reflections, we can consider the object
\[
\EC_\EM(\omega, \us) := \EC_{s_r,\EM} \star \cdots \star \EC_{s_1,\EM} \star \delta^\EM_{\omega^{-1}}
\]
in $\Db_{\ID \rtimes \Gm}(\Gr,\EM)$. The arguments in~\cite[\S 4.1]{jmw} or in~\cite[\S 5.5]{FieWil} show that $\EC_\EM(\omega, \us)$ belongs to the subcategory $\Parity_{\ID \rtimes \Gm}(\Gr,\EM)$.
We will denote similarly the images of this object in $\Parity_{\ID}(\Gr,\EM)$ and in $\Parity_{(\ID)}(\Gr,\EM)$. 

We define the category $\BStop$ with
\begin{itemize}
\item \emph{objects}: triples $(\omega, \underline{s}, i)$ with $\omega \in \Omega$, $\underline{s}$ a sequence of simple reflections indexed by $(1, \cdots, n)$ for some $n \in \ZM$, and $i \in \ZM$;
\item \emph{morphisms}: 
\[
\Hom_{\BStop} \bigl( (\omega, \underline{s}, i), (\omega', \underline{t}, j) \bigr) :=
\Hom_{\Parity_{\ID}(\Gr, \RG)} (\EC_\RG(\omega, \us) [i], \EC_\RG(\omega', \ut)[j]).
\]
\end{itemize}

\begin{prop}
\label{prop:parity-Karoubi}
The category $\Parity_{(\ID)}(\Gr, \FM)$ can be recovered from the category $\BStop$, in the sense that it is equivalent to the Karoubian closure of the additive envelope of the category 
which has the same objects as $\BStop$, and morphisms from $(\omega, \underline{s}, i)$ to $(\omega', \underline{t}, j)$ which are given by the $(j-i)$-th piece of the graded vector space
\[
\FM \otimes_{\HM^\bullet_{\ID}(\mathrm{pt}; \RG)} \left( \bigoplus_{n \in \ZM} \Hom_{\BStop} \bigl( (\omega, \underline{s}, 0), (\omega', \underline{t}, n) \bigr) \right).
\]
\end{prop}

\begin{proof}
Denote (for the duration of the proof) by $\mathsf{A}$ the category which has the same objects as $\BStop$, and whose morphisms are defined as in the statement of the proposition. We observe that, if $\FM(-)$ is the ``modular reduction functor'' defined as in~\S\ref{ss:reminder-parity}, then we have we have a canonical isomorphism
\[
\FM \bigl( \EC_\RG(\omega, \us) \bigr) \cong \EC_\FM(\omega, \us)
\]
for any $(\omega, \us)$ as above. (In fact, this follows from the commutation of the functor $\FM(-)$ with $*$-pullback, see~\cite[Proposition~2.6.5]{ks}, and with $!$-pushforward, see~\cite[Proposition~2.6.6]{ks}.) Using these isomorphisms and
Lemma~\ref{lem:properties-parity}\eqref{it:morph-parity-For},
we see that the assignment $(\omega, \us, i) \mapsto \EC_\FM(\omega,\us)[i]$ defines an equivalence of categories from $\mathsf{A}$ to the full subcategory $\mathsf{A}'$ of $\Parity_{(\ID)}(\Gr, \FM)$ whose objects are of the form $\EC_\FM(\omega,\us)[i]$.

Now we observe that, by the results of~\cite{jmw}, the category $\Parity_{(\ID)}(\Gr, \FM)$ is a Krull--Schmidt, Karoubian, additive category, and moreover that any indecomposable object in this category is isomorphic to a direct summand of an object of the form $\EC_\FM(\omega,\us)[i]$. It follows that $\Parity_{(\ID)}(\Gr, \FM)$ is equivalent to the Karoubian closure of the additive envelope of $\mathsf{A}'$, which finishes the proof.
\end{proof}

\begin{remark}
\label{rk:indec-parity}
More precisely, the results of~\cite[\S 4.1]{jmw}
imply that if $\lambda \in \XB$ and if $w_{-\lambda} = \omega s_1 \cdots s_r$ is a reduced expression for $w_\lambda$, then $\EC_\lambda$ can be characterized (up to isomorphism) as the unique direct summand of $\EC_\FM(\omega, (s_1, \cdots, s_r))$ which is not a direct summand of any object of the form $\EC_\FM(\omega', \ut)[i]$ where $\omega' \in \Omega$, $i \in \ZM$, and $\ut$ is a sequence of simple reflections of length at most $r-1$.
\end{remark}

\subsection{Equivariant cohomology functors}
\label{ss:cohomology-functors}

For any $\omega \in \Omega$, we define the functor
\[
\co_{\ID \rtimes \Gm} \colon \Parity_{\ID \rtimes \Gm}(\Gr, \RG) \to \Modgr(C_\hbar)
\]
as the composition
\begin{equation}
\label{eqn:functor-H}
\Parity_{\ID \rtimes \Gm}(\Gr, \RG) \xrightarrow{\HM^\bullet_{\ID \rtimes \Gm}(\Gr,-)}
\Modgr(\HM^\bullet_{\ID \rtimes \Gm}(\Gr;\RG)) \to \Modgr(C_\hbar),
\end{equation}
where the second functor is the ``restriction of scalars'' functor associated with the morphism $\gamma_\hbar$. We define similarly a functor
\[
\co_{\ID} \colon \Parity_{\ID}(\Gr, \RG) \to \Modgr(C).
\]
The goal of this subsection is to prove the following.

\begin{prop}
\label{prop:H-BS}
For any $\omega \in \Omega$ and any sequence $\us=(s_1, \cdots, s_n)$ of simple reflections, there exists a canonical isomorphism of graded $C$-modules
\[
\co_{\ID} (\EC_\RG(\omega,\us)) \cong D(\omega,\us).
\]
\end{prop}

Before proving the proposition, we remark that the same constructions as in~\S\ref{ss:cohomology-Gr} allow to define graded algebra morphisms
\[
\tC_\hbar \to \HM^\bullet_{\ID \rtimes \Gm}(\Fl; \RG) \qquad \text{and} \qquad \tC \to \HM^\bullet_{\ID}(\Fl; \RG),
\]
and then functors
\[
\tco_{\ID \rtimes \Gm} \colon \Parity_{\ID \rtimes \Gm}(\Fl, \RG) \to \Modgr(\tC_\hbar), \qquad \tco_{\ID} \colon \Parity_{\ID}(\Fl, \RG) \to \Modgr(\tC).
\]

\begin{lem}
\label{lem:cohomology-basic}
\begin{enumerate}
\item
\label{it:cohomology-delta}
For any $\omega \in \Omega$, there exists a canonical isomorphism of graded $C_\hbar$-modules
\[
\co_{\ID \rtimes \Gm}(\delta_\omega^\RG) \cong E^\hbar_{\omega^{-1}}.
\]
\item
\label{it:cohomology-Es}
For any simple reflection $s$, there exists a canonical isomorphism of $\tC_\hbar$-modules
\[
\tco_{\ID \rtimes \Gm}(\EC_{s,\RG}) \cong D^\hbar_s.
\]
\end{enumerate}
\end{lem}

\begin{proof}
\eqref{it:cohomology-delta}
By definition, if $\omega = t_\lambda v$ (with $v \in W$ and $\lambda \in \XB$) we have
\[
\co_{\ID \rtimes \Gm}(\delta_\omega^\RG) = \HM^\bullet_{\ID \rtimes \Gm}(\Gr, \delta_\omega^\RG) = \HM^\bullet_{\ID \rtimes \Gm} \bigl( z^\lambda \cdot (\GD(\mathscr{O}) \rtimes \Gm) / \GD(\mathscr{O}) \rtimes \Gm; \RG \bigr).
\]
We deduce a canonical isomorphism
\[
\co_{\ID \rtimes \Gm}(\delta_\omega^\RG) \cong \HM^\bullet_{(\ID \rtimes \Gm) \times (\GD(\mathscr{O}) \rtimes \Gm)} \bigl( z^\lambda \cdot (\GD(\mathscr{O}) \rtimes \Gm) ; \RG \bigr).
\]
Now if $\dot{v}$ is a lift of $v$ in $\GD$, the assignment $(a,b) \mapsto a \cdot z^\lambda \dot{v} \cdot b^{-1}$ induces an isomorphism
\[
(\ID \rtimes \Gm) \times (\GD(\mathscr{O}) \rtimes \Gm) / K \simto z^\lambda \cdot (\GD(\mathscr{O}) \rtimes \Gm),
\]
where $K=\{(a,b) \in (\ID \rtimes \Gm) \times (\GD(\mathscr{O}) \rtimes \Gm) \mid b=(z^\lambda \dot{v})^{-1} a (z^\lambda \dot{v})\}$. This group is isomorphic to $(\ID \rtimes \Gm)$ (through $a \mapsto \bigl(a, (z^\lambda \dot{v})^{-1} a (z^\lambda \dot{v}) \bigr)$); hence we obtain canonical isomorphisms
\[
\co_{\ID \rtimes \Gm}(\delta_\omega^\RG) \cong \HM^\bullet_K(\mathrm{pt}; \RG) \cong \HM^\bullet_{\ID \rtimes \Gm}(\mathrm{pt}; \RG).
\]
The right-hand side is isomorphic to $\OC(\tg^*)[\hbar]$ (see~\eqref{eqn:cohomology-TD}), with the natural action of the subalgebra $\OC(\tg^*)[\hbar] \subset C_\hbar$. And, via this isomorphism, the subalgebra $\OC(\tg^*/W)[\hbar] \subset C_\hbar$ acts via the composition of multiplication and the morphism
\[
\OC(\tg^*/W)[\hbar] \hookrightarrow \OC(\tg^*)[\hbar] \xrightarrow{f \mapsto \omega \cdot f} \OC(\tg^*)[\hbar].
\]
Hence we have constructed a canonical isomorphism of $\OC(\tg^*/W \times \tg^*)[\hbar]$-modules
\[
\co_{\ID \rtimes \Gm}(\delta_\omega^\RG) \cong E^\hbar_{\omega^{-1}}.
\]
Since these spaces have no $\hbar$-torsion, this isomorphism is automatically an isomorphism of $C_\hbar$-modules.

\eqref{it:cohomology-Es}
Using Lemma~\ref{lem:affine-simple-conjugate}, one can assume that $s$ is finite, with associated simple root $\alpha^\vee \in {\check \XB}$ and associated simple coroot $\alpha \in \XB$. 
By definition we have
\[
\tco_{\ID \rtimes \Gm}(\EC_{s,\RG}) \cong \HM^{\bullet}_{\TD \times \Gm}(\overline{\Fl_s}; \RG) \langle -1 \rangle.
\]
Since $\overline{\Fl_s}$ is isomorphic to $\PM^1_\CM$,
it is well known (see e.g.~\cite{FieWil}) that the morphism
\[
\HM^{\bullet}_{\TD \times \Gm}(\overline{\Fl_s}; \RG) \to \HM^{\bullet}_{\TD \times \Gm}(\mathrm{pt}; \RG) \oplus \HM^{\bullet}_{\TD \times \Gm}(\mathrm{pt}; \RG)
\]
induced by restriction to the fixed points $\ID / \ID$ and $s \cdot \ID / \ID$ induces an isomorphism between $\HM^{\bullet}_{\TD \times \Gm}(\overline{\Fl_s}; \RG)$ and
\[
{}' \hspace{-1pt} D^{\hbar}_s:=\{(a,b) \in \OC(\tg^*)[\hbar] \oplus \OC(\tg^*)[\hbar] \mid a=b \ \mathrm{mod} \ \alpha^\vee\}.
\]
(Here, the right copy of $\OC(\tg^*)[\hbar]$ in $\tC_\hbar$ acts diagonally, while any $f$ in the left copy of $\OC(\tg^*)[\hbar]$ in $\tC_\hbar$ acts by multiplication by $(f,s(f))$.)
Now we consider the morphism $D^\hbar_s \to {}' \hspace{-1pt} D^\hbar_s$ defined by $f \otimes g \mapsto (fg, s(f)g)$. By our assumptions on $\RG$, there exists ${\check \mu} \in {\check \XB} \otimes_\ZM \RG$ such that $\langle {\check \mu}, \alpha \rangle = 1$.
Then our morphism sends the basis of $D^\hbar_s$ as an $\OC(\tg^*)[\hbar]$-module (via the action of the right copy of $\OC(\tg^*)[\hbar]$ in $\tC_\hbar$) consisting of $1 \otimes 1$ and ${\check \mu} \otimes 1$ to the basis of ${}' \hspace{-1pt} D^{\hbar}_s$ consisting of $(1,1)$ and $({\check \mu}, {\check \mu}-\alpha^\vee)$; hence it is an isomorphism.

We have constructed an isomorphism of $\OC(\tg^* \times \tg^*)[\hbar]$-modules
\[
\tco_{\ID \rtimes \Gm}(\EC_{s,\RG}) \cong D_s^\hbar.
\]
Since both sides have no $\hbar$-torsion, this isomorphism is automatically an isomorphism of $\tC_\hbar$-modules.
\end{proof}

\begin{proof}[Proof of Proposition~{\rm \ref{prop:H-BS}}]
The arguments in~\cite[\S 4.1]{jmw} imply that for any $\FC$ in $\Parity_{\ID \rtimes \Gm}(\Gr, \RG)$ and any sequence $\ut=(t_1, \cdots, t_r)$ of simple reflections, the convolution 
\[
\EC_{t_1, \RG} \star \cdots \star \EC_{t_r, \RG} \star \FC
\]
belongs to $\Parity_{\ID \rtimes \Gm}(\Gr, \RG)$. Then the same arguments as in the proof of~\cite[Proposition~3.2.1]{by} (using Lemma~\ref{lem:properties-parity}\eqref{it:parity-cohomology} instead of~\cite[Corollary~B.4.2]{by}) imply that there exists a canonical isomorphism of graded modules over $\HM^\bullet_{\ID \rtimes \Gm}(\pt; \RG) \otimes_{\RG[\hbar]} \HM^\bullet_{\GD(\mathscr{O}) \rtimes \Gm}(\pt; \RG)$:
\begin{multline*}
\co_{\ID \rtimes \Gm}(\EC_{t_1, \RG} \star \cdots \star \EC_{t_r, \RG} \star \FC) \cong \\
\tco_{\ID \rtimes \Gm}(\EC_{t_1,\RG}) \otimes_{\HM^\bullet_{\ID \rtimes \Gm}(\pt; \RG)} \cdots \otimes_{\HM^\bullet_{\ID \rtimes \Gm}(\pt; \RG)} \tco_{\ID \rtimes \Gm}(\EC_{t_r,\RG}) \otimes_{\HM^\bullet_{\ID \rtimes \Gm}(\pt; \RG)} \co_{\ID \rtimes \Gm}(\FC).
\end{multline*}
Applying this remark to $\FC=\delta_{\omega^{-1}}^\RG$ and $\ut=(s_n, \cdots, s_1)$, and using Lemma~\ref{lem:cohomology-basic}, we obtain a canonical isomorphism of $\OC(\tg^*/W \times \tg^*)[\hbar]$-modules
\[
\co_{\ID \rtimes \Gm} (\EC_\RG(\omega,\us)) \cong D_\hbar(\omega,\us).
\]
Since both sides have no $\hbar$-torsion, this isomorphism is automatically an isomorphism of $C_\hbar$-modules. Specializing $\hbar$ to $0$, we deduce the isomorphism of the proposition.
\end{proof}

\subsection{Equivalence}
\label{ss:equivalence}

The proof of the following proposition (which is independent of the rest of the section) is postponed to \S\ref{ss:proof-H-ff}.

\begin{prop}
\label{prop:global-coh-ff}
For any $\omega, \omega' \in \Omega$, any sequences $\us,\ut$ of simple reflections, and any $n \in \ZM$, the morphism
\begin{multline*}
\Hom_{\Parity_{\ID}(\Gr, \RG)}(\EC_\RG(\omega, \us), \EC_\RG(\omega',\ut)[n]) \\
\to \Hom_{\Modgr ( \HM^\bullet_{\ID}(\Gr;\RG))}(\HM^\bullet_{\ID}(\Gr, \EC_\RG(\omega, \us)), \HM^\bullet_{\ID}(\Gr, \EC_\RG(\omega',\ut)[n]))
\end{multline*}
induced by the functor $\HM^\bullet_{\ID}(\Gr,-)$ is an isomorphism.
\end{prop}

Now we define a functor
\[
\co_{\mathsf{BS}} \colon \BStop \to \BSalg
\]
as follows. This functor sends an object $(\omega, \us, i)$ of $\BStop$ to the corresponding object $(\omega, \us, i)$ of $\BSalg$. Then if $(\omega, \us, i)$ and $(\omega', \ut, j)$ are objects of $\BStop$, the morphism
\[
\Hom_{\BStop} \bigl((\omega, \us, i), (\omega', \ut, j) \bigr) \to \Hom_{\BSalg} \bigl((\omega, \us, i), (\omega', \ut, j) \bigr)
\]
is defined as the trivial morphism if $\omega \neq \omega'$ (in which case both $\Hom$-spaces are $0$), and as the morphism induced by $\co_{\ID}$ if $\omega=\omega'$, using the canonical isomorphisms
\[
\co_{\ID}(\EC_\RG(\omega,\us)[i]) \cong D(\omega,\us)\langle -i \rangle, \qquad \co_{\ID}(\EC_\RG(\omega',\ut)[j]) \cong D(\omega',\ut)\langle -j \rangle
\]
deduced from Proposition~\ref{prop:H-BS}.

The main result of this section is the following.

\begin{thm}
\label{thm:main-top}
The functor $\co_{\mathsf{BS}}$ is an equivalence of categories.
\end{thm}

\begin{proof}
The functor $\co_{\mathsf{BS}}$ clearly induces a bijection on objects. So, what we have to prove is that if $(\omega, \us, i)$ and $(\omega', \ut, j)$ are objects of $\BStop$, then the corresponding morphism
\[
\Hom_{\BStop} \bigl((\omega, \us, i), (\omega', \ut, j) \bigr) \to \Hom_{\BSalg} \bigl((\omega, \us, i), (\omega', \ut, j) \bigr)
\]
is an isomorphism. This is obvious if $\omega \neq \omega'$.

Now, assume that $\omega=\omega'$. Then both parity complexes are supported on $\Gr_{(\omega)}$, so that our morphism is induced by the composition
\[
\Parity_{\ID}(\Gr_{(\omega)}, \RG) \xrightarrow{\HM^\bullet_{\ID}(\Gr_{(\omega)},-)} \Modgr \bigl( \HM^\bullet_{\ID}(\Gr_{(\omega)}; \RG) \bigr) \to \Modgr(C),
\]
where the second arrow is the ``restriction of scalars'' functor associated with $\gamma^{(\omega)}$. Now the first functor is fully faithful on objects of the form $\EC_\RG(\omega, \us)$ by Proposition~\ref{prop:global-coh-ff}, and the second functor is fully faithful on objects which are $\RG$-free by Lemma~\ref{lem:key} and Proposition~\ref{prop:cohomology-Gr}. Since the $\RG$-modules $\HM^\bullet_{\ID}(\Gr_{(\omega)},\EC_\RG(\omega, \us))$ are free by Lemma~\ref{lem:properties-parity}\eqref{it:parity-cohomology}, this finishes the proof of the theorem.
\end{proof}

\subsection{Proof of Proposition~\ref{prop:global-coh-ff}}
\label{ss:proof-H-ff}

The proof below is a simple variant of the main result of~\cite{ginzburg}. The key observation is from~\cite[Theorem~4.1]{arider}, where it is shown that the arguments from~\cite{ginzburg} involving weights can be replaced in a parity sheaf setting by parity arguments.  Another exposition of Ginzburg's proof in an equivariant setting appears as~\cite[Lemma~3.3.1]{by} (for coefficients in characteristic zero).

To avoid unnecessary notational complications, in this proof we will say that an object $\EC$ in $\Parity_{\ID}(\Gr, \RG)$ is a \emph{Bott--Samelson parity complex} if it is isomorphic to $\EC_\RG(\omega, \us)[i]$ for some $\omega \in \Omega$, $\us$ a sequence of simple reflections, and $i \in \ZM$.

Let us fix an extension of the partial order on $\XB$ corresponding to the closure relations among the orbits $\Gr_\lambda$ to a total ordering $\sqsubseteq$ on $\XB$, such that $(\XB, \sqsubseteq)$ is isomorphic (as an ordered set) to $\ZM_{\geq 0}$ with its standard order.  Let $\Gr_{\sqsubseteq \lambda}$ denote the union of all $\Gr_\mu$ where $\mu \sqsubseteq \lambda$, and $i_{\sqsubseteq \lambda}:\Gr_{\sqsubseteq \lambda} \to \Gr$ the closed embedding. Let also $A_{\sqsubseteq \lambda} := \HM_{\ID}^\bullet(\Gr_{\sqsubseteq \lambda};\RG)$.  We define analogously $\Gr_{\sqsubset \lambda}$, $i_{\sqsubset \lambda}$, and $A_{\sqsubset \lambda}$.

We begin with a number of preliminary lemmas.

\begin{lem}
\label{lem:cohom-lambda}
For any $\lambda \in \XB$, there exists a canonical exact sequence of $\ZM$-graded $\RG$-modules
\[
0 \to \HM^\bullet_{c,\ID}(\Gr_\lambda; \RG) \to A_{\sqsubseteq \lambda} \to A_{\sqsubset \lambda} \to 0.
\]
\end{lem}

\begin{proof}
The lemma can be proved by induction, using the fact that the existence of the exact sequences for smaller $\lambda$'s implies that $A_{\sqsubset \lambda}$ is concentrated in even degrees, and the adjunction triangle
\[
(i_{\lambda})_! (i_\lambda)^* \underline{\RG}_{\Gr_{\sqsubseteq \lambda}} \to \underline{\RG}_{\Gr_{\sqsubseteq \lambda}} \to i_{\sqsubset \lambda *} i_{\sqsubset \lambda}^* \underline{\RG}_{\Gr_{\sqsubseteq \lambda}} \xrightarrow{[1]},
\]
where $i_\lambda \colon \Gr_\lambda \hookrightarrow \Gr$ is the inclusion.
\end{proof}

Let again $\lambda \in \XB$.  To simplify notation, we set $Z= \Gr_{\sqsubset \lambda}$, $X = \Gr_{\sqsubseteq \lambda}$, $U= \Gr_\lambda$, and denote by $i:Z\to X$, resp.~$j:U \to X$, the closed, resp.~open, embedding.
A key ingredient we will need is the following lemma.

\begin{lem}
\label{lem:exact-sequences-adj}
Let $\EC$ be a Bott--Samelson parity complex, and let $\FC$ be either $i^*_{\sqsubseteq \lambda} \EC$ or $i^!_{\sqsubseteq \lambda} \EC$.  Then the adjunction triangles induce short exact sequence of $A_{\sqsubseteq \lambda}$-modules
\begin{gather}
0 \to \HM_{\ID}^\bullet(X, j_!j^* \FC) \to \HM_{\ID}^\bullet(X,\FC) \to \HM_{\ID}^\bullet(X, i_*i^*\FC) \to 0, \label{eqn:adjses1}\\
0 \to \HM_{\ID}^\bullet(X, i_*i^! \FC) \to \HM_{\ID}^\bullet(X,\FC) \to \HM_{\ID}^\bullet(X, j_*j^*\FC) \to 0 \label{eqn:adjses2}.
\end{gather}
\end{lem}

\begin{proof}
For $\FC = i^*_{\sqsubseteq \lambda} \EC$, resp.~$i^!_{\sqsubseteq \lambda} \EC$, $\FC$ is either $*$-even or $*$-odd, resp.~either $!$-even or $!$-odd, and~\eqref{eqn:adjses1}, resp.~\eqref{eqn:adjses2}, is obtained using the long exact sequence of equivariant cohomology associated with the adjunction triangle $j_! j^* \FC \to \FC \to i_*i^* \FC \xrightarrow{[1]}$, resp.~$i_* i^! \FC \to \FC \to j_*j^* \FC \xrightarrow{[1]}$, because by induction the various terms are concentrated either all in even degrees, or all in odd degrees.

Consider now the sequence~\eqref{eqn:adjses2} for $\FC = i^*_{\sqsubseteq \lambda} \EC$. In this case it is shown in~\cite[Proposition~5.9]{FieWil}\footnote{In~\cite{FieWil}, this result (as well as the other results used in this proof) is stated for a coefficient ring that is a complete local principal ideal domain, but the same proof applies verbatim for general coefficients.} that the composition of restriction morphisms
\[
\HM^\bullet_{\ID}(\Gr, \EC) \to \HM^\bullet_{\ID}(\Gr_\lambda, i_\lambda^* \EC) \to \HM^\bullet_{\ID}(\{L_\lambda\}, \imath_\lambda^* \EC)
\]
is surjective, where (as in~\S\ref{ss:intro-other}) $\imath_\lambda \colon \{L_\lambda\} \hookrightarrow \Gr$ is the inclusion. Since the second morphism is an isomorphism by~\cite[Proposition~2.3]{FieWil}, the first one is also surjective. And since this first morphism factors through the morphism
\[
\HM^\bullet_{\ID}(\Gr_{\sqsubseteq \lambda}, i^*_{\sqsubseteq \lambda} \EC) \to \HM^\bullet_{\ID}(\Gr_\lambda, i_\lambda^* \EC)
\]
considered in~\eqref{eqn:adjses2}, the latter morphism is also surjective, which finishes the proof in this case.

Finally, let us consider the sequence~\eqref{eqn:adjses1} for $\FC=i^!_{\sqsubseteq \lambda} \EC$. In this case, the arguments are similar to the ones used in the preceding case, using the fact that the natural morphism
\[
\HM^\bullet_{\ID}(\{L_\lambda\}, \imath_\lambda^! \EC) \to \HM^\bullet_{\ID}(\Gr, \EC)
\]
is injective. (This fact is shown in the proof of~\cite[Proposition 5.9]{FieWil}: in fact this morphism identifies with the morphism considered in~\cite[Proposition 5.8(2)]{FieWil}.)
\end{proof}

Since $U=\Gr_\lambda$ is isomorphic to an affine space, the graded $\HM^\bullet_{\ID}(\mathrm{pt}; \RG)$-module $\HM^\bullet_{c,\ID}(U; \RG)$ is free of rank one. Moreover, the natural morphism 
\[
\HM^{2 \dim(\Gr_\lambda)}_{c,\ID}(U; \RG) \to \HM^{2 \dim(\Gr_\lambda)}_{c}(U; \RG)
\]
is an isomorphism, and if $x_\lambda \in \HM^{2 \dim(\Gr_\lambda)}_{c,\ID}(U; \RG)$ is the inverse image of a generator of $\HM^{2 \dim(\Gr_\lambda)}_{c}(U; \RG) \cong \RG$ (as an $\RG$-module), then $x_\lambda$ gives a basis of $\HM^\bullet_{c,\ID}(U; \RG)$ over $\HM^\bullet_{\ID}(\mathrm{pt}; \RG)$. We still denote by $x_\lambda$ the image of this element in $A_{\sqsubseteq \lambda}$, under the injection of Lemma~\ref{lem:cohom-lambda}.

\begin{lem}
\label{lem:action-xlambda}
Let $\EC$ be a Bott--Samelson parity complex, and let $\FC$ be either $i^*_{\sqsubseteq \lambda} \EC$ or $i^!_{\sqsubseteq \lambda} \EC$. Then the morphism
\[
\HM^\bullet_{\ID}(X, \FC) \to \HM^{\bullet+2 \dim(\Gr_\lambda)}_{\ID}(X,\FC)
\]
given by the action of $x_\lambda \in A_{\sqsubseteq \lambda}$ factors as a composition
\[
\HM_{\ID}^\bullet(X, \FC) \onto \HM^\bullet_{\ID}(X, j_*j^* \FC) \xrightarrow{\sim} \HM_{\ID}^{\bullet + 2\dim(\Gr_\lambda)}(X, j_! j^* \FC) \into \HM^{\bullet + 2\dim(\Gr_\lambda)}_{\ID}(X, \FC),
\]
where the first, resp.~third, morphism is the morphism appearing in~\eqref{eqn:adjses2}, resp. \eqref{eqn:adjses1}, and the second one is an isomorphism.
\end{lem}

\begin{proof}
The element
$x_\lambda \in A_{\sqsubseteq \lambda}$ acts trivially on $\HM_{\ID}^\bullet(X, i_*i^*\FC)$ and $\HM_{\ID}^\bullet(X, i_*i^!\FC)$. Hence its action on the module $\HM^\bullet_{\ID}(X, \FC)$ factors through a morphism
\[
\HM^\bullet_{\ID}(X, j_*j^* \FC) \to \HM_{\ID}^{\bullet + 2\dim(\Gr_\lambda)}(X, j_! j^* \FC).
\]
To show that the latter morphism is an isomorphism, we observe that $j^* \FC$ is a direct sum of shifts of constant sheaves $\underline{\RG}_U$ (since $\EC$ is a parity complex), and that our morphism is the corresponding direct sum of shifts of the isomorphism
\[
\HM^\bullet_{\ID}(U; \RG) \simto \HM^\bullet_{\ID}(\mathrm{pt}; \RG) \simto \HM^{\bullet+2\dim(\Gr_\lambda)}_{c,\ID}(U;\RG)
\]
determined by our choice of $x_\lambda \in \HM^{2 \dim(\Gr_\lambda)}_{c,\ID}(U; \RG)$.
\end{proof}

\begin{lem}
\label{lem:morph-exact-sequence}
Let $\EC_1, \EC_2$ be Bott--Samelson parity complexes. If $\phi \colon \HM^\bullet_{\ID}(X, i_{\sqsubseteq \lambda}^* \EC_1) \to \HM^\bullet_{\ID}(X, i_{\sqsubseteq \lambda}^! \EC_2)$ is a morphism of $A_{\sqsubseteq \lambda}$-modules, then the composition
\begin{equation}
\label{eqn:composition-phi}
\HM^\bullet_{\ID}(X, i_{\sqsubseteq \lambda}^* \EC_1) \xrightarrow{\phi} \HM^\bullet_{\ID}(X, i_{\sqsubseteq \lambda}^! \EC_2) \twoheadrightarrow \HM^\bullet_{\ID}(X, j_* j^* i_{\sqsubseteq \lambda}^! \EC_2)
\end{equation}
(where the second morphism is the surjection appearing in~\eqref{eqn:adjses2}) factors uniquely through a morphism of $A_{\sqsubseteq \lambda}$-modules
\[
\phi' \colon \HM^\bullet_{\ID}(X, j_* j^* i_{\sqsubseteq \lambda}^* \EC_1) \to \HM^\bullet_{\ID}(X, j_* j^* i_{\sqsubseteq \lambda}^! \EC_2).
\]
Moreover, we have $\phi'=0$ iff $\phi$ factors as a composition
\[
\HM^\bullet_{\ID}(X, i_{\sqsubseteq \lambda}^* \EC_1) \twoheadrightarrow \HM^\bullet_{\ID}(Z, i^*_{\sqsubset \lambda} \EC_1) \xrightarrow{\phi''} \HM^\bullet_{\ID}(Z, i^!_{\sqsubset\lambda}\EC_2) \hookrightarrow \HM^\bullet_{\ID}(X, i_{\sqsubseteq \lambda}^! \EC_2)
\]
where $\phi''$ is a morphism of $A_{\sqsubset \lambda}$-modules and the other morphisms are the ones appearing in~\eqref{eqn:adjses1} and~\eqref{eqn:adjses2}.
\end{lem}

\begin{proof}
Using Lemma~\ref{lem:exact-sequences-adj}, we see that the first claim is equivalent to the statement that the composition
\[
\HM^\bullet_{\ID}(X, i_* i^! i_{\sqsubseteq \lambda}^* \EC_1) \hookrightarrow
\HM^\bullet_{\ID}(X, i_{\sqsubseteq \lambda}^* \EC_1) \xrightarrow{\phi} \HM^\bullet_{\ID}(X, i_{\sqsubseteq \lambda}^! \EC_2) \twoheadrightarrow \HM^\bullet_{\ID}(X, j_* j^* i_{\sqsubseteq \lambda}^! \EC_2)
\]
vanishes. And using Lemma~\ref{lem:action-xlambda}, to prove this it suffices to prove that the composition
\[
\HM^\bullet_{\ID}(X, i_* i^! i_{\sqsubseteq \lambda}^* \EC_1) \hookrightarrow
\HM^\bullet_{\ID}(X, i_{\sqsubseteq \lambda}^* \EC_1) \xrightarrow{\phi} \HM^\bullet_{\ID}(X, i_{\sqsubseteq \lambda}^! \EC_2) \xrightarrow{x_\lambda} \HM^\bullet_{\ID}(X, i_{\sqsubseteq \lambda}^! \EC_2)
\]
vanishes. This follows from the fact that $\phi$ commutes with the action of $x_\lambda$, and that $x_\lambda$ acts trivially on $\HM^\bullet_{\ID}(X, i_* i^! i_{\sqsubseteq \lambda}^* \EC_1)$.

Now we consider the second claim. The ``if'' direction is easy. Conversely, assume that $\phi'=0$. Then the composition~\eqref{eqn:composition-phi} vanishes, hence the image of $\phi$ is included in the image of the embedding $\HM^\bullet_{\ID}(Z, i^!_{\sqsubset\lambda}\EC_2) \hookrightarrow \HM^\bullet_{\ID}(X, i_{\sqsubseteq \lambda}^! \EC_2)$ of Lemma~\ref{lem:exact-sequences-adj}. On the other hand, using Lemma~\ref{lem:action-xlambda}, this also implies that the composition of $\phi$ with the action of $x_\lambda$ on $\HM^\bullet_{\ID}(X, i_{\sqsubseteq \lambda}^! \EC_2)$ vanishes hence, since $\phi$ is a morphism of $A_{\sqsubseteq \lambda}$-modules, that $\phi$ vanishes on the image of the action of $x_\lambda$ on $\HM^\bullet_{\ID}(X, i_{\sqsubseteq \lambda}^* \EC_1)$. By Lemma~\ref{lem:action-xlambda}, we deduce that the composition
\[
\HM^\bullet(X, j_! j^* i_{\sqsubseteq \lambda}^* \EC_1) \hookrightarrow \HM^\bullet_{\ID}(X, i_{\sqsubseteq \lambda}^* \EC_1) \xrightarrow{\phi} \HM^\bullet_{\ID}(X, i_{\sqsubseteq \lambda}^! \EC_2)
\]
vanishes which, in view of Lemma~\ref{lem:exact-sequences-adj}, proves the existence of a morphism of $A_{\sqsubseteq \lambda}$-modules $\phi'' \colon \HM^\bullet_{\ID}(X, i_* i^*_{\sqsubset \lambda} \EC_1) \to \HM^\bullet_{\ID}(X, i_* i^!_{\sqsubset\lambda}\EC_2)$ as in the lemma. Finally, the fact that $\phi''$ is a morphism of $A_{\sqsubset \lambda}$-modules follows from Lemma~\ref{lem:cohom-lambda}.
\end{proof}

Now we are ready to prove Proposition~\ref{prop:global-coh-ff}. We will proceed by induction
on $\lambda \in \XB$, showing that for any Bott--Samelson parity complexes $\EC_1,\EC_2$, the cohomology functor induces an isomorphism of graded vector spaces
\begin{equation}
\Hom^\bullet_{\Gr_{\sqsubseteq \lambda}} \bigl( i^*_{\sqsubseteq \lambda}\EC_1,i^!_{\sqsubseteq \lambda}\EC_2 \bigr) \simto \Hom_{A_{\sqsubseteq \lambda}} \bigl( \HM^\bullet_{\ID}(\Gr_{\sqsubseteq \lambda}, i^*_{\sqsubseteq \lambda}\EC_1),\HM_{\ID}^\bullet(\Gr_{\sqsubseteq \lambda}, i^!_{\sqsubseteq \lambda}\EC_2) \bigr). \label{eqn:ff}
\end{equation}

If $\lambda$ is minimal, then $\Gr_\lambda$ is a point and $i^*_{\sqsubseteq \lambda}\EC_1$ and $i^!_{\sqsubseteq \lambda}\EC_2$ are also parity.  On the other hand, the cohomology functor induces an equivalence of categories between $\Parity_{\ID}(\Gr_\lambda, \RG)$ and the full subcategory of the category of finitely generated graded $A_{\sqsubseteq \lambda}$-modules consisting of free modules. Thus \eqref{eqn:ff} is indeed an isomorphism in this case.

Now fix $\lambda \in \XB$, and suppose that~\eqref{eqn:ff} is an isomorphism for all $\lambda' \sqsubset \lambda$.  
We use the same notation as above for $X$, $Z$, $U$, $i$ and $j$.
Note that Lemma~\ref{lem:morph-exact-sequence} is equivalent to the existence of a natural sequence
\begin{equation}
\label{eqn:exact-sequence-Hom}
\begin{split}
& \Hom_{A_{\sqsubset \lambda}}(\HM^\bullet_{\ID}(Z, i^*_{\sqsubset\lambda} \EC_1),\HM^\bullet_{\ID}(Z, i^!_{\sqsubset\lambda}\EC_2)) \to \\
& \qquad \qquad \Hom_{A_{\sqsubseteq \lambda}}(\HM_{\ID}^\bullet(X, i^*_{\sqsubseteq\lambda} \EC_1),\HM^\bullet_{\ID}(X, i^!_{\sqsubseteq\lambda}\EC_2)) \to \\
& \qquad \qquad \qquad \qquad \Hom_{\HM_{\ID}^\bullet(U; \RG)}(\HM^\bullet_{\ID}(U, j^*i^*_{\sqsubseteq\lambda} \EC_1),\HM^\bullet_{\ID}(U,j^*i^!_{\sqsubseteq\lambda}\EC_2))
\end{split}
\end{equation}
which is exact at the middle term. It is also easy to check (using Lemma~\ref{lem:exact-sequences-adj}) that the first morphism is injective.

Consider the adjunction triangle 
\[
i_{\sqsubset \lambda *} i_{\sqsubset \lambda}^! \EC_2 \to i_{\sqsubseteq \lambda}^! \EC_2 \to j_* j^* i_{\sqsubseteq \lambda}^! \EC_2 \xrightarrow{[1]}
\]
and the long exact sequence
\begin{multline*}
\cdots \to \Hom^n_Z(i^*_{\sqsubset\lambda} \EC_1,i^!_{\sqsubset\lambda}\EC_2) \to \Hom^n_X(i^*_{\sqsubseteq\lambda} \EC_1,i^!_{\sqsubseteq\lambda}\EC_2) \\
\to \Hom^n_U(j^*i^*_{\sqsubseteq\lambda} \EC_1,j^*i^!_{\sqsubseteq\lambda}\EC_2) \to \cdots
\end{multline*}
obtained by appying the functor $\Hom_X(i^*_{\sqsubseteq\lambda} \EC_1,-)$. Parity considerations and~\cite[Corollary~2.8]{jmw} imply that the connecting morphisms in this long exact sequence are trivial, so that
the maps form short exact sequences in each degree.
Then one can consider the commutative diagram
\begin{equation*}
\xymatrix@R=0.5cm@C=1.2cm{
\Hom_Z^\bullet(i^*_{\sqsubset\lambda} \EC_1,i^!_{\sqsubset\lambda}\EC_2) \ar[r]^-{\HM^\bullet_{\ID}(Z,-)} \ar@{^{(}->}[d] & \Hom_{A_{\sqsubset \lambda}}(\HM^\bullet_{\ID}(Z, i^*_{\sqsubset\lambda} \EC_1),\HM^\bullet_{\ID}(Z, i^!_{\sqsubset\lambda}\EC_2)) \ar@{^{(}->}[d]  \\
\Ext_X^\bullet(i^*_{\sqsubseteq\lambda} \EC_1,i^!_{\sqsubseteq\lambda}\EC_2) \ar[r]^-{\HM^\bullet_{\ID}(X,-)} \ar@{->>}[d] & \Hom_{A_{\sqsubseteq \lambda}}(\HM_{\ID}^\bullet(X, i^*_{\sqsubseteq\lambda} \EC_1),\HM^\bullet_{\ID}(X, i^!_{\sqsubseteq\lambda}\EC_2)) \ar[d]\\
\Ext_U^\bullet(j^*i^*_{\sqsubseteq\lambda} \EC_1,j^*i^!_{\sqsubseteq\lambda}\EC_2) \ar[r]^-{\HM^\bullet_{\ID}(U,-)} & \Hom_{\HM_{\ID}^\bullet(U; \RG)}(\HM^\bullet_{\ID}(U, j^*i^*_{\sqsubseteq\lambda} \EC_1),\HM^\bullet_{\ID}(U,j^*i^!_{\sqsubseteq\lambda}\EC_2)),
}
\end{equation*}
where the right-hand column is the sequence~\eqref{eqn:exact-sequence-Hom}, and both column are exact at the middle term. By induction, the upper horizontal arrow is an isomorphism.
The lower one is an isomorphism because the cohomology functor induces an equivalence of categories between $\Parity_{\ID}(\Gr_\lambda,\RG)$ and the full subcategory of free modules in the category of finitely generated graded $\HM^\bullet_{\ID}(\Gr_\lambda; \RG)$-modules. By the five-lemma, this implies that the middle horizontal arrow is also an isomorphism, which completes the induction step.

\begin{remark}
It can be easily checked that the proof of Proposition~\ref{prop:global-coh-ff} applies to any ring $k$ of coefficients which is Noetherian and of finite global dimension, and also in the $\ID \rtimes \Gm$-equivariant setting, or for other partial flag varieties of Kac--Moody groups. If $k$ is a complete local principal ideal domain, one can also work directly with the categories $\Parity_{\ID \rtimes \Gm}(\Gr, k)$ and $\Parity_{\ID}(\Gr, k)$ instead of restricting to ``Bott--Samelson objects.''
\end{remark}

\subsection{Graded ranks of $\mathrm{Hom}$ spaces}
\label{ss:graded-ranks-parity}

We identify the $\ZM[\vv,\vv^{-1}]$-module $\MM_{\Fl}$ associated with $\Fl$ and its stratification by $\ID$-orbits as in~\S\ref{ss:reminder-parity} with $\Haff$, where $\mathbf{e}_w$ corresponds to $T_{w^{-1}}$ (for $w \in \Waff$). We also identify $\MM_{\Gr}$ (where $\Gr$ is stratified by $\ID$-orbits) with $\Msph$, where $\mathbf{e}_\lambda$
corresponds to $\mm_{-\lambda}$ (for $\lambda \in \XB$).

\begin{lem}
\label{lem:ch-parity-Gr}
Let $\EM$ be either $\FM$ or $\RG$.

\begin{enumerate}
\item
\label{it:ch-parity-Fl}
For any $\FC$, $\GC$ in $\Parity_{\ID}(\Fl,\EM)$, we have
\[
\ch_{\Fl}^!(\FC \star \GC) = \ch_{\Fl}^!(\GC) \cdot \ch_{\Fl}^!(\FC) \quad \text{and} \quad \ch_{\Fl}^*(\FC \star \GC) = \ch_{\Fl}^*(\GC) \cdot \ch_{\Fl}^*(\FC)
\]
in $\Haff$.
\item
\label{it:ch-parity-Gr}
For any 
$\GC$ in $\Parity_{\ID}(\Fl,\EM)$, we have
\[
\ch_{\Gr}^!(\GC \star \delta^\EM_1) = \mm_0 \cdot \ch_{\Fl}^!(\GC) \quad \text{and} \quad \ch_{\Gr}^*(\GC \star \delta^\EM_1) = \mm_0 \cdot \ch_{\Fl}^*(\GC)
\]
in $\Msph$.
\end{enumerate}
\end{lem}

\begin{proof}[Sketch of proof]
The case $\EM=\RG$ follows from the case $\EM=\FM$, so we concentrate on the latter case. Also, in each case, the formula for $\ch^*$ follows from the formula for $\ch^!$ using Lemma~\ref{lem:grk-Hom-parity}\eqref{it:ch-parity-D}, so we only consider the latter case.

\eqref{it:ch-parity-Fl}
First we consider the case $\GC=\EC_{s,\FM}$ for some simple reflection $s$. We let $\JD_s$ be the minimal standard parahoric subgroup of $\GD(\mathscr{K})$ associated with $s$, and define $\Fl^s:= \GD(\mathscr{K}) / \JD_s$. We let $p_s \colon \Fl \to \Fl^s$ be the projection, and set $\Fl^s_w:=p_s(\Fl_w)$ for $w \in \Waff$. Then by base change we have 
\[
\FC \star \GC = \FC \star \EC_{s,\FM} \cong (p_s)^* (p_{s})_* \FC[1],
\]
and $\ch_{\Fl}^!(\EC_{s,\FM})=T_s + \vv^{-1}$. The formula in this case can be checked by a direct computation, using the fact that
\[
\HM^\bullet(\Fl^s_w, (i_w^{\Fl^s})^! (p_{s})_* \FC) = \HM^\bullet(\Fl_w, (i_w^{\Fl})^! \FC) \oplus \HM^\bullet(\Fl_{ws}, (i_{ws}^{\Fl})^! \FC)
\]
for $w \in \Waff$, which can be derived from the base change theorem and~\cite[Proposition~2.6]{jmw}. (Here $i_w^{\Fl^s}$, $i_w^{\Fl}$ and $i_{ws}^{\Fl}$ are the obvious inclusions.)

The case where $\GC$ is the sky-scraper sheaf $\EC_{\omega,\FM}$ at the $\ID$-fixed point given by the unique point in $\Fl_\omega$ (for $\omega \in \Omega$) is easy.

Using these special cases one deduces that the formula holds when $\GC$ is of the form
\begin{equation}
\label{eqn:parity-Fl}
\EC_{s_1,\FM} \star \cdots \star \EC_{s_r,\FM} \star \EC_{\omega,\FM}[n]
\end{equation}
where $\omega \in \Omega$, $n \in \ZM$, and $s_1, \cdots, s_r$ are simple reflections. Then one can prove the formula when $\GC$ is indecomposable by induction on the dimension of its support, using the fact that any indecomposable parity complex on $\Fl$ appears as a direct summand of an object of the form~\eqref{eqn:parity-Fl} with $r$ the dimension of the support, see~\cite[\S 4.1]{jmw}. The general case follows.

\eqref{it:ch-parity-Gr}
We have $\GC \star \delta^\EM_1=p_* \GC$, where $p \colon \Fl \to \Gr$ is the natural projection. Then the lemma can be checked by a direct computation, using the formula
\[
\HM^\bullet(\Gr_\lambda, (i_\lambda)^! p_* \GC) \cong \bigoplus_{w \in t_\lambda W} \HM^\bullet(\Fl_w, (i_w^{\Fl})^! \GC)
\]
for all $\lambda \in \XB$, which can be derived from the base change theorem and~\cite[Proposition~2.6]{jmw}. 
\end{proof}

\begin{remark}
Using similar arguments one can show that $\ch_{\Gr}^?(\FC \star \GC)=\ch_{\Gr}^?(\GC) \cdot \ch_{\Fl}^?(\FC)$ for any $\GC$ in $\Parity_{\ID}(\Gr,\EM)$,
$\FC$ in $\Parity_{\ID}(\Fl,\EM)$ and $?=!$ or $*$.
\end{remark}

\begin{prop}
\label{prop:ch-BS-parity}
For $\EM=\RG$ or $\FM$, for any sequence $\underline{s}$ of simple reflections, and for any $\omega \in \Omega$ we have
\[
\ch_{\Gr}^!(\EC_{\EM}(\omega,\us)) =\ch_{\Gr}^*(\EC_{\EM}(\omega,\us)) = \mm(\omega,\us).
\]
\end{prop}

\begin{proof}
This follows immediately from Lemma~\ref{lem:ch-parity-Gr}, using the facts that $\ch_{\Fl}^!(\EC_{s,\EM})=\ch_{\Fl}^*(\EC_{s,\EM})=(T_s + \vv^{-1})$ and $\ch_{\Fl}^*(\EC_{\omega,\EM}) = \ch_{\Fl}^*(\EC_{\omega,\EM}) = T_{\omega^{-1}}$ (where $\EC_{\omega,\EM}$ is defined as in the proof of Lemma~\ref{lem:ch-parity-Gr}\eqref{it:ch-parity-Fl}).
\end{proof}

The following consequence of Theorem~\ref{thm:main-top} will play a crucial role in Section~\ref{sec:tilting-KW}.

\begin{cor}
\label{cor:grk-Hom-parity}
For any any $\omega, \omega' \in \Omega$ and any sequences $\underline{s}$, $\underline{t}$ of simple reflections, the graded $\OC(\tg^*)$-module
\[
\bigoplus_{n \in \ZM} \Hom_{\BSalg} \bigl( (\omega, \us, 0), (\omega', \ut, n) \bigr)
\]
is free. Its graded rank is equal to
$\langle \mm(\omega, \us), \mm(\omega', \ut) \rangle$.
\end{cor}

\begin{proof}
The first assertion follows from Theorem~\ref{thm:main-top} and Lemma~\ref{lem:properties-parity}\eqref{it:morph-parity-For}.
The second one follows from Lemma~\ref{lem:grk-Hom-parity}\eqref{it:dim-Hom-parity-ch} and Proposition~\ref{prop:ch-BS-parity}.
\end{proof}

\section{Kostant--Whittaker reduction}
\label{sec:KW-reduction}

In this section and the next one we work with derived categories of equivariant coherent sheaves, and usual derived functors between them. See~\cite[Appendix~A]{mr} for a brief reminder of the main definitions and properties of these objects.

\subsection{Overview}
\label{ss:overview-KW}

The goal of this section is to introduce and study the ``Kostant--Whittaker reduction'' functor for equivariant coherent sheaves on the Grothendieck resolution of $\GB$. This construction is a mild adaptation of a construction in~\cite{bf}; it relies in a crucial way on geometric results proved in~\cite{riche3}. A related construction also appears in~\cite{dodd}. This functor is used in Section~\ref{sec:tilting-KW} to obtain a description of the category $\Tilt(\ES^{\GB \times \Gm}(\tNC))$ in terms of ``Soergel bimodules.''

After introducing our notation and assumptions in~\S\ref{ss:notation-KW}, we recall the definition of the ``geometric braid group action'' of~\cite{riche, br} in~\S\ref{ss:braid-groth}, and the main results of~\cite{riche3} in~\S\ref{ss:reminder}. In~\S\ref{ss:KW-functors} we define our functors. Then, after some preparation in \S\ref{ss:KW-line-bundle}, in~\S\ref{ss:KW-braid-group} we prove the main result of the section, a certain compatibility property between Kostant--Whittaker reduction and the geometric braid group action.

\subsection{Notation}
\label{ss:notation-KW}

Let $\GB_\ZM$ be a group scheme over $\ZM$ which is a product of split simply connected quasi-simple groups, and general linear groups $\mathrm{GL}_{n,\ZM}$. In particular, $\GB_\ZM$ is a split connected reductive group over $\ZM$.
We let $\BB_\ZM \subset \GB_\ZM$ be a Borel subgroup, $\TB_\ZM \subset \BB_\ZM$ be a (split) maximal torus, and $\BB^+_\ZM \subset \GB_\ZM$ be the Borel subgroup which is opposite to $\BB_\ZM$ (with respect to $\TB_\ZM$). 
We denote by $r$ the rank of $\GB_\ZM$.

We let $N$ be the product of all the prime numbers which are not very good for some quasi-simple factor of $\GB_\ZM$, and set $\RG:=\ZM[1/N]$. Throughout the section, we use the letter $\FM$ to denote an arbitrary geometric point of $\RG$. We use the letter $\EM$ to denote either $\RG$ or $\FM$.

We let $\GB_\RG$, $\BB_\RG$, $\TB_\RG$, $\BB^+_\RG$ be the groups obtained from $\GB_\ZM$, $\BB_\ZM$, $\TB_\ZM$, $\BB^+_\ZM$ by base change to $\RG$, and $\gg_\RG$, $\bg_\RG$, $\tg_\RG$, $\bg^+_\RG$ be their respective Lie algebras. 
We also denote by $\UB_\RG$, resp.~$\UB^+_\RG$, the unipotent radical of $\BB_\RG$, resp.~$\BB^+_\RG$, by $\ng_\RG$, resp.~$\ng^+_\RG$, its Lie algebra, by $W$ the Weyl group of $(\GB_\RG, \TB_\RG)$, and by $\XB:=X^*(\TB_\RG)$ the weight lattice.

We also set
\[
\GB_\FM := \Spec(\FM) \times_{\Spec(\RG)} \GB_\RG,
\]
denote by $\gg_\FM$ the Lie algebra of $\GB_\FM$, by $\BB_\FM$, $\TB_\FM$, $\BB^+_\FM$, $\UB_\FM$, $\UB_\FM^+$ the base change of $\BB_\RG$, $\TB_\RG$, $\BB^+_\RG$, $\UB_\RG$, $\UB_\RG^+$, and by $\bg_\FM$, $\tg_\FM$, $\bg^+_\FM$, $\ng_\FM$, $\ng_\FM^+$ their respective Lie algebras. Note that we also have $\XB=X^*(\TB_\FM)$.

If $V$ is an $\EM$-module, we set $V^*:=\Hom_\EM(V,\EM)$.
In this section we will consider the Grothendieck resolution
\[
\tgg_\EM:= (\GB \times^{\BB} (\gg/\ng)^*)_\EM \hookrightarrow (\GB/\BB \times \gg^*)_\EM.
\]
The scheme $\tgg_\EM$ is a vector bundle over the flag variety $\Flag_\EM:=(\GB/\BB)_\EM$. It is endowed with an action of $\GB_\EM \times \GmE$, where $\GB_\EM$ acts naturally, and the action of $\GmE$ is induced by the action on $(\gg/\ng)_\EM^*$ where $x \in \GmE$ acts by multiplication by $x^{-2}$. 
We will consider the derived categories of equivariant coherent sheaves
$D^{\GB \times \Gm}(\tgg)_\EM$
and $D^{\GB}(\tgg)_\EM$.
We denote by $\langle 1 \rangle$ the functor of tensoring with the free rank one tautological $\GmE$-module, and by
\[
\FM(-) \colon D^{\GB \times \Gm}(\tgg)_\RG \to D^{\GB \times \Gm}(\tgg)_\FM, \quad \FM(-) \colon D^{\GB}(\tgg)_\RG \to D^{\GB}(\tgg)_\FM
\]
the ``modular reduction functors'' induced by the functor $\FM \lotimes_\RG (-)$.

For any $\lambda \in \XB$ we denote by $\OC_{\Flag_\EM}(\lambda)$ the line bundle on $\Flag_\EM$ associated naturally with $\lambda$, and by $\OC_{\tgg_\EM}(\lambda)$ the pullback of $\OC_{\Flag_\EM}(\lambda)$ to $\tgg_\EM$.

We will also consider the morphism
\[
\nu \colon \tgg_\EM \to \tg^*_\EM
\]
which is defined as follows. Consider the restriction morphism $(\gg / \ng)_\EM^* \to \tg^*_\EM$. It is easily checked that this morphism is $\BB_\EM$-equivariant, where $\BB_\EM$ acts trivially on $\tg^*_\EM$. Therefore, our morphism defines a morphism $(\GB \times^{\BB} (\gg/\ng)^*)_\EM \to \tg^*_\EM$, which is our morphism $\nu$. There is also a natural morphism
\[
\pi \colon \tgg_\EM \to \gg^*_\EM
\]
induced by the coadjoint action.

By~\cite[Lemma~4.2.3]{riche3} there exists a $\GB_\RG$-invariant symmetric bilinear form on $\gg_\RG$ which is a perfect pairing. We fix once and for all such a bilinear form, and we denote by $\varkappa \colon \gg_\RG \simto \gg_\RG^*$ the induced ($\GB_\RG$-equivariant) isomorphism. We denote similarly the induced isomorphism $\gg_\FM \simto \gg^*_\FM$.

In the case $\EM=\FM$, we
let $\gg^{\rs}_\FM \subset \gg^{\reg}_\FM \subset \gg_\FM$ be the open subsets consisting of regular semisimple elements and regular elements respectively (see~\cite[\S 2.3]{riche3}), let $\gg_\FM^{*,\rs} \subset \gg_\FM^{*,\reg} \subset \gg^*_\FM$ be their image under $\varkappa$, and let $\tgg_\FM^{\rs} \subset \tgg_\FM^{\reg}$ be the inverse images in $\tgg_\FM$. Then there exists a natural action of $W$ on $\tgg^{\reg}_\FM$ stabilizing $\tgg_\FM^{\rs}$ and commuting with the $\GB_\FM \times \GmF$-action, see~\cite[\S 1.9]{br}. Moreover, the restriction $\nu_\reg \colon \tgg_\FM^\reg \to \tg^*_\FM$ is $W$-equivariant, where $W$ acts naturally on $\tg_\FM^*$. (Indeed, this property is easily checked for the restriction of $\nu_\reg$ to $\tgg^\rs$; then the claim follows by a density argument.)

\subsection{Geometric braid group action}
\label{ss:braid-groth}

We let ${\check \XB}$ be the lattice of cocharacters of $\TB_\RG$, and denote by $\Phi \subset \XB$, resp.~${\check \Phi} \subset {\check \XB}$, the roots, resp.~coroots, of $(\GB_\RG, \TB_\RG)$. We denote by $\Phi^+$ the positive roots, i.e.~the roots appearing in $\bg^+_\RG$. To these data one can associate the affine Weyl group $\Waff$ and the affine braid group $\Baff$ as in~\S\ref{ss:Haff}. We will also denote 
by $\rho$ the half-sum of positive roots.

Let $s$ be a finite simple reflection, associated with a simple root $\alpha$.
Recall the associated subscheme $Z_s^\EM \subset (\tgg \times_{\gg^*} \tgg)_\EM$ defined in~\cite[\S\S 1.3--1.5]{br}. (If $\EM=\FM$, then $Z_s^\EM$ is the closure of the graph of the action of $s$ on $\tgg^{\reg}_\EM$.) We denote by
\[
\TM^s_{\tgg}, \, \SM^s_{\tgg} \colon D^{\GB\times \Gm}(\tgg)_\EM \to D^{\GB \times \Gm}(\tgg)_\EM
\]
the Fourier--Mukai transforms with kernels $\OC_{Z_s^\EM} \langle -1 \rangle$ and $\OC_{Z_s^\EM}(-\rho, \rho-\alpha) \langle -1 \rangle$ respectively. (Here, $\OC_{Z_s^\EM}(-\rho, \rho-\alpha)$ denotes the tensor product of $\OC_{Z_s^\EM}$ with the line bundle on $\tgg \times \tgg_\EM$ which is the pullback of the line bundle $\OC_{\Flag_\EM}(-\rho) \boxtimes \OC_{\Flag_\EM}(\rho-\alpha)$ on $\Flag \times \Flag_\EM$.) By \cite[Lemma~1.5.1 \& Proposition~1.10.3]{br}, these functors are quasi-inverse equivalences of categories. We will use the same symbols to denote the similar autoequivalences of the category  $D^{\GB}(\tgg)_\EM$.

By~\cite[Section~1]{br}, there exists a right action\footnote{Here by a (left) action of a group on a category we mean a group morphism from the given group to the group of \emph{isomorphism classes} of autoequivalences of the category. As usual, a right action of a group is a left action of the opposite group.} of the group $\Baff$ on the category $D^{\GB \times \Gm}(\tgg)_\EM$, resp.~$D^{\GB}(\tgg)_\EM$, where $T_s$ acts by the functor $\TM^s_{\tgg}$ (for any finite simple reflection $s$), and where $\theta_\lambda$ acts by tensoring with the line bundle $\OC_{\tgg_\EM}(\lambda)$ (for any $\lambda \in \XB$). (The same remarks as in~\cite[\S 3.3]{mr} on the difference with the conventions of~\cite{riche, br} apply here: the (right) action considered in the present paper differs from the (left) action of~\cite{riche,br} by the composition with the anti-automorphism of $\Baff$ fixing all generators $T_s$ for $s$ a simple reflection and $\theta_\lambda$ for $\lambda \in \XB$.) For $b \in \Baff$, we denote by
\[
\IS^\EM_b \colon D^{\GB \times \Gm}(\tgg)_\EM \simto D^{\GB \times \Gm}(\tgg)_\EM
\]
the action of $b$. (This functor is defined only up to isomorphism.) It is easily checked that we have an isomorphism of functors
\begin{equation}
\label{eqn:F-IS}
\FM \circ \IS_b^\RG \cong \IS_b^\FM \circ \FM
\end{equation}
for any $b \in \Baff$.

For $s$ a finite simple reflection, associated with a simple root $\alpha$, we also set $\tgg_s^\EM := \GB \times^{\PB_s} (\gg / (\pg_s)^{\mathrm{nil}})^*_\EM$, where $\PB^\EM_s$ is the minimal standard parabolic subgroup of $\GB_\EM$ associated with $s$, and $(\pg_s^\EM)^{\mathrm{nil}}$ is the Lie algebra of the unipotent radical of $\PB^\EM_s$. There exists a natural morphism $\widetilde{\pi}_s \colon \tgg_\EM \to \tgg^\EM_s$. 
By~\cite[Corollary~5.3.2]{riche}
there exists natural exact sequences
\begin{equation}
\label{eqn:ses-tgg}
\OC_{\Delta \tgg_\EM} \langle 2 \rangle \hookrightarrow \OC_{\tgg \times_{\tgg_s} \tgg_\EM} \twoheadrightarrow \OC_{Z_s^\EM}, \quad
\OC_{Z_s^\EM}(-\rho, \rho-\alpha) \hookrightarrow \OC_{\tgg \times_{\tgg_s} \tgg_\EM} \twoheadrightarrow \OC_{\Delta \tgg_\EM}
\end{equation}
in $D^{\GB \times \Gm}(\tgg \times \tgg)_\EM$, where in each sequence the surjection is induced by restriction of functions, and where $\Delta \tgg_\EM \subset \tgg \times \tgg_\EM$ is the diagonal copy. (In fact, in~\cite{riche} only the case $\EM=\FM$ is treated, but one can easily check that the same arguments apply in the case $\EM=\RG$.)

\subsection{Reminder on~\cite{riche3}: Kostant section and universal centralizer}
\label{ss:reminder}

We denote by $\simple \subset \Phi$ the subset of simple roots. As in~\cite[\S 4.3]{riche3},\footnote{Note that, compared to~\cite{riche3}, we have switched the roles of positive and negative roots. Another difference which appears below is that we work with $\gg^*$ instead of $\gg$, using the identification $\varkappa$.} for each $\alpha \in -\simple$ we choose an element $e_\alpha \in \gg_\ZM$ which forms a $\ZM$-basis of the $\alpha$-weight space in $\gg_\ZM$, and set
\[
e:= \sum_{\alpha \in -\simple} e_\alpha \quad \in \gg_\RG.
\]
We will consider the cocharacter $\lav_\circ:=\sum_{\alpha \in -\Phi^+} \alpha^{\vee}$, and the $\GmR$-actions on $\gg_\RG$ and $\gg_\RG^*$ defined by
\[
x \cdot y = x^{-2} \lav_\circ(x) \cdot y
\]
for $x \in \GmR$ and $y$ either in $\gg_\RG$ or in $\gg_\RG^*$. With these definitions, $\varkappa$ is $\GmR$-equivariant, $e$ is fixed by the action, and $\bg_\RG^+$ and $\ng_\RG^+$ are $\GmR$-stable.

By~\cite[Lemma~4.3.1]{riche3}, the quotient $\bg_\RG^+/[e,\ng^+_\RG]$ is free of rank $r$; therefore one can choose a $\GmR$-stable $\RG$-submodule $\sg_\RG \subset \bg_\RG^+$ such that $\bg_\RG^+=\sg_\RG \oplus [e, \ng_\RG^+]$. We set
\[
\SC_\RG := \varkappa(e + \sg)_\RG, \qquad \tSC_\RG:=\SC \times_{\gg^*} \tgg_\RG.
\]
The $\GmR$-action on $\gg^*_\RG$ defined above stabilizes $\SC_\RG$, and contracts it to $\varkappa(e)$ as $t \to \infty$. Similarly, the action on $\tgg_\RG$ defined by
\[
x \cdot [g:\xi] = [\lav_\circ(x) g : x^{-2} \xi]
\]
(for $x \in \GmR$, $g \in \GB_\RG$ and $\xi \in (\gg/\ng)^*_\RG$) stabilizes $\tSC_\RG$, and is contracting
(see~\cite[\S 3.5]{riche3} for details). We will denote by $\SC_\FM$, resp.~$\tSC_\FM$, the scheme obtained from $\SC_\RG$, resp.~$\tSC_\RG$, by base change to $\FM$.

The following result is proved in~\cite[Propositions~3.5.5 \&~4.5.2]{riche3}. Here we consider the $\GmE$-action on $\tg^*_\EM$ where $x$ acts by multiplication by $x^{-2}$.

\begin{prop}
\label{prop:Kostant-section}
The morphism $\nu \colon \tgg_\EM \to \tg^*_\EM$ induces a $\GmE$-equivariant isomorphism of $\EM$-schemes
\[
\nu_\SC \colon \tSC_\EM \simto \tg^*_\EM.
\]
\end{prop}

The following result is proved in~\cite[Lemmas 3.5.1 \& 4.5.1]{riche3}.

\begin{lem}
\label{lem:a-smooth}
The morphism
\[
a \colon \GB \times \tSC_\EM \to \tgg_\EM
\]
induced by the $\GB_\EM$-action on $\tgg_\EM$ is smooth. When $\EM=\FM$, it factorizes through a surjective morphism $\GB \times \tSC_\FM \to \tgg^\reg_\FM$.
\end{lem}

We denote by $\tIB^\EM$ the universal centralizer associated with the action of $\GB_\EM$ on $\tgg_\EM$, see~\eqref{eqn:stabilizer}. We will also denote by $\tIB_\SC^\EM$ the restriction of $\tIB^\EM$ to $\tSC_\EM$. Then by~\cite[Corollary 3.5.8 \& Proposition 4.5.3]{riche3}, $\tIB_\SC^\EM$ is a commutative smooth group scheme over $\tSC_\EM$. We denote by $\tIG_\SC^\EM$ its Lie algebra; it is a locally free sheaf of commutative $\OC_{\tSC_\EM}$-Lie algebras (see~\cite[\S 2.1]{riche3}).

Recall that the quotient $\tg^*_\EM/W$ is a smooth scheme, isomorphic to an affine space (see~\cite[Th{\'e}or{\`e}me~3 \& Corollaire on p.~296]{demazure}).
We will denote by $\varrho \colon \tg^*_\EM \to \tg^*_\EM/W$ the quotient morphism; then we denote by
$\eta_\SC \colon \tSC_\EM \to \tg^*_\EM/W$ the composition
\[
\tSC_\EM \xrightarrow{\nu_\SC} \tg^*_\EM \xrightarrow{\varrho} \tg^*_\EM/W.
\]

The following result follows from~\cite[Theorems~3.5.12 \& 4.5.5]{riche3}, using the fact that the natural morphism $\SC_\EM \to \tg^*_\EM / W$ is an isomorphism (see~\cite[Theorems 3.2.2 \& 4.3.3]{riche3}, combined with~\cite[Propositions 2.3.2 \& 4.2.1]{riche3}).

\begin{thm}
\label{thm:Lie-centralizer}
There exists a canonical isomorphism of sheaves of commutative $\OC_{\tSC_\EM}$-Lie algebras
\[
\tIG_{\SC}^\EM \cong (\eta_\SC)^* \Omega_{\tg^*_\EM/W}.
\]
\end{thm}

For the remainder of this subsection, let us
consider the case $\EM=\FM$. First we recall that there exists a unique smooth commutative group scheme $\tJB^\FM$ over $\tg^*_\FM$ whose pull-back under $\nu_\reg$ is the restriction $\tIB^\FM_\reg$ of $\tIB^\FM$ to $\tgg_\FM^\reg$, see~\cite[Proposition~3.3.9]{riche3}.
Now by~\cite[Remark~3.5.9]{riche3}, the composition
\[
\Coh^{\GB}(\tgg^\reg)_\FM \to \Rep(\tIB^\FM_\reg) \to \Rep(\tIB^\FM_\SC) \simto \Rep(\tJB^\FM)
\]
is an equivalence of categories, where the first functor is the functor~\eqref{eqn:stabilizer-functor} in our particular situation, the second functor is induced by restriction to $\tSC$, and the last functor is induced by the isomorphism of Proposition~\ref{prop:Kostant-section}.

Similarly, let $s$ be a finite simple reflection, let $W_s:=\{1,s\} \subset W$, and consider the natural morphism $\tgg_s^\FM \to \tg_\FM^*/W_s$. Let also $\tgg_s^{\FM,\reg} \subset \tgg_s^\FM$ be the inverse image of $\gg^{*,\reg}_\FM$ under the natural morphism $\tgg_s^\FM \to \gg^*_\FM$, and let $\tIB^\FM_{s}$ be the universal stabilizer associated with the action of $\GB_\FM$ on $\tgg^\FM_s$. Then the same arguments as in the case of $\tgg_\FM$ (see in particular~\cite[Remark~3.5.7]{riche3}) show that there exists a unique commutative group scheme $\tJB_s^\FM$ on $\tg^*_\FM / W_s$ whose pull-back under the morphism $\tgg^{\FM,\reg}_s \to \tg_\FM^*/W_s$ is the restriction of $\tIB^\FM_{s}$ to $\tgg_s^{\FM,\reg}$; moreover, as above we have a natural equivalence of categories
$\Coh^{\GB}(\tgg^{\reg}_s)_\FM \simto \Rep(\tJB_s^\FM)$.
There exists a canonical Cartesian diagram
\[
\xymatrix@R=0.4cm{
\tJB^\FM \ar[r] \ar[d] & \tJB_s^\FM \ar[d] \\
\tg^*_\FM \ar[r] & \tg^*_\FM / W_s,
}
\]
so that the direct and inverse image functors under the quotient morphism $\tg^*_\FM \to \tg^*_\FM / W_s$ induce functors $\Rep(\tJB^\FM) \to \Rep(\tJB^\FM_s)$ and $\Rep(\tJB^\FM_s) \to \Rep(\tJB^\FM)$ respectively. Under the equivalences considered above, these functors correspond to the direct and inverse image functors under the restriction of $\widetilde{\pi}_s$ to $\tgg_\FM^\reg$, respectively.

\subsection{Definition of the functors}
\label{ss:KW-functors}

Let us denote by $\TF(\tg^*_\EM/W)$ the tangent bundle of the smooth $\EM$-scheme $\tg^*_{\EM} / W$. We consider the $\GmE$-action on $\tg^*_{\EM} / W$ such that the corresponding grading on $\OC(\tg^*_{\EM} / W)=\OC(\tg^*_{\EM})^W$ is obtained by restriction of the grading on $\OC(\tg^*_{\EM})$ where the generators $\tg_\EM \subset \OC(\tg^*_{\EM})$ are in degree $2$, so that the morphism $\varrho$ of~\S\ref{ss:reminder} is $\GmE$-equivariant. This action induces a $\GmE$-equivariant structure on $\Omega_{\tg^*_\EM/W}$. We consider the $\GmE$-action on $\TF(\tg^*_\EM/W)$ such that $\OC(\TF(\tg^*_\EM/W))$ is the symmetric algebra of $\Omega_{\tg^*_\EM/W} \langle -2 \rangle$ as a graded algebra. In other words, the $\GmE$-action on $\TF(\tg^*_\EM/W)$ is the combination of the action induced by the action on $\tg^*_\EM/W$, with multiplication by $x^2$ in the fibers of the projection $\TF(\tg^*_\EM/W) \to \tg^*_\EM/W$.

We will consider the bounded derived categories of (equivariant) coherent sheaves
$D^{\Gm}(\tg^* \times_{\tg^* / W} \TF(\tg^* / W))_\EM$
and $D(\tg^* \times_{\tg^* / W} \TF(\tg^* / W))_{\EM}$.
In this context also we have ``modular reduction functors''
\begin{align*}
\FM(-) \colon D^{\Gm}(\tg^* \times_{\tg^* / W} \TF(\tg^* / W))_{\RG} &\to D^{\Gm}(\tg^* \times_{\tg^* / W} \TF(\tg^* / W))_{\FM}, \\
\FM(-) \colon D(\tg^* \times_{\tg^* / W} \TF(\tg^* / W))_{\RG} &\to D(\tg^* \times_{\tg^* / W} \TF(\tg^* / W))_{\FM}.
\end{align*}

Now we can explain the construction of the ``Kostant--Whittaker reduction'' functors
\begin{align}
\kappa_\EM \colon D^{\GB \times \Gm}(\tgg)_\EM &\to D^{\Gm}(\tg^* \times_{\tg^* / W} \TF(\tg^* / W))_{\EM}, \label{eqn:KW1} \\
\overline{\kappa}_\EM \colon D^{\GB}(\tgg)_\EM &\to D(\tg^* \times_{\tg^* / W} \TF(\tg^* / W))_{\EM}. \label{eqn:KW2}
\end{align}
First, 
consider the functor
\[
\overline{\kappa}'_\EM \colon \Coh^{\GB}(\tgg)_\EM \to \Coh(\tg^*)_{\EM}
\]
defined as the pullback functor associated with the embedding $\tSC_\EM \hookrightarrow \tgg_\EM$, where we identify $\tSC_\EM$ with $\tg^*_{\EM}$ using the isomorphism $\nu_\SC$ from Proposition~\ref{prop:Kostant-section}.

\begin{remark}
Consider the closed subvariety $\tUp_\FM \subset \tgg_\FM$ defined in~\cite[\S 3.5]{riche3}, and the analogously defined closed subscheme $\tUp_\RG \subset \tgg_\RG$. Then $\UB_\EM$ acts freely on $\tUp_\EM$, and $\nu$ induces an isomorphism of $\EM$-schemes
$\tUp / \UB_\EM \simto \tg^*_\EM$. (In the case $\EM=\FM$, this fact follows from~\cite[Proposition~3.2.1, Theorem~3.2.2 \& Proposition~3.5.5]{riche3}; the case $\EM=\RG$ is similar.) Using this isomorphism, the functor $\overline{\kappa}_\EM'$ can be described more canonically as the composition of restriction to $\tUp_\EM$ with the natural equivalences $\Coh^{\UB}(\tUp)_\EM \simto \Coh(\tUp / \UB_\EM) \cong \Coh(\tg^*_{\EM})$.
\end{remark}

\begin{lem}
\label{lem:KW-exact}
The functor $\overline{\kappa}'_\EM$
is exact.
\end{lem}

\begin{proof}
Our functor can be written as the composition
\[
\Coh^{\GB}(\tgg)_\EM \xrightarrow{a^*} \Coh^{\GB}(\GB \times \tSC)_\EM \simto \Coh(\tSC)_\EM \xrightarrow[\sim]{(\nu_\SC)_*} \Coh(\tg^*)_\EM,
\]
where $a$ is the morphism considered in Lemma~\ref{lem:a-smooth}, and the middle arrow is the obvious equivalence. Now $a$ is a smooth (in particular, flat) morphism by Lemma~\ref{lem:a-smooth}, which implies exactness.
\end{proof}

Now we can explain the definition of $\overline{\kappa}_\EM$. This functor will be induced by an exact functor $\Coh^{\GB}(\tgg)_\EM \to \Coh(\tg^* \times_{\tg^* / W} \TF(\tg^* / W))_{\EM}$, which we denote similarly. In fact, starting from an equivariant coherent sheaf $\FC$ on $\tgg_\EM$, its restriction to $\tSC_\EM$ is naturally endowed with an action of the universal centralizer $\tIB_\SC^\EM$, see~\S\ref{ss:equiv-coh}. Differentiating this action we obtain an action of the Lie algebra $\tIG_\SC^\EM$ of $\tIB_\SC^\EM$. By Theorem~\ref{thm:Lie-centralizer} one can identify $\tIG_\SC^\EM$ with $(\eta_\SC)^* \Omega_{\tg^*_\EM /W}$, considered as a sheaf of commutative Lie algebras on $\tSC_\EM$. Identifying $\tSC_\EM$ with $\tg^*_\EM$ via $\nu_\SC$ (see Proposition~\ref{prop:Kostant-section}), we obtain an action of the commutative Lie algebra $\varrho^* \Omega_{\tg_\EM^*/W}$ on $\overline{\kappa}'_\EM(\FC)$, hence also an action of its symmetric algebra $\varrho^*\mathrm{S}_{\OC_{\tg_\EM^*/W}}(\Omega_{\tg^*_\EM/W})$. This action defines a coherent sheaf $\overline{\kappa}_\EM(\FC)$ on $\tg^* \times_{\tg^* / W} \TF(\tg^* / W)_{\EM}$ whose direct image under the affine morphism $\tg^* \times_{\tg^* / W} \TF(\tg^* / W)_{\EM} \to \tg^*_\EM$ is $\overline{\kappa}'_\EM(\FC)$.
It follows from Lemma~\ref{lem:KW-exact} that the functor
\[
\overline{\kappa}_\EM \colon \Coh^{\GB}(\tgg)_\EM \to \Coh(\tg^* \times_{\tg^* / W} \TF(\tg^* / W))_{\EM}
\]
that we have just defined is exact. Then the functor~\eqref{eqn:KW2} is defined as the induced functor between bounded derived categories.

The construction of the functor $\kappa_\EM$ is similar, simply keeping track of the appropriate $\GmE$-actions. More precisely, recall the action of $\GmE$ on $\tSC_\EM$ defined in \S\ref{ss:reminder}. One can ``extend'' this action to the group scheme $\GB_\EM \times \tSC_\EM$ via 
\[
x \cdot (g,y)=(\lav_\circ(x) g \lav_\circ(x)^{-1}, x \cdot y)
\]
for $x \in \GmE$, $g \in \GB_\EM$ and $y \in \tSC_\EM$. Then the subgroup $\tIB_{\SC}^\EM$ is $\GmE$-stable, and the projection $\tIB_\SC^\EM \to \tSC_\EM$ is $\GmE$-equivariant. This action induces a $\GmE$-equivariant structure of the coherent sheaf $\tIG_\SC^\EM$ (on $\tSC_\EM$), and it is easily seen that the isomorphism of Theorem~\ref{thm:Lie-centralizer} is $\GmE$-equivariant, where the action on the right-hand side is induced by the action on $\TF(\tg^*_\EM/W)$ considered at the beginning of this subsection (see~\cite[Proof of Theorem~3.4.2]{riche3}). Now if $\FC$ is in $\Coh^{\GB \times \Gm}(\tgg)_\EM$, then the restriction of $\FC$ to $\tSC_\EM$ is a $\GmE$-equivariant coherent sheaf, and the action of $\tIB_\SC^\EM$ is compatible with this structure in the obvious sense, which allows to define the exact functor
\[
\kappa_\EM \colon \Coh^{\GB \times \Gm}(\tgg)_\EM \to \Coh^{\Gm}(\tg^* \times_{\tg^* / W} \TF(\tg^* / W))_{\EM}
\]
by the same recipe as for $\overline{\kappa}_\EM$. Then the functor~\eqref{eqn:KW1} is defined as the induced functor between bounded derived categories.

The proof of the following lemma is easy, and left to the reader.

\begin{lem}
\label{lem:F-kappa}
There exist canonical isomorphisms of functors
\[
\FM \circ \kappa_\RG \cong \kappa_\FM \circ \FM, \qquad \FM \circ \overline{\kappa}_\RG \cong \overline{\kappa}_\FM \circ \FM,
\]
where, in both equations, the functor $\FM$ on the left-hand side is the modular reduction functor on the $\RG$-scheme $(\tg^* \times_{\tg^* / W} \TF(\tg^* / W))_{\RG}$, and the functor $\FM$ on the right-hand side is the similar functor on $\tgg_\RG$.\qed
\end{lem}

\subsection{Another geometric braid group action}
\label{ss:weyl-action}

The goal of this subsection is to define a ``geometric'' (right) action of the group $\Baff$ on the category $D^{\Gm}(\tg^* \times_{\tg^* / W} \TF(\tg^* / W))_{\EM}$. We begin by defining an action of $\Waff$.

First, there exists a natural action of $W$ on the $\EM$-scheme $(\tg^* \times_{\tg^* / W} \TF(\tg^* / W))_{\EM}$, induced by the action on $\tg^*_\EM$. For $w \in W$, we denote by 
\[
{}' \KS_{w}^\EM \colon D^{\Gm}(\tg^* \times_{\tg^* / W} \TF(\tg^* / W))_{\EM} \simto D^{\Gm}(\tg^* \times_{\tg^* / W} \TF(\tg^* / W))_{\EM}
\]
the pullback functor associated with the morphism given by the action of $w$.

Then, let $\lambda \in \XB$. To define the functor associated with $t_\lambda$ we will identify $D^{\Gm}(\tg^* \times_{\tg^* / W} \TF(\tg^* / W))_{\EM}$ with the derived category of finitely generated graded modules over the commutative $\OC(\tg^*_\EM)$-Lie algebra $\bigl( \OC(\tg^*) \otimes_{\OC(\tg^*/W)} \Omega(\tg^*/W) \bigr)_\EM$. We define a module $F_\lambda^\EM$ for this Lie algebra as follows. As an $\OC(\tg^*_\EM)$-module, $F_\lambda^\EM$ is free of rank one. Then to define the action of the commutative Lie algebra $\OC(\tg^*) \otimes_{\OC(\tg^*/W)} \Omega(\tg^*/W)_\EM$ it is enough to define a morphism of $\OC(\tg^*_\EM)$-modules
\begin{equation}
\label{eqn:definition-F-lambda}
\bigl( \OC(\tg^*) \otimes_{\OC(\tg^*/W)} \Omega(\tg^*/W) \bigr)_\EM \to \OC(\tg^*_\EM).
\end{equation}
In order to do so, we interpret the left-hand side as the module of sections of the projection
\begin{equation}
\label{eqn:proj-T*}
(\tg^* \times_{\tg^* / W} \TF^*(\tg^* / W))_\EM \to \tg^*_\EM,
\end{equation}
where $\TF^*(\tg^*_{\EM} / W)$ is the cotangent bundle of the smooth $\EM$-scheme $\tg^*_{\EM} / W$.
The morphism $\varrho \colon \tg^*_\EM \to \tg^*_\EM/W$ defines a natural morphism
\[
d^*\varrho \colon (\tg^* \times_{\tg^* / W} \TF^*(\tg^* / W))_\EM \to \TF^*(\tg^*_\EM).
\]
Now since $\tg^*_\EM$ is an affine space, the right-hand side is canonically isomorphic to $\tg^*_\EM \times \tg_\EM$. Hence, starting with a section $\sigma$ of~\eqref{eqn:proj-T*} the composition of $\sigma$ with
\[
(\tg^* \times_{\tg^* / W} \TF^*(\tg^* / W))_\EM \xrightarrow{d^*\varrho} \TF^*(\tg^*_\EM) \cong \tg^* \times \tg_\EM \twoheadrightarrow \tg_\EM \xrightarrow{d\lambda} \EM
\]
defines an element in $\OC(\tg^*_\EM)$. This construction provides the definition of~\eqref{eqn:definition-F-lambda}, hence also of the module $F_\lambda^\EM$. It is clear that if $\lambda,\mu \in \XB$ and $w \in W$ we have canonical isomorphisms of $\bigl( \OC(\tg^*) \otimes_{\OC(\tg^*/W)} \Omega(\tg^*/W) \bigr)_\EM$-modules
\begin{equation}
\label{eqn:compositions-F}
F_\lambda^\EM \otimes_{\OC(\tg^*_\EM)} F_\mu^\EM \cong F_{\lambda+\mu}^\EM \quad \text{and} \quad {}' \KS_{w}^\EM(F_\lambda^\EM) = F^\EM_{w^{-1}\lambda}.
\end{equation}
(In the first isomorphism, the left-hand side is endowed with the natural action on the tensor product).

With this definition at hand, we define the functor
\[
{}' \KS_{t_\lambda}^\EM \colon D^{\Gm}(\tg^* \times_{\tg^* / W} \TF(\tg^* / W))_\EM \simto D^{\Gm}(\tg^* \times_{\tg^* / W} \TF(\tg^* / W))_{\EM}
\]
as the functor of tensoring (over $\OC(\tg^*_\EM)$) with the module $F_\lambda^\EM$.

Using~\eqref{eqn:compositions-F} one can easily check that we have canonical isomorphisms
\[
{}' \KS_{v}^\EM \circ {}' \KS_{w}^\EM \cong {}' \KS_{w v}^\EM, \quad {}' \KS_{t_\lambda}^\EM \circ {}' \KS_{t_\mu}^\EM \cong {}' \KS_{t_{\lambda+\mu}}^\EM, \quad {}' \KS_{w}^\EM \circ {}' \KS_{t_\lambda}^\EM \cong {}' \KS_{t_{w^{-1} \lambda}}^\EM \circ {}' \KS_{w}^\EM
\]
for all $v,w \in W$ and $\lambda, \mu \in \XB$. In other words, these functors define a right action of the group $\Waff$ on the category $D^{\Gm}(\tg^* \times_{\tg^* / W} \TF(\tg^* / W))_{\EM}$. For $w \in \Waff$, we denote by $'\KS^\EM_w$ the functor giving the action of $w$.

Now we ``renormalize'' this action as follows.
For $w \in W$ and $\lambda \in \XB$, we set
\[
\KS_{T_w}^\EM := {}' \KS_w^\EM \langle -\ell(w) \rangle, \qquad \KS_{\theta_\lambda}^\EM := {}' \KS^\EM_{t_\lambda} \langle \lambda(\lav_\circ)\rangle.
\]
Then these functors extend to a right action of $\Baff$ on $D^{\Gm}(\tg^* \times_{\tg^* / W} \TF(\tg^* / W))_{\EM}$. For any $b \in \Baff$, we denote by $\KS_b^\EM$ the action of $b$. Note that if $\overline{b}$ is the image of $b$ under the canonical surjection $\Baff \twoheadrightarrow \Waff$, then there exists $n(b) \in \ZM$ such that
\begin{equation}
\label{eqn:K-'K}
\KS_b^\EM \cong {}' \KS^\EM_{\overline{b}} \langle n(b) \rangle.
\end{equation}
Note also that if $\omega \in \Omega$ then we have necessarily
\begin{equation}
\label{eqn:n-Tomega}
n(T_\omega)=0.
\end{equation}
(Indeed, 
if the claim is known for a power of $\omega$ then it follows for $\omega$. Now this claim is obvious if $\omega=t_\lambda$ for some $\lambda \in \XB$ such that $\lambda( \alpha^\vee )=0$ for all $\alpha \in \Phi$. Since the quotient of $\Omega$ by the subgroup consisting of such elements is finite, this suffices to imply the claim for all $\omega \in \Omega$.)

\subsection{Kostant--Whittaker reduction of line bundles}
\label{ss:KW-line-bundle}

The goal of this subsection is to prove the following result (whose proof is quite technical, even though the statement is very natural).

\begin{prop}
\label{prop:KW-line-bundle}
For any $\lambda \in \XB$, there exists an isomorphism
\[
\kappa_\EM(\OC_{\tgg_\EM}(\lambda)) \cong F_\lambda^\EM \langle \lambda(\lav_\circ) \rangle.
\]
\end{prop}

We start with two preliminary results. 
The first lemma is a generalization of a lemma in~\cite{gk}.

\begin{lem}
\label{lem:S-position}
If $g \in \GB_\FM$ and $\xi \in (\gg/\ng)_\FM^*$ are such that $g \cdot \xi \in \SC_\FM$, then $g \in \UB^+_\FM \cdot \BB_\FM$.
\end{lem}

\begin{proof}
Our assumption ensures that $[g:\xi] \in \tSC_\FM$.
Recall the contracting $\GmF$-actions on $\SC_\FM$ and $\tSC_\FM$ defined in~\S\ref{ss:reminder}. We have $\lim_{x \to \infty} x \cdot [g:\xi] = [1:\varkappa(e)]$, hence $\lim_{x \to \infty} \lav_\circ(x) g \BB_\FM = 1 \BB_\FM$. 
Since $1 \BB_\FM$ belongs to the $\lav_\circ$-stable open subset $\UB_\FM^+ \BB_\FM / \BB_\FM \subset \Flag_\FM$, we deduce that 
$g\BB_\FM$ also belongs to $\UB_\FM^+ \BB_\FM / \BB_\FM$, which finishes the proof.
\end{proof}

\begin{lem}
\label{lem:Flambda-R-F}
Let $\lambda \in \XB$. Let $M$ be an object of $\Coh^{\Gm}(\tg^* \times_{\tg^* / W} \TF(\tg^* / W))_\RG$ which is flat over $\RG$, and whose direct image to $\tg^*_\RG$ is coherent. Assume that for any geometric point $\FM$ of $\RG$ there exists an isomorphism $\FM \otimes_\RG M \cong F_\lambda^\FM$ in $\Coh^{\Gm}(\tg^* \times_{\tg^* / W} \TF(\tg^* / W))_\FM$. Then there exists an isomorphism $M \cong F_\lambda^\RG$.
\end{lem}

\begin{proof}
In this proof we identify the category $\Coh^{\Gm}(\tg^* \times_{\tg^* / W} \TF(\tg^* / W))_\RG$ with the category of finitely generated graded $\OC(\tg^* \times_{\tg^* / W} \TF(\tg^* / W))_\RG$-modules, and similarly for $\FM$.

First we construct an isomorphism of graded $\OC(\tg^*_\RG)$-modules $\OC(\tg^*_\RG) \simto M$. In fact, if $M_0$ denotes the degree $0$-part of $M$ (a free $\RG$-module of finite rank), our assumption ensures that $\dim_\CM(\CM \otimes_\RG M_0)=1$; we deduce that $M_0$ has rank $1$.
Choosing any basis for this module, we obtain a morphism of graded $\OC(\tg^*_\RG)$-modules
\[
\phi \colon \OC(\tg^*_\RG) \to M.
\]
Our assumption implies that the induced morphism $\OC(\tg^*_\FM) \to \FM \otimes_\RG M$ is an isomorphism for all $\FM$. Since all graded pieces in $\OC(\tg^*_\RG)$ and $M$ are free $\RG$-modules of finite rank, we deduce that $\phi$ is an isomorphism.

We claim that $\phi$ induces an isomorphism of graded $\OC(\tg^* \times_{\tg^* / W} \TF(\tg^* / W))_\RG$-modules $F_\lambda^\RG \simto M$. In fact, it is enough to prove that for all $\omega \in \Omega(\tg^*_\RG/W)$ we have $\omega \cdot \phi(1) = (d\lambda)(\omega') \cdot \phi(1)$, where $\omega'$ is the image of $\omega$ in $\Omega(\tg^*_\RG) = \tg_\RG \otimes_\RG \OC(\tg^*_\RG)$ as in~\S\ref{ss:weyl-action} (and we still write $d\lambda$ for the morphism $(d\lambda) \otimes 1 \colon \tg_\RG \otimes_\RG \OC(\tg^*_\RG) \to \OC(\tg^*_\RG)$). However, 
the embedding
\[
\End_{\Mod^{\mathrm{gr}}(\OC(\tg^* \times_{\tg^* / W} \TF(\tg^* / W))_{\CM})}(F_\lambda^\CM) \hookrightarrow \End_{\Mod^{\mathrm{gr}}(\OC(\tg^*_{\CM}))}(F_\lambda^\CM)
\]
is an isomorphism, since the right-hand side is reduced to scalars. Hence $\CM \otimes_\RG \phi$ induces an isomorphism of $\OC(\tg^* \times_{\tg^* / W} \TF(\tg^* / W))_\CM$-modules $F_\lambda^\CM \simto \CM \otimes_\RG M$. We deduce that the image of $\omega \cdot \phi(1) - (d\lambda)(\omega') \cdot \phi(1)$ in $\CM \otimes_\RG M$ is $0$; it follows that this element is zero, which finishes the proof.
\end{proof}

\begin{proof}[Proof of Proposition~{\rm \ref{prop:KW-line-bundle}}]
Using Lemmas~\ref{lem:F-kappa} and~\ref{lem:Flambda-R-F}, it is enough to prove the isomorphism in the case $\EM=\FM$.
For simplicity, in the proof we omit the subscripts ``$\FM$.''

We denote by $\tgg^\diamond$ the inverse image in $\tgg$ of the open subset $\UB^+ \BB / \BB \subset \Flag$; more concretely we have $\tgg^\diamond = \UB^+ \BB \times^\BB (\gg/\ng)^*$. This open subset is stable under the action of $\Gm$ obtained by restricting the action of $\GB \times \Gm$ along the embedding $\Gm \to \GB \times \Gm$ sending $x$ to $(\lav_\circ(x), x)$. Moreover,
by Lemma~\ref{lem:S-position}, we have $\tSC \subset \tgg^\diamond$.

First, we claim that there exists a canonical isomorphism of $\Gm$-equivariant line bundles (in other words a canonical trivialization)
\begin{equation}
\label{eqn:rest-line-bundle}
\OC_{\tgg}(\lambda)_{| \tgg^\diamond} \simto \OC_{\tgg^\diamond} \langle \lambda(\lav_\circ) \rangle.
\end{equation}
In fact, consider the variety $\XC:=\GB \times^{\UB} (\gg/\ng)^*$. This variety is endowed with a (free) action of $\TB$ defined by $t \cdot [g:\xi] = [g t^{-1}: t\cdot \xi]$, and $\tgg$ is the quotient of $\XC$ by this action. It is also endowed with a natural action of $\GB \times \Gm$, such that the quotient morphism $q \colon \XC \to \tgg$ is $\GB \times \Gm$-equivariant. Moreover, it follows from the definition that we have
\[
\OC_{\tgg}(\lambda) = (q_* \OC_{\XC})^{\TB,-\lambda}
\]
as $\GB \times \Gm$-equivariant line bundles,
where on the right-hand side we mean sections on which $\TB$ acts by the character $-\lambda$. Consider now the following commutative diagram:
\[
\xymatrix@R=0.5cm{
\TB \times \UB^+ \times (\gg/\ng)^* \ar[r]^-{\sim} \ar[d]^{q^\diamond} & \UB^+ \BB \times^{\UB} (\gg/\ng)^* \ar[d] \ar@{^{(}->}[r] & \XC \ar[d]^-{q} \\
\UB^+ \times (\gg/\ng)^*  \ar[r]^-{\sim} & \UB^+ \BB \times^\BB (\gg/\ng)^* \ar@{^{(}->}[r] & \tgg.
}
\]
Here $q^\diamond$ is the projection, the right horizontal arrows are the natural (open) embeddings, and the left arrows are defined by $(t,u,\xi) \mapsto [ut^{-1}: t \cdot \xi]$ and $(u,\xi) \mapsto [u:\xi]$ respectively. All the maps in this diagram are $\Gm$-equivariant, if $\Gm$ acts on $\XC$ and $\tgg$ via the morphism $\Gm \to \GB \times \Gm$ considered above, and on $\TB \times \UB^+ \times (\gg/\ng)^*$, resp.~$\UB^+ \times (\gg/\ng)^*$, via 
\begin{multline*}
x \cdot (t,u,\xi)=(t \cdot \lav_\circ^{-1}(x), \lav_\circ(x) u \lav_\circ^{-1}(x), x^{-2}\lav_\circ(x) \cdot \xi), \\
\text{resp.} \quad x \cdot (u,\xi)=(\lav_\circ(x) u \lav_\circ^{-1}(x), x^{-2}\lav_\circ(x) \cdot \xi).
\end{multline*}
We deduce an isomorphism of $\Gm$-equivariant line bundles
\[
\OC_{\tgg}(\lambda)_{| \UB^+ \times (\gg/\ng)^*} \simto \bigl( (q^\diamond)_* \OC_{\TB \times \UB^+ \times (\gg/\ng)^*} \bigr)^{\TB,-\lambda}.
\]
Now we have $(q^\diamond)_* \OC_{\TB \times \UB^+ \times (\gg/\ng)^*} \cong \FM[\TB] \otimes_\FM \OC_{\UB^+ \times (\gg/\ng)^*}$, and the subsheaf where $\TB$ acts by $-\lambda$ is $(\FM \lambda) \otimes_\FM \OC_{\UB^+ \times (\gg/\ng)^*}$. We deduce~\eqref{eqn:rest-line-bundle}.

Using~\eqref{eqn:rest-line-bundle} we obtain a canonical isomorphism of $\Gm$-equivariant $\OC_{\tg^*}$-modules
\begin{equation}
\label{eqn:kappa-line-bundle}
\kappa(\OC_{\tgg}(\lambda)) \cong \OC_{\tg^*} \langle \lambda(\lav_\circ) \rangle.
\end{equation}
To conclude we have to identify the action of $\Omega(\tg^*/W)$ on $\kappa(\OC_{\tgg}(\lambda))$. 

In order to do so we can restrict to the open subset $\tg^{*,\rs} \subset \tg^*$ (the complement of the coroot hyperplanes). Note that we have $\nu^{-1}(\tg^{*,\rs})=\tgg^\rs$. To compute the action we will use another, more elementary section of the restriction of $\nu$ to $\tgg^\rs$. Namely, we set
\[
\Sigma^\rs := \{1 \BB\} \times \varkappa(\tg^\rs) \subset \tgg^\rs \subset \GB/\BB \times \gg^*,
\]
where $\tg^\rs := \gg^\rs \cap \tg$.
Clearly, $\nu$ restricts to an isomorphism $\Sigma^\rs \simto \tg^{*,\rs}$. Hence, if $\tSC^\rs := \tgg^\rs \cap \tSC$, we have canonical identifications
$\tSC^\rs \simto \tg^{*,\rs} \xleftarrow{\sim} \Sigma^\rs$.
In fact, since the restriction of $\tIB$ to $\tgg^\reg$ is the pullback of the group scheme $\tJB$ on $\tg^*$, we also obtain a canonical identification of the corresponding restrictions of $\tIB$, and then of their Lie algebras:
\begin{equation}
\label{eqn:Lie-tIB-sections}
\vcenter{
\xymatrix@R=0.5cm{
\LM \mathrm{ie}(\tIB_\SC^\rs) \ar[r]^-{\sim} \ar[d] & \LM \mathrm{ie}(\tJB^\rs) \ar[d] & \LM \mathrm{ie}(\tIB_{\Sigma}^\rs) \ar[l]_-{\sim} \ar[d] \\
\tSC^\rs \ar[r]^-{\sim} & \tg^{*,\rs} & \Sigma^\rs. \ar[l]_-{\sim}
}
}
\end{equation}
(Here, for $\AB$ a smooth group scheme over a scheme $X$, $\LM \mathrm{ie}(\AB)$ denotes the vector bundle whose sections are the Lie algebra of $A$, considered as an $\OC_X$-module; the superscript ``$\rs$'' means restriction to the regular semisimple locus, and $\tIB^\rs_\Sigma$ is the restriction of $\tIB$ to $\Sigma^\rs$.)

Using the notation we have just introduced,
Theorem~\ref{thm:Lie-centralizer} defines an isomorphism $\LM \mathrm{ie}(\tIB_\SC^\rs) \simto \tSC^\rs \times_{\SC^\rs }\mathsf{T}^*(\SC^\rs)$. In fact, since the morphism $\tSC^\rs \to \SC^\rs$ is {\'e}tale (because the restriction of $\varrho \colon \tg^* \to \tg^*/W$ to $\tg^{*,\rs}$ is {\'e}tale, see~\cite[Proof of Lemma~3.5.3]{riche3}), this can be rewritten as an isomorphism $\LM \mathrm{ie}(\tIB_\SC^\rs) \simto \mathsf{T}^*(\tSC^\rs)$. Now we clearly have $\tIB_\Sigma^\rs = \TB \times \Sigma^\rs \subset \GB \times \Sigma^\rs$, hence $\LM \mathrm{ie}(\tIB_\Sigma^\rs) = \tg \times \Sigma^\rs \simto \mathsf{T}^*(\Sigma^\rs)$, where we identify $\tg$ with $\varkappa(\tg)^*=\gg / (\ng^+ \oplus \ng)$ in the natural way. We claim that the following diagram commutes, where the horizontal isomorphisms are induced by the identifications in~\eqref{eqn:Lie-tIB-sections} and the vertical isomorphisms are the ones we have just defined:
\begin{equation}
\label{eqn:Lie-tIB-sections-2}
\vcenter{
\xymatrix@R=0.5cm{
\LM \mathrm{ie}(\tIB_\SC^\rs) \ar[r]^-{\sim} \ar[d]_-{\wr} & \LM \mathrm{ie}(\tIB_{\Sigma}^\rs) \ar[d]^-{\wr} \\
\mathsf{T}^*(\tSC^\rs) \ar[r]^-{\sim} & \mathsf{T}^*(\Sigma^\rs).
}
}
\end{equation}
In fact, let $y \in \varkappa(\tg^\rs)$, and let $\tilde{z}$ the point of $\tSC^\rs$ corresponding to $\tilde{y}=(1 \BB,y) \in \Sigma^\rs$ under the identification in~\eqref{eqn:Lie-tIB-sections}. Let us also fix $g \in \GB$ such that $\tilde{z}=g \cdot \tilde{y}$. Then, unravelling the various definitions (see in particular~\cite[Remark~3.3.4]{riche3}), we see that the fiber over $\tilde{z} \leftrightarrow \tilde{y}$ of~\eqref{eqn:Lie-tIB-sections-2} can be described as
\[
\xymatrix@R=0.5cm{
\gg_{g \cdot y} \ar[rrrrr]_-{\mathrm{ad}_{g^{-1}}}^-{\sim} \ar@{^{(}->}[d] &&&&& \gg_y  \ar@{=}[d] \\
\gg \ar@{->>}[d] &&&&& \tg \ar[d]^-{\wr} \\
\varkappa(\sg)^* & \gg \ar@{->>}[l] & [\gg,g \cdot y]^\bot \ar@{_{(}->}[l] \ar[r]^-{\mathrm{ad}_{g^{-1}}}_-{\sim} \ar@/_15pt/@{-->}[ll]_-{\sim} & [\gg, y]^\bot \ar@{^{(}->}[r] \ar@/^15pt/@{-->}[rr]^-{\sim} & \gg \ar@{->>}[r] & \varkappa(\tg)^*.
}
\]
Now we have $\gg_{g \cdot y} = [\gg, g \cdot y]^\bot$ and $\gg_y=[\gg, y]^\bot$, so the commutativity is obvious.

With this comparison at hand, we can now identify the restriction of $\kappa(\OC_{\tgg}(\lambda))$ to $\tg^{*,\rs}$ with $\OC_{\tgg}(\lambda)_{| \Sigma^\rs}$, endowed with the action of $\IG_\Sigma^\rs$, identified with $\tg \otimes \OC_{\Sigma^\rs}$ as above. But the restriction of $\OC_{\tgg}(\lambda)$ to $\{1 \BB\} \times (\gg/\ng)^*$ is $\OC_{(\gg/\ng)^*} \otimes_\FM \FM_\lambda$ (where $\FM_\lambda$ is the one-dimensional $\BB$-module defined by $\lambda$), and we finally deduce that~\eqref{eqn:kappa-line-bundle} defines an isomorphism $\kappa(\OC_{\tgg}(\lambda)) \cong F_\lambda \langle \lambda({\check \lambda}_\circ) \rangle$, as desired.
\end{proof}

\subsection{Kostant--Whittaker reduction and geometric actions}
\label{ss:KW-braid-group}

The main result of this subsection is the following.

\begin{prop}
\label{prop:KW-actions}
\begin{enumerate}
\item
\label{it:KW-actions-F}
For any $b \in \Baff$, there exists an isomorphism of functors
\[
\KS^\FM_b \circ \kappa_\FM \cong \kappa_\FM \circ \IS^\FM_b.
\]
\item
\label{it:KW-actions-R}
Let $b \in \Baff$, and let $\FC$ be an object of $D^{\GB \times \Gm}(\tgg)_\RG$ such that $\kappa_\RG(\FC)$ is concentrated in degree $0$, and $\RG$-free. Then $\kappa_\RG \circ \IS^\RG_b(\FC)$ is concentrated in degree $0$ and $\RG$-free, and moreover there exists an isomorphism
\[
\KS^\RG_b \circ \kappa_\RG(\FC) \cong \kappa_\RG \circ \IS^\RG_b(\FC)
\]
in $D^{\Gm}(\tg^* \times_{\tg^* / W} \TF(\tg^*/W))_\RG$.
\end{enumerate}
\end{prop}

\begin{remark}
It is probably true that there exists an isomorphism of functors $\KS^\RG_b \circ \kappa_\RG \cong \kappa_\RG \circ \IS^\RG_b$. The weaker statement in Proposition~\ref{prop:KW-actions}\eqref{it:KW-actions-R} will be sufficient for our purposes.
\end{remark}

Before proving this result we need a preliminary lemma. In this lemma we fix a finite simple reflection $s$, and consider the morphism $Z_s^\EM \to \tgg_\EM$ induced by the second projection.

\begin{lem}
\label{lem:fiber-product-Zs}
\begin{enumerate}
\item
The Cartesian square
\[
\xymatrix@R=0.5cm{
(\tSC \times_{\tgg} Z_s)_\EM \ar@{^{(}->}[r] \ar[d] & Z_s^\EM \ar[d] \\
\tSC_\EM  \ar@{^{(}->}[r] & \tgg_\EM
}
\]
is tor-independent in the sense of~\cite[Definition~3.10.2]{lipman}.
\label{it:square-tor-indep}
\item
The embedding $Z_s^\EM \hookrightarrow (\tgg \times \tgg)_\EM$ identifies $(\tSC \times_{\tgg} Z_s)_\EM$ with a closed subscheme of $(\tSC \times \tSC)_\EM$. Then,
identifying $\tSC_\EM$ with $\tg^*_\EM$ via the isomorphism of Proposition~{\rm \ref{prop:Kostant-section}}, $(\tSC \times_{\tgg} Z_s)_\EM$ identifies with the graph of the action of $s$ on $\tg^*_\EM$.
\label{it:Zs-graph}
\end{enumerate}
\end{lem}

\begin{proof}
\eqref{it:square-tor-indep}
We decompose our Cartesian square into two squares:
\[
\xymatrix@R=0.5cm{
\bigl(\tSC \times_{\tgg} Z_s\bigr)_\EM \ar@{^{(}->}[r] \ar[d] & \GB \times \bigl( \tSC \times_{\tgg} Z_s \bigr)_\EM \ar[r] \ar[d] & Z_s^\EM \ar[d] \\
\tSC_\EM  \ar@{^{(}->}[r] & (\GB \times \tSC)_\EM \ar[r]^-{a} & \tgg_\EM.
}
\]
Here the left arrows are induced by the embedding $\Spec(\EM) \hookrightarrow \GB_\EM$ given by the identity. The left-hand hand square is obviously Cartesian and tor-independent (since $\GB_\EM$ is flat over $\Spec(\EM)$), and the right-hand square is Cartesian by $\GB_\EM$-equivariance, and tor-independent since $a$ is flat (see Lemma~\ref{lem:a-smooth}). The claim follows.

\eqref{it:Zs-graph}
The first claim is a consequence of the definition of $\tSC_\EM$ and of the fact that $Z_s^\EM$ is included in $(\tgg \times_{\gg^*} \tgg)_\EM$. Now we consider the second claim.

In the case $\EM=\FM$, the claim follows from the definition of $Z_s^\FM$, using the facts that $\tSC_\FM$ is included in $\tgg_\FM^\reg$ (see~\cite[Equation~(3.1.1)]{riche3}), and that the morphism $\nu_\reg$ is $W$-equivariant (see~\S\ref{ss:notation-KW}).

Now, let us deduce the case $\EM=\RG$. It follows in particular from the first claim that $(\tSC \times_{\tgg} Z_s)_\RG$ is an affine scheme. Since $\GB_\RG \times (\tSC \times_{\tgg} Z_s)_\RG$ is flat over $Z_s^\RG$ (see the proof of~\eqref{it:square-tor-indep}), which is itself flat over $\RG$, and since $\RG$ is a direct summand in $\OC(\GB_\RG)$, the scheme $(\tSC \times_{\tgg} Z_s)_\RG$ is also flat over $\RG$.

Let us denote by $f$ the morphism $(\tSC \times_{\tgg} Z_s)_\RG \to \tSC_\RG$ induced by the first projection. We claim that $f$ is an isomorphism. In fact, consider the morphism
\[
f^* \colon \OC(\tSC_\RG) \to \OC(\tSC \times_{\tgg} Z_s)_\RG.
\]
Since $f$ is projective, the right-hand side is finite over $\OC(\tSC_\RG)$. Moreover, the morphism $\FM \otimes_\RG (f^*)$ is an isomorphism for any geometric point $\FM$ of $\RG$, by the case of fields treated first. Using~\cite[Lemma~1.4.1]{br} we deduce that $f^*$ is an isomorphism, which finishes the proof of the claim.

Now, consider the morphism $\tau_s \colon \tSC_\RG \to \tSC_\RG$ defined as the composition of the inverse of $f$ with the natural projection $(\tSC \times_{\tgg} Z_s)_\RG \to \tSC_\RG$. 
The induced morphism $\tSC_\CM \to \tSC_\CM$ coincides with the action of $s$ (via the identification $\tSC_\CM \simto \tg^*_\CM$). It follows that the morphism $\tau_s$ itself coincides with the action of $s$, finishing the proof.
\end{proof}

\begin{proof}[Proof of Proposition~{\rm \ref{prop:KW-actions}}]
In each case, it is sufficient to prove the claim when $b=T_s$ for $s$ a finite simple reflection, or when $b=\theta_\lambda$ for some $\lambda \in \XB$.

First, assume that $b=\theta_\lambda$. Since $\kappa_\EM$ is compatible with tensoring with a line bundle, Proposition~\ref{prop:KW-line-bundle} implies that there exists an isomorphism of functors $\KS^\EM_{\theta_\lambda} \circ \kappa_\EM \cong \kappa_\EM \circ \IS^\EM_{\theta_\lambda}$, which proves the claims in~\eqref{it:KW-actions-F} and~\eqref{it:KW-actions-R} in this case.

Now we consider the case $b=T_s$. In this case we have $\IS^\EM_{T_s} \cong \Rder (p_s)_* \circ \Lder (q_s)^* \langle -1 \rangle$, where $p_s, q_s \colon Z_s^\EM \to \tgg_\EM$ are induced by the first and second projections, respectively. Let $\kappa'_\EM$ be the composition of $\kappa_\EM$ with the direct image under the (affine) morphism $\varsigma \colon (\tg^* \times_{\tg^* / W} \TF(\tg^*/W))_\EM \to \tg^*_\EM$. (In other words, $\kappa'_\EM$ is the composition of restriction to $\tSC_\EM$ with the functor $(\nu_\SC)_*$.)
By Lemma~\ref{lem:fiber-product-Zs}\eqref{it:square-tor-indep} and the base-change theorem (see~\cite[Theorem~3.10.3]{lipman}), we have
$\kappa'_\EM \circ \IS^\EM_{T_s} \cong \Rder (p'_s)_* \circ \Lder (q'_s)^* \circ \kappa'_\EM \langle -1 \rangle$, where 
$p_s', q_s' \colon (\tSC \times_{\tgg} Z_s)_\EM \to \tg_\EM^*$
are the compositions of $\nu_\SC$ with the morphisms obtained by restriction from $p_s, q_s$. Using Lemma~\ref{lem:fiber-product-Zs}\eqref{it:Zs-graph}, we deduce a canonical isomorphism
\begin{equation}
\label{eqn:isom-KW-action-Ts}
\kappa'_\EM \circ \IS^\EM_{T_s} \cong (\tau'_s)^* \circ \kappa'_\EM \langle -1 \rangle,
\end{equation}
where $\tau'_s \colon \tg^*_\FM \to \tg^*_\FM$ is the action of $s$.

In the case $\EM=\FM$, since the automorphism $\tau_s \colon \tSC_\FM \simto \tSC_\FM$ is the restriction of a $\GB_\FM \times \GmF$-equivariant automorphism of $\tgg_\FM^\reg$, isomorphism~\eqref{eqn:isom-KW-action-Ts} is induced by an isomorphism of functors
\[
\KS^\FM_{T_s} \circ \kappa_\FM \cong \kappa_\FM \circ \IS^\FM_{T_s},
\]
which finishes the proof of~\eqref{it:KW-actions-F}. Then the case $\EM=\RG$ follows from the case $\EM=\CM$: indeed by~\eqref{eqn:isom-KW-action-Ts} we have an isomorphism
\begin{equation}
\label{eqn:isom-KW-action-Ts-R}
\varsigma_* \bigl( \KS^\RG_{T_s} \circ \kappa_\RG(\FC) \bigr) \cong \varsigma_* \bigl( \kappa_\RG \circ \IS^\RG_{T_s}(\FC) \bigr).
\end{equation}
Hence if $\kappa_\RG (\FC)$ is concentrated in degree $0$ and $\RG$-free, the same is true for $\kappa_\RG \circ \IS^\RG_{T_s}(\FC)$. And in this case, since the image of~\eqref{eqn:isom-KW-action-Ts-R} under $\CM \otimes_\RG (-)$ is $\OC(\tg^* \times_{\tg^* / W} \TF(\tg^*/W))_\CM$-linear, we deduce that~\eqref{eqn:isom-KW-action-Ts-R} is $\OC(\tg^* \times_{\tg^* / W} \TF(\tg^*/W))_\RG$-linear, which finishes the proof of~\eqref{it:KW-actions-R}.
\end{proof}

We finish this section with the following variant of Lemma~\ref{lem:fiber-product-Zs}, to be used later. In this lemma we consider the morphism $(\tgg \times_{\tgg_s} \tgg)_\EM \to \tgg_\EM$ induced by the second projection.

\begin{lem}
\label{lem:fiber-product-xis}
\begin{enumerate}
\item
\label{it:square-tor-indep-xi}
The Cartesian square
\[
\xymatrix@R=0.5cm{
\bigl( \tSC \times_{\tgg} (\tgg \times_{\tgg_s} \tgg) \bigr)_\EM \ar@{^{(}->}[r] \ar[d] & (\tgg \times_{\tgg_s} \tgg)_\EM \ar[d] \\
\tSC_\EM  \ar@{^{(}->}[r] & \tgg_\EM
}
\]
is tor-independent in the sense of~\cite[Definition~3.10.2]{lipman}.
\item
\label{it:xis-graph}
Identifying $\tSC_\EM$ with $\tg^*_\EM$ via the isomorphism of Proposition~{\rm \ref{prop:Kostant-section}}, the fiber product $\bigl( \tSC \times_{\tgg} (\tgg \times_{\tgg_s} \tgg) \bigr)_\EM$ identifies with the closed subscheme of $(\tg^* \times \tg^*)_\EM$ given by $(\tg^* \times_{\tg^* / W_s} \tg^*)_\EM$.
\end{enumerate}
\end{lem}

\begin{proof}
The proof of~\eqref{it:square-tor-indep-xi} is identical to the proof of Lemma~\ref{lem:fiber-product-Zs}\eqref{it:square-tor-indep}. The case $\EM=\FM$ of~\eqref{it:xis-graph} follows from the observation that the intersection of $(\tgg \times_{\tgg_s} \tgg)_\FM$ with $\tgg_\FM^\reg \times \tgg_\FM^\reg$ is the union of the diagonal copy of $\tgg_\FM^\reg$ and the graph of the action of $s$. Now, let us deduce the case $\EM=\RG$.

By the same arguments as in the proof of Lemma~\ref{lem:fiber-product-Zs}\eqref{it:Zs-graph}, the fiber product $\bigl( \tSC \times_{\tgg} (\tgg \times_{\tgg_s} \tgg) \bigr)_\RG$ is a flat $\RG$-scheme, and a closed subscheme of the affine scheme $(\tSC \times \tSC)_\RG$. Consider the algebra 
\[
A:= \OC \bigl( \tSC \times_{\tgg} (\tgg \times_{\tgg_s} \tgg) \bigr)_\RG.
\]
We claim that $A$ is $\RG$-free. Indeed, consider the contracting $\GmR$-action on $\tSC_\RG$ considered in~\S\ref{ss:reminder}. The diagonal action on $(\tSC \times \tSC)_\RG$ stabilizes the subscheme $(\tSC \times_{\tgg} (\tgg \times_{\tgg_s} \tgg))_\RG$, so that $A$ is endowed with a $\ZM$-grading. Each graded piece of $A$ is finite over $\RG$, and flat, hence free, which proves our claim.

Now, consider the surjection
\[
\OC(\tg^* \times \tg^*)_\RG \twoheadrightarrow A.
\]
For any $x \in \OC(\tg^*_\RG)^s$, the image of $x \otimes 1 - 1 \otimes x$ in $A$ becomes zero in $\CM \otimes_\RG A$,
by the case $\EM=\CM$. Since $A$ is $\RG$-free, this implies that this image is $0$, hence that our morphism factors through a surjection
\[
\bigl( \OC(\tg^*) \otimes_{\OC(\tg^*)^s} \OC(\tg^*) \bigr)_\RG \twoheadrightarrow A.
\]
This surjection is compatible with the $\ZM$-gradings, and on both sides the graded pieces are free of finite rank over $\RG$. (In fact, this property has been proved above for $A$. For the left-hand side, we have $\OC(\tg^*_\RG)=\OC(\tg^*_\RG)^s \oplus \OC(\tg^*_\RG)^s \cdot \delta$, where $\delta \in \tg_\RG$ is any element such that $\langle \delta, \alpha \rangle=1$ -- see e.g.~\cite[Claim~3.9]{ew} --
which implies our claim.) Hence to conclude we only have to prove that the induced morphism
\[
\FM \otimes_\RG \bigl( \OC(\tg^*) \otimes_{\OC(\tg^*)^s} \OC(\tg^*) \bigr)_\RG \twoheadrightarrow \FM \otimes_\RG A
\]
is an isomorphism for all geometric points $\FM$ of $\RG$. 

Using again the case $\EM=\FM$ treated above, we only have to prove that the natural morphism
\[
\FM \otimes_\RG \bigl( \OC(\tg^*) \otimes_{\OC(\tg^*)^s} \OC(\tg^*) \bigr)_\RG \to \bigl( \OC(\tg^*) \otimes_{\OC(\tg^*)^s} \OC(\tg^*) \bigr)_\FM
\]
is an isomorphism. In turn, this follows easily from the decompositions $\OC(\tg^*_\EM)=\OC(\tg^*_\EM)^s \oplus \OC(\tg^*_\EM)^s \cdot \delta$ for $\EM=\RG$ and $\FM$ (where $\delta$ is as above, and we denote similarly its image in $\tg_\FM$).
\end{proof}

\section{Tilting exotic sheaves}
\label{sec:tilting-KW}

In this section we use the same notation as in Section~\ref{sec:KW-reduction}.

\subsection{Overview}
\label{ss:overview-exotic}

In this section we give a description of the category of tilting objects in $\ES^{\GB \times \Gm}(\tNC)$ in terms of ``Soergel bimodules.'' This description is based on a ``Bott--Samelson type'' description of these tilting objects, due to Dodd~\cite{dodd} in the case $p=0$ and generalized to the modular setting in~\cite{mr}, and on the use of the ``Kostant--Whittaker reduction'' of Section~\ref{sec:KW-reduction}.

In~\S\ref{ss:notation-KW-tilting} we review the basic definitions and results on Bezrukavnikov's exotic t-structure, and in~\S\ref{ss:standard-KW} we introduce and study some variants of the associated ``standard'' and ``costandard'' objects. In~\S\ref{ss:BS-objects-geom} we recall the ``Bott--Samelson'' description of tilting objects in $\ES^{\GB \times \Gm}(\tNC)$. In~\S\ref{ss:grk-Hom-tilting} we compute the graded ranks of $\Hom$-spaces between our ``Bott--Samelson objects.'' Finally, in~\S\S\ref{ss:BS-category-tilting}--\ref{ss:equivalence-tilting} we obtain the desired description in terms of Soergel bimodules.

\subsection{Reminder on the exotic t-structure}
\label{ss:notation-KW-tilting}

In this section we will
consider the Springer resolution 
\[
\tNC_\EM:=(\GB \times^{\BB} (\gg/\bg)^*)_\EM \hookrightarrow \Flag_\EM \times \gg^*_\EM,
\]
a sub-vector bundle of the Grothendieck resolution $\tgg_\EM$ studied in Section~\ref{sec:KW-reduction}. We denote by $i \colon \tNC_\EM \hookrightarrow \tgg_\EM$ the inclusion. For $\lambda \in \XB$ we denote by $\OC_{\tNC_\EM}(\lambda)$ the restriction of $\OC_{\tgg_\EM}(\lambda)$ to $\tNC_\EM$.
We will consider the derived categories of equivariant coherent sheaves
$D^{\GB \times \Gm}(\tNC)_\EM$ and $D^{\GB}(\tNC)_\EM$.

By~\cite[Section~1]{br}, the geometric braid group actions considered in~\S\ref{ss:braid-groth} ``restrict'' to the categories $D^{\GB \times \Gm}(\tNC)_\EM$ and $D^{\GB}(\tNC)_\EM$ in the following sense. For $s$ a finite simple reflection, associated with a simple root $\alpha$, we define $Z_s^{\prime \EM}:= Z_s^\EM \cap (\tNC_\EM \times \tNC_\EM)$, and denote by
\[
\TM^s_{\tNC}, \, \SM^s_{\tNC} \colon D^{\GB \times \Gm}(\tNC)_\EM \to D^{\GB \times \Gm}(\tNC)_\EM
\]
the Fourier--Mukai transforms with kernels $\OC_{Z_s^{\prime \EM}} \langle -1 \rangle$ and $\OC_{Z_s^{\prime \EM}}(-\rho, \rho-\alpha) \langle -1 \rangle$ respectively. We use the same symbols for the analogous endofunctors of $D^{\GB}(\tNC)_\EM$. (Here $\OC_{Z_s^{\prime \EM}}(-\rho, \rho-\alpha)$ is defined by the same recipe as for $\OC_{Z_s^{\EM}}(-\rho, \rho-\alpha)$ in~\S\ref{ss:braid-groth}.) Then $\TM^s_{\tNC}$ and $\SM^s_{\tNC}$ are quasi-inverse equivalences of categories, and there exists a right action of the group $\Baff$ on the categories $D^{\GB \times \Gm}(\tNC)_\EM$ and $D^{\GB}(\tNC)_\EM$ such that $T_s$ acts by $\TM^s_{\tNC}$ for any finite simple reflection $s$, and $\theta_\lambda$ acts by tensoring with $\OC_{\tNC_\EM}(\lambda)$ for any $\lambda \in \XB$.

For $b \in \Baff$, we denote by
\[
\JS^\EM_b \colon D^{\GB \times \Gm}(\tNC)_\EM \simto D^{\GB \times \Gm}(\tNC)_\EM
\]
the action of $b$. Then there exist isomorphisms of functors
\begin{equation}
\label{eqn:action-i*}
\IS_b^\EM \circ \Rder i_* \cong \Rder i_* \circ \JS_b^\EM, \qquad \Lder i^* \circ \IS_b^\EM \cong \JS_b^\EM \circ \Lder i^*
\end{equation}
and
\begin{equation}
\label{eqn:action-modular-reduction}
\FM \circ \JS_b^\RG \cong \JS_b^\FM \circ \FM.
\end{equation}

Recall the elements $w_\lambda \in \Waff$ defined in~\S\ref{ss:Haff}. We set
\[
\nabla^\lambda_{\tNC,\EM} := \JS^\EM_{T_{w_\lambda}}(\OC_{\tNC_\EM}), \qquad
\Delta^\lambda_{\tNC,\EM} := \JS^\EM_{(T_{w_\lambda^{-1}})^{-1}}(\OC_{\tNC_\EM}).
\]
By~\eqref{eqn:action-modular-reduction}, these objects satisfy
\begin{equation}
\label{eqn:F-D-N-tNC}
\FM(\nabla^\lambda_{\tNC,\RG}) \cong \nabla^\lambda_{\tNC,\FM}, \quad \FM(\Delta^\lambda_{\tNC,\RG}) \cong \Delta^\lambda_{\tNC,\FM}.
\end{equation}

\begin{remark}
\label{rk:nabla-delta-Baff}
For any $w \in W$ we have $\JS^\EM_{T_w}(\OC_{\tNC_\EM}) \cong \OC_{\tNC_\EM} \langle -\ell(w) \rangle$. (See~\cite[Equation~(3.10)]{mr} for the case $\EM=\FM$; the case $\EM=\RG$ can be deduced using the arguments in the proof of~\cite[Proposition~1.4.3]{br}, or proved along the same lines.) Therefore, as in~\cite{mr}, for any $w \in Wt_\lambda$ we have
\[
\nabla^\lambda_{\tNC,\EM} \cong \JS^\EM_{T_w}(\OC_{\tNC_\EM}) \langle -\ell(w_\lambda) + \ell(w) \rangle, \quad
\Delta^\lambda_{\tNC,\EM} \cong \JS^\EM_{(T_{w^{-1}})^{-1}}(\OC_{\tNC_\EM}) \langle - \ell(w) + \ell(w_\lambda) \rangle.
\]
\end{remark}

In the case $\EM=\FM$,
the objects $\nabla^\lambda_{\tNC,\FM}$ and $\Delta^\lambda_{\tNC,\FM}$ were studied in~\cite{mr} (see in particular~\cite[Proposition~3.7]{mr}). In fact, if we denote by $D^{\leq 0}$, resp.~$D^{\geq 0}$, the subcategory of $D^{\GB \times \Gm}(\tNC)_\FM$ generated under extensions by the objects $\Delta^\lambda_{\tNC,\FM} \langle n \rangle [m]$ with $n \in \ZM$ and $m \in \ZM_{\geq 0}$, resp.~by the objects $\nabla^\lambda_{\tNC,\FM} \langle n \rangle [m]$ with $n \in \ZM$ and $m \in \ZM_{\leq 0}$, then the pair $(D^{\leq 0}, D^{\geq 0})$ constitutes a bounded t-structure on $D^{\GB \times \Gm}(\tNC)_\FM$, called the \emph{exotic t-structure}. We denote by $\ES^{\GB \times \Gm}(\tNC)_\FM$ the heart of this t-structure.

By~\cite[Corollary~3.10]{mr}, the objects $\Delta^\lambda_{\tNC,\FM}$ and $\nabla^\lambda_{\tNC,\FM}$ belong to $\ES^{\GB \times \Gm}(\tNC)_\FM$. Let us fix an order $\leq'$ on $\XB$ as in~\cite[\S 2.5]{mr}. Then by~\cite[\S 3.5]{mr}, the category $\ES^{\GB \times \Gm}(\tNC)_\FM$ is a graded highest weight category with weight poset $(\XB, \leq')$, standard objects $\{\Delta^\lambda_{\tNC,\FM}, \, \lambda \in \XB\}$ and costandard objects $\{\nabla^\lambda_{\tNC, \FM}, \, \lambda \in \XB\}$. In particular, it makes sense to consider the \emph{tilting objects} in $\ES^{\GB \times \Gm}(\tNC)_\FM$, i.e.~the objects which admit both a standard filtration (i.e.~a filtration with subquotients of the form $\Delta^\lambda_{\tNC, \FM} \langle m \rangle$ with $\lambda \in \XB$ and $m \in \ZM$), 
and a costandard filtration (i.e.~a filtration with subquotients of the form $\nabla^\lambda_{\tNC, \FM} \langle m \rangle$ with $\lambda \in \XB$ and $m \in \ZM$).
The subcategory consisting of such objects will be denoted by $\Tilt(\ES^{\GB \times \Gm}(\tNC)_\FM)$.

If $X$ admits a standard filtration, resp.~a costandard filtration, we denote by $(X : \Delta^\lambda_{\tNC, \FM} \langle m \rangle)$, resp.~$(X : \nabla^\lambda_{\tNC, \FM} \langle m \rangle)$, the number of times $\Delta^\lambda_{\tNC, \FM} \langle m \rangle$, resp.~$\nabla^\lambda_{\tNC, \FM} \langle m \rangle$, appears in a standard, resp.~costandard, filtration of $X$. (This number does not depend on the filtration.) The general theory of graded highest weight categories implies that for any $\lambda \in \XB$ there exists a unique (up to isomorphism) indecomposable object $\TC^\lambda$ in $\Tilt(\ES^{\GB \times \Gm}(\tNC)_\FM)$ which satisfies 
\[
(\TC^\lambda : \Delta_{\tNC, \FM}^\lambda)=1 \quad \text{and} \qquad (\TC^\lambda : \Delta^\mu_{\tNC, \FM} \langle m \rangle) \neq 0 \Rightarrow \mu \leq' \lambda.
\]
Moreover, every object in $\Tilt(\ES^{\GB \times \Gm}(\tNC)_\FM)$ is a direct sum of objects of the form $\TC^\lambda \langle m \rangle$ for some $\lambda \in \XB$ and $m \in \ZM$. (See e.g.~\cite[Appendix~A]{modrap2} for references on this subject.)

Recall also from~\cite[Equation~(2.3)]{mr} that we have
\begin{equation}
\label{eqn:Hom-Delta-nabla}
\Hom_{D^{\GB \times \Gm}(\tNC)_\FM}(\Delta^\lambda_{\tNC,\FM}, \nabla^\mu_{\tNC,\FM} \langle n \rangle [m]) =
\begin{cases}
\FM & \text{if $\lambda=\mu$ and $n=m=0$;} \\
0 & \text{otherwise.}
\end{cases}
\end{equation}

\begin{remark}
One can easily check that none of our constructions depends on the choice of the order $\leq'$.
\end{remark}

\subsection{More ``standard'' and ``costandard'' objects}
\label{ss:standard-KW}

If $\HB$ is an affine group scheme over a ring $k$, we denote by $\Inv^\HB_k$ the functor of derived $\HB$-invariants; see~\cite[\S A.3]{mr}.

\begin{lem}
\label{lem:morphisms-fg}
For any $\FC,\GC$ in $D^{\GB \times \Gm}(\tNC)_\RG$, 
the complex of $\RG$-modules
\[
R\Hom_{D^{\GB \times \Gm}(\tNC)_\RG}(\FC,\GC)
\]
is bounded, and has finitely generated cohomology modules.
\end{lem}

\begin{proof}
By~\cite[Proposition~A.6]{mr},
we have a natural isomorphism
\[
R\Hom_{D^{\GB \times \Gm}(\tNC)_\RG}(\FC,\GC) \cong \Inv^{\GB_\RG}_\RG \circ \Inv^{\GmR}_{\RG} \bigl( R\Hom_{\Db \Coh(\tNC_\RG)}(\FC,\GC) \bigr).
\]
Since the natural morphism $\tNC_\RG \to \gg^*_\RG$ (the restriction of the morphism $\pi$ from~\S\ref{ss:notation-KW}) is projective, by~\cite[Theorem~III.8.8]{hartshorne} $R\Hom_{\Db \Coh(\tNC_\RG)}(\FC,\GC)$ is a bounded complex of finitely generated $\OC(\gg^*_\RG)$-modules, which implies that the complex of $\RG$-modules $\Inv^{\GmR}_{\RG} ( R\Hom_{\Db \Coh(\tNC_\RG)}(\FC,\GC) )$ is bounded, and has finitely generated cohomology modules. Hence the claim follows from the fact that if $M$ is a $\GB_\RG$-module which is finitely generated over $\RG$, then $\Inv^{\GB_\RG}_\RG(M)$ is a bounded complex of finitely generated $\RG$-modules, see~\cite[Lemma~II.B.5 and its proof]{jantzen}.
\end{proof}

\begin{prop}
\label{prop:morphism-D-N-tNC}
For $\lambda,\mu \in \XB$ and $n,m \in \ZM$
we have
\[
\Hom_{D^{\GB \times \Gm}(\tNC)_\RG}(\Delta^\lambda_{\tNC,\RG}, \nabla^{\mu}_{\tNC,\RG} \langle n \rangle [m]) = \begin{cases}
\RG & \text{if $\lambda=\mu$ and $n=m=0$;} \\
0 & \text{otherwise.}
\end{cases}
\]
Moreover, the natural morphism
\[
\FM \otimes_\RG \Hom_{D^{\GB \times \Gm}(\tNC)_\RG}(\Delta^\lambda_{\tNC,\RG}, \nabla^{\mu}_{\tNC,\RG} \langle n \rangle [m])
\to \Hom_{D^{\GB \times \Gm}(\tNC)_\FM}(\Delta^\lambda_{\tNC,\FM}, \nabla^{\mu}_{\tNC,\FM} \langle n \rangle [m])
\]
is an isomorphism.
\end{prop}

\begin{proof}
By the same arguments as in~\cite[Proof of Lemma~4.11]{mr}, there exists a canonical isomorphism of complexes of $\FM$-vector spaces
\[
\FM \lotimes_\RG R\Hom_{D^{\GB \times \Gm}(\tNC)_\RG}(\Delta^\lambda_{\tNC,\RG}, \nabla^{\mu}_{\tNC,\RG} \langle n \rangle)
\simto R\Hom_{D^{\GB \times \Gm}(\tNC)_\FM}(\Delta^\lambda_{\tNC,\FM}, \nabla^{\mu}_{\tNC,\FM} \langle n \rangle).
\]
By~\eqref{eqn:Hom-Delta-nabla}, the right-hand side is isomorphic to $\FM$ if $\lambda=\mu$ and $n=0$, and is $0$ otherwise. Using Lemma~\ref{lem:morphisms-fg}, we deduce the claim.
\end{proof}

As in~\cite[\S 4.1]{mr}, for $\lambda \in \XB$, we set
\[
\nabla^\lambda_{\tgg,\EM} := \IS^\EM_{T_{w_\lambda}}(\OC_{\tgg_\EM}), \qquad
\Delta^\lambda_{\tgg,\EM} := \IS^\EM_{(T_{w_\lambda^{-1}})^{-1}}(\OC_{\tgg_\EM}).
\]
Then by~\eqref{eqn:F-IS} we have
\[
\FM(\nabla^\lambda_{\tgg,\RG}) \cong \nabla^\lambda_{\tgg,\FM}, \qquad \FM(\Delta^\lambda_{\tgg,\RG}) \cong \Delta^\lambda_{\tgg,\FM},
\]
and by~\eqref{eqn:action-i*} we have
\begin{equation}
\label{eqn:delta-nabla-i^*}
\Lder i^*(\nabla^\lambda_{\tgg,\EM}) \cong \nabla^\lambda_{\tNC,\EM}, \qquad \Lder i^*(\Delta^\lambda_{\tgg,\EM}) \cong \Delta^\lambda_{\tNC,\EM}.
\end{equation}

Recall that for $\FC,\GC$ in $D^{\GB \times \Gm}(\tgg)_\EM$ the morphism
\[
\bigoplus_{n \in \ZM} \Hom_{D^{\GB \times \Gm}(\tgg)_\EM}(\FC, \GC \langle -n \rangle) \to \Hom_{D^{\GB}(\tgg)_\EM}(\FC,\GC)
\]
induced by the forgetful functor
is an isomorphism. In particular, this isomorphism endows the right-hand side with a natural $\ZM$-grading.

\begin{prop}
\label{prop:morphism-D-N-tgg}
For $\lambda,\mu \in \XB$ and $m \in \ZM$ there exists an isomorphism of graded $\OC(\tg^*_\EM)$-modules
\[
\Hom_{D^{\GB}(\tgg)_\EM}(\Delta^\lambda_{\tgg,\EM}, \nabla^{\mu}_{\tgg,\EM} [m]) = \begin{cases}
\OC(\tg^*_\EM) & \text{if $\lambda=\mu$ and $m=0$;} \\
0 & \text{otherwise.}
\end{cases}
\]
Moreover, the natural morphisms
\[
\EM \otimes_{\OC(\tg^*_\EM)} \Hom_{D^{\GB}(\tgg)_\EM}(\Delta^\lambda_{\tgg,\EM}, \nabla^{\mu}_{\tgg,\EM} [m])
\to \Hom_{D^{\GB}(\tNC)_\EM}(\Delta^\lambda_{\tNC,\EM}, \nabla^{\mu}_{\tNC,\EM} [m])
\]
(induced by $\Lder i^*$ via isomorphisms~\eqref{eqn:delta-nabla-i^*})
and
\[
\FM \otimes_\RG \Hom_{D^{\GB}(\tgg)_\RG}(\Delta^\lambda_{\tgg,\RG}, \nabla^{\mu}_{\tgg,\RG} [m])
\to \Hom_{D^{\GB}(\tgg)_\FM}(\Delta^\lambda_{\tgg,\FM}, \nabla^{\mu}_{\tgg,\FM} [m])
\]
are isomorphisms.
\end{prop}

\begin{proof}
As in~\cite[Lemma~4.11]{mr},
for $\FC,\GC$ in $D^{\GB}(\tgg)_\EM$
we have a canonical isomorphism
\[
\EM \lotimes_{\OC(\tg_\EM^*)} R\Hom_{D^{\GB}(\tgg)_\EM}(\FC,\GC) \cong R\Hom_{D^{\GB}(\tNC)_\EM}(\Lder i^* \FC, \Lder i^* \GC).
\]
Choosing $\FC=\Delta^\lambda_{\tgg,\EM}$ and $\GC=\nabla^{\mu}_{\tgg,\EM}$ and using~\eqref{eqn:delta-nabla-i^*}, we deduce an isomorphism
\[
\EM \lotimes_{\OC(\tg_\EM^*)} R\Hom_{D^{\GB}(\tgg)_\EM}(\Delta^\lambda_{\tgg,\EM}, \nabla^{\mu}_{\tgg,\EM}) \cong R\Hom_{D^{\GB}(\tNC)_\EM}(\Delta^\lambda_{\tNC,\EM}, \nabla^{\mu}_{\tNC,\EM}).
\]
Using Proposition~\ref{prop:morphism-D-N-tNC}, the right-hand side is concentrated in degree $0$, and isomorphic either to $\EM$ (when $\lambda=\mu$) or to $0$ (when $\lambda \neq \mu$). By Lemma~\ref{lem:nakayama}\eqref{it:nakayama-complex}, this implies the first claim, and the first isomorphism.

The proof of the second isomorphism is similar to the proof of the corresponding claim in Proposition~\ref{prop:morphism-D-N-tNC}.
\end{proof}

\subsection{``Coherent'' Bott--Samelson objects}
\label{ss:BS-objects-geom}

As in~\cite[\S 4.2]{mr},
if $s$ is a finite simple reflection,
we denote by 
\[
\Xi_s^\EM \colon D^{\GB \times \Gm}(\tgg)_\EM \to D^{\GB \times \Gm}(\tgg)_\EM
\]
the Fourier--Mukai transform associated with the kernel $\OC_{(\tgg \times_{\tgg_s} \tgg)_\EM} \langle -1 \rangle$ (see~\S\ref{ss:braid-groth} for the notation). We also make a similar definition for $s_0$ an affine simple reflection: by Lemma~\ref{lem:affine-simple-conjugate} we can choose (once and for all) a finite simple reflection $t$ and an element $b \in \Baff$ such that $T_{s_0} = b T_t b^{-1}$, and set $\Xi^\EM_{s_0}:=\IS^\EM_{b^{-1}} \circ \Xi^\EM_t \circ \IS^\EM_b$. 
One can easily check that for any simple reflection $s \in \Waff$ there exists an isomorphism
\begin{equation}
\label{eqn:F-Xi}
\FM \circ \Xi_s^\RG \cong \Xi_s^\FM \circ \FM.
\end{equation}

For $\us=(s_1, \cdots, s_n)$ a sequence of simple reflections and $\omega \in \Omega$, we set
\[
\MC_{\EM}(\omega,\us):=  \Xi^\EM_{s_n} \circ \cdots \circ \Xi^\EM_{s_1} \circ \IS^\EM_{T_\omega} (\OC_{\tgg_\EM}).
\]
By~\eqref{eqn:F-IS} and~\eqref{eqn:F-Xi}, these objects satisfy
\[
\FM(\MC_{\RG}(\omega,\us)) \cong \MC_{\FM}(\omega,\us).
\]

Our interest in the objects $\MC_{\EM}(\omega,\us)$ comes from the following result, which
is proved in~\cite[Corollary~4.2]{mr}.

\begin{prop}
\label{prop:tilting-Bott-Samelson}
For any sequence $\us$ of simple reflections and $\omega \in \O$, the object $\Lder i^* \bigl( \MC_{\FM}(\omega,\us) \bigr)$ is in $\ES^{\GB \times \Gm}(\tNC)_\FM$, and is tilting. Moreover, any indecomposable object of $\Tilt(\ES^{\GB \times \Gm}(\tNC)_\FM)$ is a direct summand in an object of the form $\Lder i^* \bigl( \MC_{\FM}(\omega,\us) \bigr) \langle n \rangle$ with $(\omega, \us)$ as before and $n \in \ZM$. \qed
\end{prop}

The properties of $\Hom$-spaces between objects of the form $\MC_{\EM}(\underline{s},\omega)$ are summarized in the following proposition.

\begin{prop}
\label{prop:morphisms-BS-geom}
For any sequences $\underline{s}$ and $\underline{t}$ of simple reflections, for $\omega, \omega' \in \Omega$, and for $k \in \ZM$, we have
\begin{multline*}
\Hom_{D^{\GB}(\tgg)_\EM}(\MC_{\EM}(\omega,\us), \MC_{\EM}(\omega',\ut) [k]) = 0
\\
\text{and} \qquad \Hom_{D^{\GB}(\tNC)_\EM}( \Lder i^*(\MC_{\EM}(\omega,\us)), \Lder i^*(\MC_{\EM}(\omega',\ut)) [k]) = 0
\end{multline*}
unless $k=0$. Moreover, the graded $\OC(\tg_\EM^*)$-module
\[
\Hom_{D^{\GB}(\tgg)_\EM}(\MC_{\EM}(\omega,\us), \MC_{\EM}(\omega',\ut) )
\]
is graded free, and the functor $\Lder i^*$ induces an isomorphism
\begin{multline*}
\EM \otimes_{\OC(\tg_\EM^*)} \Hom_{D^{\GB}(\tgg)_\EM}(\MC_{\EM}(\omega,\us), \MC_{\EM}(\omega',\ut)) \\
\simto \Hom_{D^{\GB}(\tNC)_\EM}(\Lder i^*(\MC_{\EM}(\omega,\us)), \Lder i^*(\MC_{\EM}(\omega',\ut)) ).
\end{multline*}
Finally, the functor $\FM(-)$ induces isomorphisms of graded $\FM$-vector spaces
\[
\FM \otimes_{\RG} \Hom_{D^{\GB}(\tgg)_\RG}( \MC_{\RG}(\omega,\us), \MC_{\RG}(\omega',\ut) )
\simto \Hom_{D^{\GB}(\tgg)_\FM}(\MC_{\FM}(\omega,\us), \MC_{\FM}(\omega',\ut) )
\]
and
\begin{multline*}
\FM \otimes_{\RG} \Hom_{D^{\GB}(\tNC)_\RG}( \Lder i^*(\MC_{\RG}(\omega,\us)), \Lder i^*(\MC_{\RG}(\omega',\ut)) )  \\
\simto \Hom_{D^{\GB}(\tNC)_\FM}( \Lder i^*(\MC_{\FM}(\omega,\us)), \Lder i^* (\MC_{\FM}(\omega',\ut)) ).
\end{multline*}
\end{prop}

\begin{proof}
By~\cite[Lemma~4.1]{mr}, the objects of the form $\MC_{\EM}(\omega,\us)$ belong to the subcategory of $D^{\GB \times \Gm}(\tgg)_\EM$ generated under extensions by the objects $\nabla^\lambda_{\tgg,\EM}\langle m \rangle$ for $\lambda \in \XB$ and $m \in \ZM$, and also to the subcategory generated under extensions by the objects $\Delta^\lambda_{\tgg,\EM}\langle m \rangle$ for $\lambda \in \XB$ and $m \in \ZM$. (In~\cite{mr} only the case $\EM=\FM$ is considered, but the same proof applies in the case $\EM=\RG$; in fact this proof relies on the existence of the exact sequences~\eqref{eqn:ses-tgg}.) Then the claims follow from Propositions~\ref{prop:morphism-D-N-tNC} and~\ref{prop:morphism-D-N-tgg}.
\end{proof}

\subsection{Graded ranks of $\mathrm{Hom}$-spaces}
\label{ss:grk-Hom-tilting}

Recall the $\Haff$-module $\Msph$ of \S\ref{ss:Haff}. If $\FC$ is an object of $\ES^{\GB \times \Gm}(\tNC)_\FM$ which admits a standard filtration, we define 
\[
\ch_\Delta(\FC) := \sum_{\substack{\lambda \in \XB \\ n \in \ZM}} (\FC : \Delta^\lambda_{\tNC} \langle n \rangle) \cdot \vv^n \cdot \mm_\lambda \quad \in \Msph.
\]
Similarly, if $\GC$ is an object of $\ES^{\GB \times \Gm}(\tNC)_\FM$ which admits a costandard filtration, we define 
\[
\ch_\nabla(\GC) := \sum_{\substack{\lambda \in \XB \\ n \in \ZM}} (\GC : \nabla^\lambda_{\tNC} \langle n \rangle) \cdot \vv^{-n} \cdot \mm_\lambda \quad \in \Msph.
\]
One can easily check that for $\FC,\GC$ tilting objects in $\ES^{\GB \times \Gm}(\tNC)_\FM$ we have
\begin{equation}
\label{eqn:grk-tilting}
\grk_\FM \bigl( \Hom_{D^{\GB}(\tNC)_\FM}(\FC,\GC) \bigr) = \langle \ch_\Delta(\FC), \ch_\nabla(\GC) \rangle.
\end{equation}

\begin{prop}
\label{prop:ch-tilting}
For any sequence $\us$ of simple reflections and any $\omega \in \Omega$ we have
\[
\ch_\Delta(\Lder i^* \MC_\FM(\omega,\us)) = \ch_\nabla(\Lder i^* \MC_\FM(\omega,\us)) = \mm(\omega,\us).
\]
\end{prop}

\begin{proof}
We only consider the case of $\ch_\Delta$; the case of $\ch_\nabla$ is similar. For any $w \in \Waff$ we set
\[
\Delta^w_{\tNC,\FM} := \JS_{(T_{w^{-1}})^{-1}}(\OC_{\tNC_\FM}), \qquad \Delta^w_{\tgg,\FM} := \IS_{(T_{w^{-1}})^{-1}}(\OC_{\tgg_\FM})
\]
and
$\mm_w:=\mm_0 \cdot T_w \in \Msph$. Then $\Delta^w_{\tNC,\FM}$ is a standard object in $\ES^{\GB \times \Gm}(\tNC)_\FM$, and we have
\[
\ch_\Delta(\Delta^w_{\tNC,\FM}) = \mm_w.
\]
Indeed, the formula holds by definition if $w$ is minimal in $Ww$. For a general $w \in \Waff$, write $w=uv$ with $u \in W$ and $v$ minimal in $Ww$. Then $T_{w^{-1}}=T_{v^{-1}} T_{u^{-1}}$, so that we have
\begin{multline*}
\Delta^w_{\tNC,\FM} = \JS_{(T_{w^{-1}})^{-1}}(\OC_{\tNC_\FM}) = \JS_{(T_{u^{-1}})^{-1} \cdot (T_{v^{-1}})^{-1}}(\OC_{\tNC_\FM}) \\
= \JS_{(T_{v^{-1}})^{-1}} \circ \JS_{(T_{u^{-1}})^{-1}}(\OC_{\tNC_\FM}) = \JS_{(T_{v^{-1}})^{-1}}(\OC_{\tNC_\FM}) \langle \ell(u) \rangle
\end{multline*}
(see Remark~\ref{rk:nabla-delta-Baff}),
so that
\[
\ch_\Delta(\Delta^w_{\tNC,\FM}) = \vv^{\ell(u)} \ch_{\Delta}(\Delta^v_{\tNC,\FM}) = \vv^{\ell(u)} \mm_v=\mm_w.
\]

Using this fact and the proof of~\cite[Lemma~4.1]{mr}, one can check that for all $w \in \Waff$ and all simple reflections $s$, the object $\Lder i^*(\Xi_s^\FM (\Delta^w_{\tgg,\FM}))$ belongs to $\ES^{\GB \times \Gm}(\tNC)_\FM$ and admits a standard filtration, and that moreover we have
\[
\ch_\Delta \bigl( \Lder i^*(\Xi_s^\FM (\Delta^w_{\tgg,\FM})) \bigr) = \mm_w \cdot (T_s + \vv^{-1}) = \ch_\Delta \bigl( \Lder i^* (\Delta^w_{\tgg,\FM}) \bigr) \cdot (T_s + \vv^{-1}).
\]
Then one deduces that for any object $\FC$ in the subcategory of $D^{\GB \times \Gm}(\tgg)_\FM$ generated under extensions by the objects $\Delta^\lambda_{\tgg,\FM} \langle n \rangle$ ($\lambda \in \XB$, $n \in \ZM$), the object $\Lder i^*(\Xi_s^\FM ( \FC))$ belongs to $\ES^{\GB \times \Gm}(\tNC)_\FM$ and admits a standard filtration, and that moreover we have
\[
\ch_\Delta \bigl( \Lder i^*(\Xi_s^\FM ( \FC)) \bigr) = \ch_\Delta(\Lder i^*(\FC)) \cdot (T_s + \vv^{-1}).
\]
Since, for $\omega \in \Omega$, we have
\[
\ch_\Delta \bigl( \Lder i^*(\IS^\FM_{T_\omega}(\OC_{\tgg_\FM})) \bigr) = \ch_\Delta \bigl( \JS^\FM_{(T_{\omega^{-1}})^{-1}}(\OC_{\tNC_\FM}) \bigr) = \ch_\Delta(\Delta^\omega_{\tNC,\FM}) = \mm_\omega = \mm_0 \cdot T_\omega,
\]
the formula follows.
\end{proof}

\begin{prop}
\label{prop:grk-Hom-tilting}
Let $\us,\ut$ be sequences of simple reflections, and let $\omega,\omega' \in \Omega$. Then the graded $\OC(\tg^*_\RG)$-module
\[
\Hom_{D^{\GB}(\tgg)_\RG}( \MC_{\RG}(\omega,\us), \MC_{\RG}(\omega',\ut))
\]
is graded free, of graded rank $\langle \mm(\omega,\us),\mm(\omega',\ut) \rangle$.
\end{prop}

\begin{proof}
The first assertion follows from Proposition~\ref{prop:morphisms-BS-geom}. For the second assertion, we observe that, by Proposition~\ref{prop:morphisms-BS-geom} again, the graded rank under consideration is equal to the graded dimension of the graded $\FM$-vector space
\[
\Hom_{D^{\GB}(\tNC)_\FM} \bigl( \Lder i^*(\MC_{\FM}(\omega,\us)), \Lder i^*(\MC_{\RG}(\omega',\ut)) \bigr)
\]
(where $\FM$ is any geometric point of $\RG$).
Then the formula follows from~\eqref{eqn:grk-tilting} and Proposition~\ref{prop:ch-tilting}.
\end{proof}

\subsection{``Coherent'' Bott--Samelson category}
\label{ss:BS-category-tilting}

We define a ``coherent'' category of Bott--Samelson objects $\BSgeom$ as follows. The objects in this category are the triples $(\omega, \us, n)$ as in \S\ref{ss:topological-BS}. The morphisms are given by
\[
\Hom_{\BSgeom} \bigl( (\omega, \us, n), (\omega', \ut,  m) \bigr) =
\Hom_{D^{\GB \times \Gm}(\tgg)_\RG}(\MC_{\RG}(\omega,\us) \langle -n \rangle, \MC_{\RG}(\omega',\ut) \langle -m \rangle ).
\]

\begin{prop}
\label{prop:tilt-Karoubi}
The category $\Tilt(\ES^{\GB \times \GM}(\tNC)_\FM)$ can be recovered from the category $\BSgeom$, in the sense that it is equivalent to the Karoubian closure of the additive envelope of the category 
which has the same objects as $\BSgeom$, and morphisms from $(\omega, \us, n)$ to $(\omega', \ut, m)$ which are given by the $(m-n)$-th piece of the graded vector space
\[
\FM \otimes_{\OC(\tg^*_\RG)} \left( \bigoplus_{k \in \ZM} \Hom_{\BSgeom} \bigl( (\omega, \us, 0), (\omega', \ut, k) \bigr) \right).
\]
\end{prop}

\begin{proof}
Consider the category which has the same objects as $\BSgeom$, and whose morphisms are defined as in the statement of the proposition. By Proposition~\ref{prop:morphisms-BS-geom}, this category is equivalent to the full subcategory $\mathsf{A}$ of $D^{\GB \times \Gm}(\tNC)_\FM$ whose objects are of the form $\Lder i^*(\MC_\FM(\omega,\us))$.
It follows from Proposition~\ref{prop:tilting-Bott-Samelson} that the category $\Tilt(\ES^{\GB \times \Gm}(\tNC)_\FM)$ is equivalent to the Karoubian closure of the additive envelope of $\mathsf{A}$; the claim follows.
\end{proof}

\subsection{Kostant--Whittaker reduction and Bott--Samelson objects}

Let us consider the constructions of~\S\S\ref{ss:deformations}--\ref{ss:algebraic-BS}, with $\XB$ now defined as $X^*(\TB_\RG)$. Then the space $\tg^*$ of Section~\ref{sec:constructible-side} coincides with the space denoted $\tg^*_\RG$ in the present section, and we can define the algebra $C$ and the category $\BSalg$ of ``Bott--Samelson'' $C$-modules with the present data. We use the same notation as in~\S\ref{ss:algebraic-BS} for the modules $E_w$ and $D_s$.

\begin{lem}
\label{lem:isom-C-tangent}
There exists a natural $W$-equivariant isomorphism of graded algebras
\[
C \simto \OC \bigl( \tg^* \times_{\tg^* / W} \TF(\tg^* / W) \bigr)_\RG.
\]
\end{lem}

\begin{proof}
By definition, the algebra $C$ is generated by $\OC(\tg^*_\RG)$ and the images of the elements of the form $\hbar^{-1}(f \otimes 1 - 1 \otimes f) \in C_\hbar$ for $f \in \OC(\tg^*_\RG)^W$. On the other hand, the algebra  $\OC \bigl( \tg^* \times_{\tg^* / W} \TF(\tg^* / W) \bigr)_\RG$ is the symmetric algebra (over $\OC(\tg^*_\RG)$) of the module
\[
\bigl( \OC(\tg^*) \otimes_{\OC(\tg^*)^W} \Omega(\tg^*/W) \bigr)_\RG.
\]
(Note that this module is free of finite rank since $\tg^*_\RG/W$ is an affine space, see Lemma~\ref{lem:t/W-affine-space}.) One can easily check that the assignment
\[
\hbar^{-1}(f \otimes 1 - 1 \otimes f) \mapsto d(f)
\]
induces a morphism of graded $\OC(\tg^*_\RG)$-algebras $C \to \OC \bigl( \tg^* \times_{\tg^* / W} \TF(\tg^* / W) \bigr)_\RG$, and that this morphism is a $W$-equivariant isomorphism.
\end{proof}

From now on we identify the two algebras in Lemma~\ref{lem:isom-C-tangent} via the isomorphism constructed in its proof. Then one can consider the functor $\kappa_\RG$ of~\S\ref{ss:KW-functors} as a functor taking values in $\Db \Modgr(C)$.

\begin{prop}
\label{prop:kappa-BS}
For any sequence $\underline{s}$ of simple reflections, any $\omega \in \Omega$, and any $n \in \ZM$, there exists an isomorphism
\[
\kappa_\RG \bigl( \MC_\RG(\omega,\us) \langle n \rangle \bigr) \cong D(\omega,\us) \langle n \rangle \qquad \text{in $\Db  \Modgr(C)$.}
\]
\end{prop}

Before proving the proposition we begin with a lemma.

\begin{lem}
\label{lem:kappa-Xi}
Let $s$ be a finite simple reflection. Let $\FC$ be an object of $D^{\GB \times \Gm}(\tgg)_\RG$ such that $\kappa_\RG(\FC)$ is concentrated in degree $0$, and $\RG$-free. Then $\kappa_\RG \circ \Xi^\RG_s(\FC)$ is concentrated in degree $0$ and $\RG$-free, and moreover there exists an isomorphism of graded $C$-modules
\[
\kappa_\RG \circ \Xi^\RG_s(\FC) \cong \kappa_\RG(\FC) \otimes_{\OC(\tg^*_\RG)} D_s.
\]
\end{lem}

\begin{proof}
The proof is very similar to the proof of the case $b=T_s$ of Proposition~\ref{prop:KW-actions}\eqref{it:KW-actions-R}. In fact, using the same notation as in this proof, by Lemma~\ref{lem:fiber-product-xis} and the base change theorem (and using also~\eqref{eqn:Ds-tensor-product}) we have a canonical isomorphism of functors
\begin{equation}
\label{eqn:isom-kappa-Xi}
\kappa'_\EM \circ \Xi^\EM_s(-) \cong \kappa'_\EM(-) \otimes_{\OC(\tg^*_\EM)} (\EM \otimes_\RG D_s)
\end{equation}
for $\EM=\RG$ or $\FM$. (Note that $D_s$ is free over $\OC(\tg^*_\RG)$ -- see the proof of Lemma~\ref{lem:fiber-product-xis}\eqref{it:xis-graph} -- so that the tensor product on the right-hand side makes sense in the derived category.)

When $\EM=\FM$, using the fact that $\Xi_s^\FM \cong \Lder (\widetilde{\pi}_s)^* \circ \Rder (\widetilde{\pi}_s)_*$ 
(see~\cite[Proposition~5.2.2]{riche}) and the remarks at the end of~\S\ref{ss:reminder}, one can check that this isomorphism is induced by an isomorphism of functors
\[
\kappa_\FM \circ \Xi^\FM_s(-) \cong \kappa_\FM(-) \otimes_{\OC(\tg^*_\FM)} (\FM \otimes_\RG D_s)
\]
(where the tensor product on the right-hand side is defined a similar way as in the case of $\RG$).
Then we can come back to the case $\EM=\RG$: if $\kappa_\RG(\FC)$ is concentrated in degree $0$ and $\RG$-free, \eqref{eqn:isom-kappa-Xi} shows that the same holds for $\kappa_\RG \circ \Xi^\RG_s(\FC)$. And the final claim follows from the case $\EM=\CM$ treated above.
\end{proof}

\begin{proof}[Proof of Proposition~{\rm \ref{prop:kappa-BS}}]
We remark that the statement of Lemma~\ref{lem:kappa-Xi} also holds when $s$ is an \emph{affine} simple reflection, by definition of $\Xi_s$ in this case and Proposition~\ref{prop:KW-actions}\eqref{it:KW-actions-R} (see also~\eqref{eqn:K-'K}). This observation reduces the proof of Proposition~\ref{prop:kappa-BS} to the proof that
\[
\kappa_\RG(\IS^\RG_{T_\omega}(\OC_{\tgg_\RG})) \cong E_\omega
\]
when $\omega \in \Omega$. However by Proposition~\ref{prop:KW-actions} (see also~\eqref{eqn:K-'K} and~\eqref{eqn:n-Tomega}) we have
\[
\kappa_\RG(\IS^\RG_{T_\omega}(\OC_{\tgg_\RG})) \cong {}' \KS^\RG_\omega(\OC_{\tg_\RG^*}),
\]
where $\tg_\RG^*$ is identified with the zero section in $\tg^* \times_{\tg^*/W} \TF(\tg^*/W)_\RG$. Now writing $\omega=v t_\lambda$ (with $v \in W$ and $\lambda \in \XB$) we have
\[
{}' \KS^\RG_\omega(\OC_{\tg_\RG^*}) \cong {}' \KS^\RG_{t_\lambda} \circ {}' \KS^\RG_v (\OC_{\tg_\RG^*}) \cong {}' \KS^\RG_{t_\lambda} (\OC_{\tg_\RG^*}) \cong F_\lambda^\RG.
\]
Similarly we have $E_\omega \cong E_{t_\lambda}$, and one can check that $F_\lambda^\RG$ and $E_{t_\lambda}$ are isomorphic under our identification $C \cong \OC \bigl( \tg^* \times_{\tg^* / W} \TF(\tg^* / W) \bigr)_\RG$.
\end{proof}

\subsection{Equivalence}
\label{ss:equivalence-tilting}

Let us fix isomorphisms as in Proposition~\ref{prop:kappa-BS}, for all $\us$ and $\omega$. Then the functor $\kappa_\RG$ induces a functor
\[
\kappa_{\mathsf{BS}} \colon \BSgeom \to \BSalg.
\]

The main result of this section is the following.

\begin{thm}
\label{thm:main-geom}
The functor $\kappa_{\mathsf{BS}}$ is an equivalence of categories.
\end{thm}

Before proving the theorem we first establish a lemma. In this statement, we identify quasi-coherent sheaves on $\tSC_\RG$ with $\OC(\tSC_\RG)$-modules, and denote by $\Rep(\tIG_\SC^\RG)$ the abelian category of representations of the commutative $\OC(\tSC_\RG)$-Lie algebra $\tIG_\SC^\RG$.

\begin{lem}
\label{lem:differentiation-fully-faithful}
The natural functor
\[
\Rep(\tIB_{\SC}^\RG) \to \Rep(\tIG_\SC^\RG)
\]
defined by differentiating the action is fully faithful on representations which are free over $\OC(\tSC_\RG)$.
\end{lem}

\begin{proof}
The group scheme $\tIB_\SC^\RG$ is smooth over $\tSC_\RG$ (see~\S\ref{ss:reminder}), hence it is infinitesimally flat as an $\OC(\tSC_\RG)$-group scheme (in the sense of~\cite[\S I.7.4]{jantzen}) by~\cite[Expos{\'e}~VII, Proposition~1.10]{sga6}. This group scheme is also integral, since it is flat over $\tSC_\RG$ and its pullback to $\tSC_\CM \cap \tgg_\CM^\rs$ is irreducible. (In fact, it can be deduced from~\cite[Lemma~2.3.3]{riche3} that the restriction of $\tIB^\CM$ to $\tgg_\CM^\rs$ identifies with $\TB \times \tgg_\CM^\rs$.) Using~\cite[Lemma~I.7.16]{jantzen} we deduce that, if we denote by $\mathrm{Dist}(\tIB_\SC^\RG)$ the distribution algebra of this group scheme, the natural functor
\[
\Rep(\tIB_\SC^\RG) \to \Mod \bigl( \mathrm{Dist}(\tIB_\SC^\RG) \bigr)
\]
is fully faithful on representations which are projective over $\OC(\tSC_\RG)$.

Now, consider the natural algebra morphism
\[
\UC_{\OC(\tSC_\RG)}(\tIG_\SC^\RG) \to \mathrm{Dist}(\tIB_\SC^\RG),
\]
where the left-hand side is the universal enveloping algebra of $\tIG_\SC^\RG$, see~\cite[\S I.7.10]{jantzen}. If $\KM$ denotes the fraction field of $\OC(\tSC_\RG)$, we claim that this morphism induces an isomorphism
\begin{equation}
\label{eqn:morph-Lie-Dist}
\KM \otimes_{\OC(\tSC_\RG)} \UC_{\OC(\tSC_\RG)}(\tIG_\SC^\RG) \simto \KM \otimes_{\OC(\tSC_\RG)} \Dist(\tIB_\SC^\RG).
\end{equation}
Indeed, the left-hand side is isomorphic to $\UC_{\KM}(\KM \otimes_{\OC(\tSC_\RG)} \tIG_\SC^\RG)$ and, if we set $\tIB_\KM:=\Spec(\KM) \times_{\tSC_\RG} \tIB_\SC^\RG$, by~\cite[Lemma~2.1.1]{riche3} and~\cite[\S I.7.4, Equation (1)]{jantzen} respectively, there are natural isomorphisms
\[
\KM \otimes_{\OC(\tSC_\RG)} \tIG_\SC^\RG \cong \Lie(\tIB_\KM), \qquad \KM \otimes_{\OC(\tSC_\RG)} \Dist(\tIB_\SC^\RG) \cong \Dist(\tIB_\KM),
\]
so that~\eqref{eqn:morph-Lie-Dist} gets identified with the natural morphism
\[
\UC_\KM(\Lie(\tIB_\KM)) \to \Dist(\tIB_\KM).
\]
Since $\KM$ is a field of characteristic zero, the latter morphism is an isomorphism by~\cite[\S I.7.10, Equation (1)]{jantzen}, which finishes the proof of our claim.

Now that the claim is established, the lemma follows from Lemma~\ref{lem:key}.
\end{proof}

\begin{proof}[Proof of Theorem~{\rm \ref{thm:main-geom}}]
By definition, $\kappa_{\mathsf{BS}}$ is essentially surjective. Hence what we have to prove is that for any sequences $\us, \ut$ of simple reflections and any $\omega, \omega' \in \Omega$, the morphism
\[
\Hom_{D^{\GB}(\tgg)_\RG}(\MC_\RG(\omega,\us), \MC_\RG(\omega',\ut)) \to \Hom_{C}(D(\omega,\us), D(\omega',\ut))
\]
induced by $\overline{\kappa}_\RG$ (via the isomorphisms of Proposition~\ref{prop:kappa-BS}) is an isomorphism. By Corollary~\ref{cor:grk-Hom-parity} and Proposition~\ref{prop:grk-Hom-tilting}, both sides are graded free over $\OC(\tg^*_\RG)$, of the same graded rank. Hence, using Lemma~\ref{lem:morphism-grk}, to finish the proof it suffices to prove that the image of our morphism under the functor $\OC(\tg^{*,\rs}_{\FM}) \otimes_{\OC(\tg^*_\RG)} (-)$ is injective for any geometric point $\FM$ of $\RG$. Then, using the definition of $\overline{\kappa}_\RG$ (see~\S\ref{ss:KW-functors}) and Lemma~\ref{lem:differentiation-fully-faithful}, it suffices to prove that the morphism
\begin{multline*}
\OC(\tg^*_{\rs,\FM}) \otimes_{\OC(\tg^*_\RG)} \Hom_{D^{\GB}(\tgg)_\RG}(\MC_\RG(\omega,\us), \MC_\RG(\omega',\ut)) \\
\to \OC(\tg^{*,\rs}_{\FM}) \otimes_{\OC(\tg^*_\RG)} \Hom_{\Rep(\tIB_\SC^\RG)}(\MC_\RG(\omega,\us){}_{| \tSC_\RG}, \MC_\RG(\omega',\ut){}_{| \tSC_\RG})
\end{multline*}
induced by restriction to $\tSC_\RG$ is injective. And finally, to prove this injectivity it suffices to prove that the natural morphism
\begin{multline*}
\OC(\tg^{*,\rs}_{\FM}) \otimes_{\OC(\tg^*_\RG)} \Hom_{D^{\GB}(\tgg)_\RG}(\MC_\RG(\omega,\us), \MC_\RG(\omega',\ut)) \\
\to \Hom_{\Coh(\tSC_\FM^\rs)}(\MC_\FM(\omega,\us){}_{| \tSC_\FM^\rs}, \MC_\FM(\omega',\ut){}_{| \tSC_\FM^\rs})
\end{multline*}
is injective.

Consider the following chain of isomorphisms:
\begin{multline*}
\OC(\tg^{*,\rs}_{\FM}) \otimes_{\OC(\tg^*_\RG)} \Hom_{D^{\GB}(\tgg)_\RG}(\MC_\RG(\omega,\us), \MC_\RG(\omega', \ut)) \\
\cong \OC(\tg^{*,\rs}_{\FM}) \otimes_{\OC(\tg^*_\FM)} \bigl( \FM \otimes_\RG \Hom_{D^{\GB}(\tgg)_\RG}(\MC_\RG(\omega,\us), \MC_\RG(\omega',\ut)) \bigr) \\
\cong \OC(\tg^{*,\rs}_{\FM}) \otimes_{\OC(\tg^*_\FM)} \bigl( \Hom_{D^{\GB}(\tgg)_\FM}(\MC_\FM(\omega,\us), \MC_\FM(\omega',\ut)) \bigr) \\
\cong \Hom_{\Coh^{\GB}(\tgg^\rs)_\FM}(\MC_\FM(\omega,\us){}_{|_{\tgg^\rs_\FM}}, \MC_\FM(\omega',\ut){}_{|_{\tgg^\rs_\FM}}).
\end{multline*}
Here the second isomorphism follows from Proposition~\ref{prop:morphisms-BS-geom}, and the other ones are easy. (Observe that, on the last line, the objects $\MC_\FM(\omega,\us){}_{|_{\tgg^\rs_\FM}}$ and $\MC_\FM(\omega',\ut){}_{|_{\tgg^\rs_\FM}}$ are concentrated in degree $0$, i.e.~equivariant coherent sheaves.)
Using these identifications, the morphism under consideration is the morphism
\begin{multline*}
\Hom_{\Coh^{\GB_\FM}(\tgg^\rs_\FM)} \bigl( \MC_\FM(\omega,\us){}_{|_{\tgg^\rs_\FM}}, \MC_\FM(\omega',\ut){}_{|_{\tgg^\rs_\FM}} \bigr) \\
\to \Hom_{\Coh(\tSC^\rs)_\FM} \bigl( \MC_\FM(\omega,\us){}_{| \tSC_\FM^\rs}, \MC_\FM(\omega',\ut){}_{| \tSC_\FM^\rs} \bigr)
\end{multline*}
induced by restriction to $\tSC_\FM^\rs$. Then injectivity follows from the observation that the functor $\Coh^{\GB}(\tgg^\rs)_\FM \to \Coh(\tSC^\rs)_\FM$ is faithful, since it can be written as the composition
\[
\Coh^{\GB}(\tgg^\rs)_\FM \to \Coh^{\GB}(\GB \times \tSC^\rs)_\FM \simto \Coh(\tSC^\rs)_\FM
\]
where the first functor is the inverse image under the natural morphism $a^{\rs} \colon \GB_\FM \times \tSC_\FM^\rs \to \tgg_\FM^\rs$ (which is faithful since $a^\rs$ is smooth and surjective, see Lemma~\ref{lem:a-smooth}), and the second functor is the natural equivalence.
\end{proof}

\section{Proofs of the main results}
\label{sec:proofs}

In this section (except in~\S\ref{ss:proof-parity-perverse}) we come back to the assumptions of~\S\ref{ss:intro}: $\GB$ is a product of simply-connected quasi-simple groups and general linear groups over $\FM$, and the characteristic $p$ of $\FM$ is very good for each quasi-simple factor of $\GB$. Such a group can be obtained by base change to $\FM$ from a split connected reductive group scheme $\GB_\ZM$ over $\ZM$, and we set $\GB_\RG=\Spec(\RG) \times_{\Spec(\ZM)} \GB_\ZM$, where $\RG$ is the localization of $\ZM$ at all the prime numbers which are not very good for a quasi-simple factor of $\GB$. These data satisfy the assumptions of Sections~\ref{sec:KW-reduction}--\ref{sec:tilting-KW}, and we use the same notation as in these sections. (The only exception is that we drop the subscripts ``$\FM$,'' since the geometric point will not vary anymore.)

We let also $\GD$ be the complex Langlands dual group. This group and $\RG$ satisfy the assumptions of Section~\ref{sec:constructible-side}, and we also use the notation of this section. (Note a slight conflict of notation:  $\tg$ denotes ${\check \XB} \otimes_\ZM \RG$ in Section~\ref{sec:constructible-side}, while in the present section it will denote ${\check \XB} \otimes_\ZM \FM$.)
We assume that the roots of $\BD$ with respect to $\TD$ are the coroots of $\BB$ with respect to $\TB$.

\subsection{Proof of Theorem~\ref{thm:main}}
\label{ss:proof-main}

Combining Theorem~\ref{thm:main-top} and Theorem~\ref{thm:main-geom} we obtain an equivalence of categories
\begin{equation}
\label{eqn:equiv-BS}
\BStop \simto \BSgeom
\end{equation}
which is the identity morphism on objects. Then using Proposition~\ref{prop:parity-Karoubi} and Proposition~\ref{prop:tilt-Karoubi} we deduce the desired equivalence of additive categories
\[
\Theta \colon
\Parity_{(\ID)}(\Gr, \FM) \simto \Tilt(\ES^{\GB \times \Gm}(\tNC)).
\]
By construction, this equivalence
satisfies 
\begin{equation}
\label{eqn:Theta'-BS}
\Theta \bigl( \EC_\FM(\omega,\us) \bigr) \cong \Lder i^* \MC_\FM(\omega,\us)
\end{equation}
for any sequence $\us$ of simple reflections and any $\omega \in \Omega$, and $\Theta \circ [1] \cong \langle -1 \rangle \circ \Theta$. 
It also satisfies property~\eqref{it:main-thm-indec},
by~\eqref{eqn:Theta'-BS} and a standard argument based on induction on $\ell(w_\lambda)$ (see in particular Remark~\ref{rk:indec-parity} and \cite[Remark~4.3]{mr}).

Finally we prove property~\eqref{it:main-thm-ch}. In fact, this property can be expressed as the following equalities for any $\EC$ in $\Parity_{(\ID)}(\Gr,\FM)$:
\begin{equation}
\label{eqn:equality-characters}
\ch_{\Gr}^*(\EC) = \ch_\Delta(\Theta(\EC)); \qquad \ch_{\Gr}^!(\EC) = \ch_{\nabla}(\Theta(\EC)).
\end{equation}
However these formulas hold when $\EC=\EC_\FM(\omega,\us)$ by Proposition~\ref{prop:ch-BS-parity}, Proposition~\ref{prop:ch-tilting}, and~\eqref{eqn:Theta'-BS}. And one can easily check that if they hold for $\EC$ then they hold for $\EC[1]$. Using these observations one can prove, by standard arguments, that the formulas~\eqref{eqn:equality-characters} hold when $\EC=\EC_\lambda[i]$ for some $\lambda \in \XB$ and $i \in \ZM$, by induction on $\ell(w_\lambda)$. The general case follows.


\subsection{Proof of Theorem~\ref{thm:equivalence-tNC}}
\label{ss:proof-equivalence-tNC}

We construct the equivalence $\Phi$ as the composition
\[
\Dmix_{(\ID)}(\Gr, \FM) := \Kb \Parity_{(\ID)}(\Gr, \FM) \xrightarrow[\sim]{\Kb(\Theta)} \Kb \Tilt(\ES^{\GB \times \Gm}(\tNC))
\xrightarrow{\sim} D^{\GB \times \Gm}(\tNC),
\]
where the last equivalence is provided by~\cite[Proposition~3.11]{mr}.
Then it follows from this construction and the properties of $\Theta$ proved in Theorem~\ref{thm:main} that we have
\[
\Phi \circ \langle 1 \rangle \cong \langle 1 \rangle [1] \circ \Phi \qquad \text{and} \qquad \Phi(\EC^\mix_\lambda) \cong \TC^{-\lambda}.
\]
It remains to prove that
\[
\Phi(\dmix_{-\lambda}) \cong \Delta_{\tNC}^{\lambda} \qquad \text{and} \qquad \Phi(\nmix_{-\lambda}) \cong \nabla_{\tNC}^{\lambda}
\]
for any $\lambda \in \XB$.
We explain the proof of the first isomorphism; the proof of the second one is similar.

Our proof is similar to the proof of a similar claim in~\cite[Lemma~5.2]{modrap2}. Namely, we prove the isomorphism by induction on $\ell(w_\lambda)$. If $\ell(w_\lambda)=0$, then we have $\Delta^{\mix}_{-\lambda} \cong \EC^{\mix}_{-\lambda}$ and $\Delta^\lambda_{\tNC} \cong \TC^\lambda$, hence our claim is clear. For a general $\lambda$, we consider a non-zero morphism $\Delta^\lambda_{\tNC} \to \TC^\lambda$ (which is unique up to an invertible scalar), and the associated triangle
\begin{equation}
\label{eqn:triangle-Delta-Phi}
\Delta^\lambda_{\tNC} \to \TC^\lambda \to N \xrightarrow{[1]}.
\end{equation}
Then $N$ belongs to the triangulated subcategory of $D^{\GB \times \Gm}(\tNC)$ generated by the objects $\Delta^\mu_{\tNC} \langle m \rangle$ with $m \in \ZM$ and $\mu \in \XB$ which satisfies $\Gr_{-\mu} \subset \overline{\Gr_{-\lambda}}$ and $\mu \neq \lambda$ (see Theorem~\ref{thm:main}\eqref{it:main-thm-ch}--\eqref{it:main-thm-indec}). In particular, using induction we deduce that $\Phi^{-1}(N)$ is supported on $\overline{\Gr_{-\lambda}} \smallsetminus \Gr_{-\lambda}$.

Using again induction we also observe that
\[
\Hom_{\Dmix_{(\ID)}(\Gr, \FM)}(\Phi^{-1}(\Delta^\lambda_{\tNC}), \Delta^{\mix}_{-\mu} \langle m \rangle [n]) = 0
\]
for all $n,m \in \ZM$ and $\mu \in \XB$ such that $\Gr_{-\mu} \subset \overline{\Gr_{-\lambda}}$ and $\mu \neq \lambda$. Hence the triangle
\[
\Phi^{-1}(\Delta^\lambda_{\tNC}) \to \EC^\mix_{-\lambda} \to \Phi^{-1}(N) \xrightarrow{[1]}
\]
obtained from~\eqref{eqn:triangle-Delta-Phi}
satisfies the two properties which characterize uniquely the triangle whose first arrow is the unique (up to scalar) non-zero morphism $\Delta^{\mix}_{-\lambda} \to \EC^{\mix}_{-\lambda}$. In particular, we deduce the wished-for isomorphism $\Phi^{-1}(\Delta^\lambda_{\tNC}) \cong \Delta^{\mix}_{-\lambda}$.

\begin{remark}
\label{rk:real-tNC}
The equivalence $\Kb \Tilt(\ES^{\GB \times \Gm}(\tNC))
\xrightarrow{\sim} D^{\GB \times \Gm}(\tNC)$ used in the proof of Theorem~\ref{thm:equivalence-tNC} is the composition
\begin{equation}
\label{eqn:real-1}
\Kb \Tilt(\ES^{\GB \times \Gm}(\tNC)) \hookrightarrow \Kb \ES^{\GB \times \Gm}(\tNC) \xrightarrow{\mathsf{real}} D^{\GB \times \Gm}(\tNC)
\end{equation}
where $\mathsf{real}$ is the functor constructed in~\cite[Lemme~3.1.11]{bbd}. On the other hand it follows from~\cite[Corollary~4.16]{mr} that the objects in $\Tilt(\ES^{\GB \times \Gm}(\tNC))$ are concentrated in degree $0$, so that one can also construct such an equivalence as the composition
\begin{equation}
\label{eqn:real-2}
\Kb \Tilt(\ES^{\GB \times \Gm}(\tNC)) \hookrightarrow \Kb \Coh^{\GB \times \Gm}(\tNC) \xrightarrow{\mathsf{can}} D^{\GB \times \Gm}(\tNC),
\end{equation}
where $\mathsf{can}$ is the canonical functor. We claim that~\eqref{eqn:real-1} and~\eqref{eqn:real-2} are isomorphic. Indeed, it is clear that $\mathsf{can}$ coincides with the realization functor from~\cite[Lemme~3.3.11]{bbd} associated with the tautological t-structure on $D^{\GB \times \Gm}(\tNC)$. Then the claim follows from the observation that if $\mathscr{C}$ and $\mathscr{C}'$ are the hearts of two bounded t-structures on $\Db \mathscr{A}$ for some abelian category $\mathscr{A}$, then the following diagram commutes up to isomorphism:
\[
\xymatrix@R=0.6cm{
\Kb (\mathscr{C} \cap \mathscr{C}') \ar@{^{(}->}[r] \ar@{^{(}->}[d] & \Kb \mathscr{C} \ar[d]^-{\mathsf{real}} \\
\Kb\mathscr{C}' \ar[r]^{\mathsf{real'}} & \Db \mathscr{A}.
}
\]
\end{remark}

\subsection{Proof of Theorem~\ref{thm:equivalence-tgg}}
\label{ss:proof-equivalence-tgg}

Before considering the proof of Theorem~\ref{thm:equivalence-tgg}, we need some preliminary results.

First, recall that 
we use the same notation for the indecomposable objects in $\Parity_{\ID}(\Gr, \FM)$ 
and in $\Parity_{(\ID)}(\Gr, \FM)$, see~\S\ref{ss:notation-constructible}.
We define the additive subcategory ${\EuScript Tilt}$ of $D^{\GB \times \Gm}(\tgg)$ as the Karoubian closure of the additive envelope of the subcategory generated by the objects $\MC_\FM(\omega,\us)$ for $\us$ a sequence of simple reflections and $\omega \in \Omega$. Using~\eqref{eqn:equiv-BS} it is not difficult to construct an equivalence of categories
\begin{equation}
\label{eqn:equiv-Parity-Tilt-equiv}
\Parity_{\ID}(\Gr, \FM) \simto {\EuScript Tilt}
\end{equation}
sending $\EC_\FM(\omega,\us) [n]$ to $\MC_\FM(\omega,\us) \langle -n \rangle$.

This equivalence shows that the category ${\EuScript Tilt}$ is Krull--Schmidt, and that one can transfer the known classification of indecomposable objects in $\Parity_{\ID}(\Gr, \FM)$ proved in~\cite{jmw} to deduce a classification of the indecomposable objects in ${\EuScript Tilt}$. Namely, for any $\lambda \in \XB$, if $w_{\lambda}=\omega s_1 \cdots s_r$ is a reduced decomposition, then there exists a unique indecomposable factor $\tTC^\lambda$ of $\MC_\FM(\omega,(s_1, \cdots, s_r))$ which is not isomorphic to any object of the form $\MC_\FM(\omega',\ut) \langle m \rangle$ where $m \in \ZM$, $\omega' \in \Omega$, and $\ut$ is a sequence of simple reflections of length at most $r-1$. This object does not depend on the choice of the reduced decomposition up to isomorphism. Moreover, any indecomposable object in ${\EuScript Tilt}$ is isomorphic to some $\tTC^\lambda \langle m \rangle$ for some $\lambda \in \XB$ and $m \in \ZM$. Finally, the image of $\EC_{-\lambda}$ under~\eqref{eqn:equiv-Parity-Tilt-equiv} is $\tTC^\lambda$.

The following lemma implies that the objects $\tTC^\lambda$ are ``deformations'' of the tilting exotic sheaves $\TC^\lambda$.

\begin{lem}
\label{lem:tTC-i*}
For any $\lambda \in \XB$ we have $\Lder i^*(\tTC^\lambda) \cong \TC^\lambda$.
\end{lem}

\begin{proof}
The result easily follows if we can prove that $\Lder i^*(\tTC^\lambda)$ is indecomposable. However, since $\tTC^\lambda$ is a direct summand in an object of the form $\MC_\FM(\omega,\us)$, it follows from Proposition~\ref{prop:morphisms-BS-geom} that the functor $\Lder i^*$ induces an isomorphism
\[
\FM \otimes_{\OC(\tg^*)} \End_{D^{\GB}(\tgg)}(\tTC^\lambda) \simto \End_{D^{\GB}(\tNC)} \bigl( \Lder i^* (\tTC^\lambda) \bigr).
\]
Then the proof proceeds as for Lemma~\ref{lem:parity-indec-For}.
\end{proof}

\begin{lem}
\label{lem:Tilt-generates}
The subcategory ${\EuScript Tilt}$ generates $D^{\GB \times \Gm}(\tgg)$ as a triangulated category.
\end{lem}

\begin{proof}
Looking at the proof of~\cite[Corollary~4.2]{mr}, one can check that the triangulated subcategory generated by ${\EuScript Tilt}$ coincides with the triangulated subcategory generated by the objects $\Delta^\lambda_{\tgg} \langle m \rangle$ for $\lambda \in \XB$ and $m \in \ZM$. By~\cite[Lemma~1.11.3(2)]{br}, for $w \in W$ we have $\IS_{T_w}(\OC_{\tgg}) \cong \OC_{\tgg} \langle - \ell(w) \rangle$. We deduce (as in Remark~\ref{rk:nabla-delta-Baff}) that we have
\[
\Delta^\lambda_{\tgg} \cong \IS_{(T_{t_{-\lambda}})^{-1}}(\OC_{\tgg}) \langle - \ell(t_\lambda) + \ell(w_\lambda) \rangle.
\]
Then, using~\cite[Lemma~2.3]{mr} and~\cite[Lemma~1.11.3]{br}, one can check that the triangulated subcategory generated by the object $\Delta^\lambda_{\tgg}\langle m \rangle$ coincides with the triangulated subcategory generated by the objects $\OC_{\tgg}(\lambda) \langle m\rangle$ for $\lambda \in \XB$ and $m \in \ZM$. Hence the claim of the lemma is equivalent to the claim that $D^{\GB \times \Gm}(\tgg)$ is generated by this collection of objects, which can be proved by the same arguments as the corresponding claim for $\tNC$ in~\cite[Corollary~5.8]{achar}.
\end{proof}

\begin{lem}
\label{lem:KbTilt-tgg}
There exists an equivalence of triangulated categories
\[
\Kb {\EuScript Tilt} \simto D^{\GB \times \Gm}(\tgg)
\]
commuting with $\langle 1 \rangle$ and
sending $\tTC^\lambda$ to $\tTC^\lambda$ for any $\lambda \in \XB$.
\end{lem}

\begin{proof}
It follows from~\cite[Corollary~4.16]{mr} that the complexes $\TC^\lambda$ are concentrated in degree $0$. Using Lemma~\ref{lem:tTC-i*} and Lemma~\ref{lem:nakayama}\eqref{it:nakayama-complex} (over an affine open covering of $\tgg$) we deduce that the complexes $\tTC^\lambda$ are also concentrated in degree $0$. We construct the functor of the lemma as the composition
\[
\Kb {\EuScript Tilt} \to \Kb \Coh^{\GB \times \Gm}(\tgg) \to D^{\GB \times \Gm}(\tgg).
\]
It follows from standard arguments, using Proposition~\ref{prop:morphisms-BS-geom} and Lemma~\ref{lem:Tilt-generates}, that this functor is an equivalence of categories.
\end{proof}

\begin{proof}[Proof of Theorem~{\rm \ref{thm:equivalence-tgg}}]
We define $\Psi$ as the composition
\[
\Dmix_{\ID}(\Gr, \FM) := \Kb \Parity_{\ID}(\Gr, \FM) \xrightarrow[\sim]{\eqref{eqn:equiv-Parity-Tilt-equiv}} \Kb {\EuScript Tilt}
\xrightarrow[\sim]{{\rm Lemma~\ref{lem:KbTilt-tgg}}} D^{\GB \times \Gm}(\tgg).
\]
By construction, this equivalence satisfies $\Psi \circ \langle 1 \rangle \cong \langle 1 \rangle [1] \circ \Psi$ and $\Psi(\EC^{\mix}_{-\lambda}) \cong \tTC^{\lambda}$. The isomorphisms involving standard and costandard objects can be proved using the same arguments as for their counterparts in Theorem~\ref{thm:equivalence-tNC}.
\end{proof}

\subsection{Compatibility}
\label{ss:compatibilities}

In this subsection we prove
that the functors $\Phi$ and $\Psi$ as constructed in~\S\S\ref{ss:proof-equivalence-tNC}--\ref{ss:proof-equivalence-tgg}
are compatible in the natural way.
We will denote by
$\For \colon \Dmix_{\ID}(\Gr, \FM) \to \Dmix_{(\ID)}(\Gr, \FM)$
the ``forgetful functor'' induced by the forgetful functor $\Parity_{\ID}(\Gr, \FM) \to \Parity_{(\ID)}(\Gr, \FM)$.

\begin{prop}
\label{prop:Phi-Psi-For}
The following diagram commutes up to isomorphisms of functors:
\[
\xymatrix@R=0.6cm@C=1.5cm{
\Dmix_{\ID}(\Gr, \FM) \ar[r]^-{\Psi}_-{\sim} \ar[d]_-{\For} & D^{\GB \times \Gm}(\tgg) \ar[d]^-{\Lder i^*} \\
\Dmix_{(\ID)}(\Gr, \FM) \ar[r]^-{\Phi}_-{\sim} & D^{\GB \times \Gm}(\tNC).
}
\]
\end{prop}

\begin{proof}
It is clear from construction that the diagram
\[
\xymatrix@R=0.6cm@C=1.5cm{
\Parity_{\ID}(\Gr, \FM) \ar[r]_-{\sim}^-{\eqref{eqn:equiv-Parity-Tilt-equiv}} \ar[d]_-{\For} & {\EuScript Tilt} \ar[d]^-{\Lder i^*} \\
\Parity_{(\ID)}(\Gr, \FM) \ar[r]_-{\sim}^-{\eqref{eqn:equiv-main}} & \Tilt(\ES^{\GB \times \Gm}(\tNC))
}
\]
commutes. Hence to conclude we only have to prove that the diagram
\[
\xymatrix@R=0.6cm@C=1.5cm{
\Kb {\EuScript Tilt} \ar[r]^-{\sim} \ar[d]_-{\Kb(\Lder i^*)} & D^{\GB \times \Gm}(\tgg) \ar[d]^-{\Lder i^*} \\
\Kb \Tilt(\ES^{\GB \times \Gm}(\tNC)) \ar[r]^-{\sim} & D^{\GB \times \Gm}(\tgg)
}
\]
commutes, where the horizontal arrows are the functors considered in~\S\ref{ss:proof-equivalence-tgg} and~\S\ref{ss:proof-equivalence-tNC} respectively. The latter fact follows from Remark~\ref{rk:real-tNC} and the fact that the objects in ${\EuScript Tilt}$, considered as equivariant coherent sheaves, are acyclic for the functor $i^*$ (see Lemma~\ref{lem:nakayama}\eqref{it:nakayama-complex}).
\end{proof}

\subsection{Proof of Corollary~\ref{cor:parity-perverse}}
\label{ss:proof-parity-perverse}

We begin with an easy lemma. Let $\Bbbk$ be a field, and $X=\bigsqcup_{s \in \mathscr{S}} X_s$ be an algebraic variety endowed with an algebraic stratification, where $X_s$ is simply connected for any $s \in \mathscr{S}$. We denote by $\DM_X$ the Grothendieck--Verdier duality functor.

\begin{lem}
\label{lem:perverse-criterion}
Let $\FC$ be an object in $\Db_{\mathscr{S}}(X,\Bbbk)$ which satisfies $\DM_X(\FC) \cong \FC$ and $\Hom(\FC, \FC[n])=0$ for any $n \in \ZM_{<-1}$. Then $\FC$ is perverse.
\end{lem}

\begin{proof}
Assume that $\FC$ is not perverse, and
let $N=\max\{n \in \ZM_{> 0} \mid {}^p \hspace{-1pt} \HC^n(\FC) \neq 0\}$. Then since $\DM_X(\FC) \cong \FC$, $N$ is also the largest integer such that ${}^p \hspace{-1pt} \HC^{-N}(\FC) \neq 0$, and we have $\HC^{-N}(\FC) \cong \DM_X(\HC^{N}(\FC))$. In particular, we deduce that $\mathrm{top}(\HC^{N}(\FC)) \cong \DM_X(\mathrm{soc}(\HC^{-N}(\FC)))$. Under our assumptions each simple $\mathscr{S}$-constructible perverse sheaf on $X$ is stable under $\DM_X$, hence we deduce that there exists an isomorphism $\mathrm{top}(\HC^{N}(\FC)) \cong \mathrm{soc}(\HC^{-N}(\FC))$; in particular there exists a non-zero morphism $\phi \colon \HC^{N}(\FC) \to \HC^{-N}(\FC)$. Now consider the following morphism (where the first and third morphisms come from the appropriate perverse truncation triangles):
\[
\psi \colon \FC \to \HC^{N}(\FC)[-N] \xrightarrow{\phi[-N]} \HC^{-N}(\FC)[-N] \to \FC[-2N].
\]
Then ${}^p \hspace{-1pt} \HC^N(\psi) \neq 0$, hence $\psi$ is a non zero element in $\Hom(\FC, \FC[-2N])$, contradicting our assumption.
\end{proof}

\begin{proof}[Proof of Corollary~{\rm \ref{cor:parity-perverse}}]
By standard reductions (see e.g.~\cite[Lemma~3.6]{jmw2}) one can assume that $\GB$ is a product of simply connected quasi-simple groups not of type $A$ and general linear groups. Then $\GB$ and $\FM$ satisfy the assumptions of Theorem~\ref{thm:main}.

Using Lemma~\ref{lem:perverse-criterion}, to prove our claim it suffices to prove that for $\lambda \in -\XB^+$ and $n<-1$ we have
\[
\Hom_{\Parity_{(\ID)}(\Gr, \FM)}(\EC_\lambda, \EC_\lambda[n])=0.
\]
However we have $\Theta(\EC_\lambda) \cong \TC^{-\lambda}$ by Theorem~\ref{thm:main}\eqref{it:main-thm-indec}, hence $\Theta(\EC_\lambda) \cong \mathsf{T}(-\lambda) \otimes \OC_{\tNC}$ by~\cite[Corollary~4.8]{mr}. We deduce that
\[
\Hom_{\Parity_{(\ID)}(\Gr, \FM)}(\EC_\lambda, \EC_\lambda[n]) \cong \Hom_{D^{\GB \times \Gm}(\tNC)}(\mathsf{T}(-\lambda) \otimes \OC_{\tNC}, \mathsf{T}(-\lambda) \otimes \OC_{\tNC} \langle -n \rangle).
\]
Now by~\cite[Proposition~A.6]{mr} the right-hand side is isomorphic to
\begin{multline*}
\mathsf{H}^0 \bigl( \Inv^{\GB} \circ \Inv^{\Gm}(\mathsf{T}(-\lambda)^* \otimes \mathsf{T}(-\lambda) \otimes R\Gamma(\tNC, \OC_{\tNC}) \langle -n \rangle) \bigr) \\
\cong \mathsf{H}^0 \Bigl( \Inv^{\GB} \bigl(\mathsf{T}(-\lambda)^* \otimes \mathsf{T}(-\lambda) \otimes \Inv^{\Gm} ( R\Gamma(\tNC, \OC_{\tNC}) \langle -n \rangle) \bigr) \Bigr).
\end{multline*}
Now, using the same arguments as in~\cite[Lemma~1.4.2]{br} one can check that
\[
\Inv^{\Gm}(R\Gamma(\tNC, \OC_{\tNC}) \langle -n \rangle)=0
\]
unless $n \in 2\ZM_{\geq 0}$. The claim follows.
\end{proof}

\subsection{Proof of Proposition~\ref{prop:equiv-parity-tilt}}
\label{ss:proof-Satake}


By~\cite[Corollary~4.8]{mr}, for any $\lambda \in \XB^+$ we have $\TC^\lambda \cong \mathsf{T}(\lambda) \otimes \OC_{\tNC}$, hence $\Theta(\EC_{-\lambda}) \cong \mathsf{T}(\lambda) \otimes \OC_{\tNC}$.
Now we observe that, by the same arguments as in the proof of Corollary~\ref{cor:parity-perverse} and since $\Inv^{\Gm} \bigl( R\Gamma(\tNC, \OC_{\tNC}) \bigr) = \FM$, the functor
\[
\Tilt(\GB) \to D^{\GB \times \Gm}(\tNC) \colon V \mapsto V \otimes \OC_{\tNC}
\]
is fully faithful. 
Hence $\Theta$ induces an equivalence
\begin{equation}
\label{eqn:Theta'-perv}
\PParity_{(\GD(\mathscr{O}))}(\Gr) \simto \Tilt(\GB)
\end{equation}
sending $\EC^\lambda = \EC_{w_0 \lambda}$ to $\mathsf{T}(-w_0 \lambda)$ for all $\lambda \in \XB^+$.

Now by~\cite[Proposition~2.1]{mv} the categories of $\GD(\mathscr{O})$-constructible and $\GD(\mathscr{O})$-equivariant perverse sheaves on $\Gr$ are canonically equivalent. Using this property and the antiautomorphism of $\GD(\mathscr{K})$ defined by $g \mapsto g^{-1}$, one can construct an autoequivalence of $\PParity_{(\GD(\mathscr{O}))}(\Gr)$ sending $\EC^\lambda$ to $\EC^{-w_0\lambda}$.
Composing this equivalence with~\eqref{eqn:Theta'-perv} provides the equivalence $\overline{\mathsf{S}}_\FM$.

Formula~\eqref{eqn:formula-stalks-parity} follows from the following chain of equalities for $\lambda,\mu \in \XB^+$:
\begin{align*}
\sum_{k \in \ZM} \dim \bigl( \HM^{k-\dim(\Gr^{\mu})}(\imath_\mu^* \EC^\lambda) \bigr) \cdot \vv^k &= \sum_{k \in \ZM} \dim \bigl( \HM^{k-\dim(\Gr_{w_0 \mu})}(\Gr_{w_0 \mu}, i_{w_0 \mu}^* \EC_{w_0\lambda}) \bigr) \cdot \vv^k \\
&= \sum_{k \in \ZM} (\TC^{-w_0\lambda} : \Delta_{\tNC}^{-w_0 \mu} \langle k \rangle) \cdot \vv^k \\
&= \sum_{\nu \in \XB^+} (\mathsf{T}(-w_0\lambda) : \mathsf{M}(\nu)) \cdot \MC_\nu^{-w_0 \mu}(\vv^{-2})
\\
& = \sum_{\nu \in \XB^+} \bigl( \mathsf{T}(\lambda) : \mathsf{N}(-w_0\nu) \bigr) \cdot \MC^{-w_0\mu}_\nu(\vv^{-2}).
\end{align*}
Here the second equality follows from
Theorem~\ref{thm:main}\eqref{it:main-thm-ch}, the third one follows from the formula for $\ch_\Delta(\mathsf{T}(-w_0 \lambda) \otimes \OC_{\tNC})$ provided by~\cite[Proposition~4.6]{mr}, and the last one from the fact that $(\mathsf{T}(-w_0 \lambda) : \mathsf{N}(\nu)) = (\mathsf{T}(\lambda) : \mathsf{M}(-w_0 \nu))$ (since $\mathsf{T}(\lambda)^* \cong \mathsf{T}(-w_0 \lambda)$).

\subsection{Proof of Proposition~\ref{prop:cohomology}}
\label{ss:proof-cohomology}

Before giving the proof of the proposition, we start with a lemma describing the objects $\tTC^\lambda$ introduced in~\S\ref{ss:proof-equivalence-tgg} in the case $\lambda \in \XB^+$.

\begin{lem}
\label{lem:ttilting-dominant}
For $\lambda \in \XB^+$, we have $\tTC^\lambda \cong T(\lambda) \otimes \OC_{\tgg}$ in $D^{\GB \times \Gm}(\tgg)$.
\end{lem}

\begin{proof}
Recall that, by~\cite[Corollary~4.8]{mr}, we have $\TC^\lambda \cong \mathsf{T}(\lambda) \otimes \OC_{\tNC}$. 
Using this fact,
the isomorphisms
\[
\Lder i^*(\tTC^\lambda) \cong \TC^\lambda, \qquad \Lder i^*(\mathsf{T}(\lambda) \otimes \OC_{\tgg}) \cong \mathsf{T}(\lambda) \otimes \OC_{\tNC}
\] 
(see Lemma~\ref{lem:tTC-i*}),
and the same arguments as in the proof of Proposition~\ref{prop:morphism-D-N-tgg},
one can check that the graded $\OC(\tg^*)$-modules $\Hom_{D^{\GB}(\tgg)}(\tTC^\lambda, \mathsf{T}(\lambda) \otimes \OC_{\tgg})$, $\Hom_{D^{\GB}(\tgg)}(\mathsf{T}(\lambda) \otimes \OC_{\tgg}, \tTC^\lambda)$, $\End_{D^{\GB}(\tgg)}(\tTC^\lambda)$ and $\End_{D^{\GB}(\tgg)}(\mathsf{T}(\lambda) \otimes \OC_{\tgg})$ are free, and that the morphisms
\begin{align*}
\FM \otimes_{\OC(\tg^*)} \Hom_{D^{\GB}(\tgg)}(\tTC^\lambda, \mathsf{T}(\lambda) \otimes \OC_{\tgg}) & \to \Hom_{D^{\GB}(\tNC)}(\TC^\lambda, \mathsf{T}(\lambda) \otimes \OC_{\tNC}), \\
\FM \otimes_{\OC(\tg^*)} \Hom_{D^{\GB}(\tgg)}(\mathsf{T}(\lambda) \otimes \OC_{\tgg}, \tTC^\lambda) & \to \Hom_{D^{\GB}(\tNC)}(\mathsf{T}(\lambda) \otimes \OC_{\tNC}, \TC^\lambda), \\
\FM \otimes_{\OC(\tg^*)} \End_{D^{\GB}(\tgg)}(\tTC^\lambda) & \to \End_{D^{\GB}(\tNC)}(\TC^\lambda), \\
\FM \otimes_{\OC(\tg^*)} \End_{D^{\GB}(\tgg)}(\mathsf{T}(\lambda) \otimes \OC_{\tgg}) & \to \End_{D^{\GB}(\tNC)}(\mathsf{T}(\lambda) \otimes \OC_{\tNC})
\end{align*}
induced by $\Lder i^*$
are isomorphisms. Since $\End_{D^{\GB}(\tNC)}(\mathsf{T}(\lambda) \otimes \OC_{\tNC}) \cong \End_{D^{\GB}(\tNC)}(\TC^\lambda)$ is concentrated in non-negative degrees (see the proof of Corollary~\ref{cor:parity-perverse}), we deduce that the morphisms
\begin{align*}
\End_{D^{\GB \times \Gm}(\tgg)}(\tTC^\lambda) &\to \End_{D^{\GB \times \Gm}(\tNC)}(\TC^\lambda), \\  \End_{D^{\GB \times \Gm}(\tgg)}(\mathsf{T}(\lambda) \otimes \OC_{\tgg}) &\to \End_{D^{\GB \times \Gm}(\tNC)}(\mathsf{T}(\lambda) \otimes \OC_{\tNC})
\end{align*}
induced by $\Lder i^*$ are isomorphisms.

Choose a pair of inverse isomorphisms $f \colon \TC^\lambda \simto \mathsf{T}(\lambda) \otimes \OC_{\tNC}$ and $g \colon \mathsf{T}(\lambda) \otimes \OC_{\tNC} \simto \TC^\lambda$ in $D^{\GB \times \Gm}(\tNC)$. Then by the preceding observations there exist morphisms $f' \colon \tTC^\lambda \to \mathsf{T}(\lambda) \otimes \OC_{\tgg}$ and $g' \colon \mathsf{T}(\lambda) \otimes \OC_{\tgg} \to \tTC^\lambda$ in $D^{\GB \times \Gm}(\tgg)$ such that $\Lder i^*(f')=f$, $\Lder i^*(g')=g$. And these observations also show that the fact that $g \circ f=\mathrm{id}$, resp.~$f \circ g = \mathrm{id}$, implies that $g' \circ f'=\mathrm{id}$, resp.~$f' \circ g' = \mathrm{id}$.
\end{proof}

\begin{proof}[Proof of Proposition~{\rm \ref{prop:cohomology}}]
First we prove~\eqref{it:cohom}.
It follows from the construction of our equivalence~\eqref{eqn:equiv-Parity-Tilt-equiv} that the following diagram commutes:
\[
\xymatrix@R=0.5cm{
\Parity_{\ID}(\Gr, \FM) \ar[rr]_-{\sim}^-{\eqref{eqn:equiv-Parity-Tilt-equiv}} \ar[rd]_-{\HM^\bullet_{\ID}(\Gr,-)} & & {\EuScript Tilt} \ar[ld]^-{\kappa} \\
& \Mod^\gr(\OC(\tg^*)). &
}
\]
(In both downward arrows, we omit the forgetful functor from graded $C$-modules to graded $\OC(\tg^*)$-modules.) It is clear that $\kappa(\mathsf{T}(\lambda) \otimes \OC_{\tgg}) \cong \mathsf{T}(\lambda) \otimes \OC(\tg^*)$, with the grading indicated in the statement of the proposition. On the other hand, we have
\begin{multline*}
\HM^\bullet_{\ID}(\Gr, \EC^\lambda) \cong \HM^\bullet_{\ID}(\mathrm{pt}; \FM) \otimes_{\HM^\bullet_{\GD(\mathscr{O})}(\mathrm{pt}; \FM)} \HM^\bullet_{\GD(\mathscr{O})}(\Gr, \EC^\lambda) \\
\cong \HM^\bullet_{\ID}(\mathrm{pt}; \FM) \otimes_{\HM^\bullet_{\GD(\mathscr{O})}(\mathrm{pt}; \FM)} \HM^\bullet_{\GD(\mathscr{O})}(\Gr, \EC^{-w_0 \lambda}) \cong \HM^\bullet_{\ID}(\Gr, \EC_{-\lambda}).
\end{multline*}
(Here we use the autoequivalence considered in~\S\ref{ss:proof-Satake} and the fact that the two natural morphisms $\HM^\bullet_{\GD(\mathscr{O})}(\mathrm{pt}; \FM) \to \HM^\bullet_{\GD(\mathscr{O})}(\Gr; \FM)$ coincide, see the proof of
Lemma~\ref{lem:morphism-cohomology-factorization}.)
These observations prove the second isomorphism in~\eqref{it:cohom}. The first one follows, using Lemma~\ref{lem:properties-parity}\eqref{it:parity-cohomology}.

Now we prove the first isomorphism in~\eqref{it:cohom-stalks}. By adjunction, for $m \in \ZM$ we have
\begin{multline*}
\Hom_{\Dmix_{(\ID)}(\Gr,\FM)}^{m}(\Delta^{\mix}_{-\mu}, \EC^{\mix}_{-\lambda} \langle - m \rangle) \\
\cong \Hom_{\Dmix_{(\ID)}(\Gr_{-\mu},\FM)}(\underline{\FM}_{\Gr_{-\mu}}\{\dim (\Gr_{-\mu})\}, (i_{-\mu})^!  \EC^{\mix}_{-\lambda} \{ m \}).
\end{multline*}
(Here, as in~\cite{modrap2}, $\{1\} = \langle - 1 \rangle [1]$ is the autoequivalence of the triangulated category $\Dmix_{(\ID)}(\Gr_{-\mu},\FM)=\Kb \Parity_{(\ID)}(\Gr_{-\mu}, \FM)$ induced by the cohomological shift in $\Parity_{(\ID)}(\Gr_{-\mu},\FM)$.)
By~\cite[Remark~2.7]{modrap2}, $(i_{-\mu})^!  \EC^{\mix}_{-\lambda}$ is the complex whose $0$-th term is the parity complex $(i_{-\mu})^!  \EC_{-\lambda}$, and whose other terms vanish. We deduce that
\begin{multline*}
\Hom_{\Dmix_{(\ID)}(\Gr_{-\mu},\FM)}(\underline{\FM}_{\Gr_{-\mu}}\{\dim(\Gr_{-\mu})\}, (i_{-\mu})^!  \EC^{ \mix}_{-\lambda} \{ m \}) \\
\cong \HM^{m-\dim(\Gr_{-\mu})}(\Gr_{-\mu}, (i_{-\mu})^!  \EC_{-\lambda}).
\end{multline*}
Finally, using the same considerations as in~\S\ref{ss:proof-Satake} (or as in the proof of~\eqref{it:cohom}), we deduce a canonical isomorphism
\[
\Hom_{\Dmix_{(\ID)}(\Gr,\FM)}^{m}(\Delta^{\mix}_{-\mu}, \EC^{\mix}_{-\lambda} \langle - m \rangle) \cong \HM^{m-\dim(\Gr^\mu)}(\imath_\mu^! \EC^\lambda).
\]

On the other hand, using the equivalence $\Phi$
we obtain that
\[
\Hom_{\Dmix_{(\ID)}(\Gr,\FM)}^{m}(\Delta^{\mix}_{-\mu}, \EC^{\mix}_{-\lambda} \langle - m \rangle) \cong \Hom_{D^{\GB \times \Gm}(\tNC)}(\Delta^\mu_{\tNC}, \TC^\lambda \langle -m \rangle).
\]
Using~\cite[Equation~(4.10)]{mr} and the fact that $\TC^\lambda \cong \mathsf{T}(\lambda) \otimes \OC_{\tNC}$, we deduce that
\[
\Hom_{\Dmix_{(\ID)}(\Gr,\FM)}^{m}(\Delta^{\mix}_{-\mu}, \EC^{\mix}_{-\lambda} \langle - m \rangle) \cong \bigl( \mathsf{T}(\lambda) \otimes \Gamma(\tNC, \OC_{\tNC}(-w_0 \mu))_m \bigr)^\GB,
\]
where the subscript ``$m$'' denotes the $m$-th graded part.
This finishes the proof.

The proof of the second isomorphism in~\eqref{it:cohom-stalks} is similar, using Lemma~\ref{lem:ttilting-dominant} and replacing $\Dmix_{(\ID)}(\Gr,\FM)$ by $\Dmix_{\ID}(\Gr,\FM)$, ordinary cohomology by equivariant cohomology, $\Phi$ by $\Psi$, and $\tNC$ by $\tgg$.
\end{proof}

\end{document}